\documentclass[11pt,a4paper]{article}
\usepackage[left=2.45cm, top=2.45cm,bottom=2.45cm,right=2.45cm]{geometry}

\usepackage{times}
\usepackage[utf8]{inputenc}
\usepackage[T1]{fontenc}
\usepackage[ngerman, english]{babel}
\usepackage{mathrsfs}
\usepackage{tikz}
\usepackage{pgfplots}

\usepackage{graphicx}
\usepackage{latexsym}
\usepackage{dsfont}
\usepackage{amssymb,amsmath,amsfonts,amsthm}
\usepackage{array}

\usepackage{tikz}
\usetikzlibrary{matrix,arrows,patterns}

\newtheorem{Satz}{Satz}

\newtheorem{Lemma}[Satz]{Lemma}	

\newtheorem{Corollary}[Satz]{Corollary}
\newtheorem{theorem}[Satz]{Theorem}

\newtheorem*{Lemma*}{Lemma}

\theoremstyle{definition}

\newtheorem{remark}{Remark}

\newcommand{\E}{\mathbb{E}}
\newcommand{\R}{\mathbb{R}}

\newcommand{\N}{\mathbb{N}}

\renewcommand{\S}{\mathbb S}

\newcommand{\cK}{\mathcal{K}}

\newcommand{\cH}{\mathcal{H}}

\newcommand{\Cov}{\mathbb{C}\text{\rm ov}}
\newcommand{\Var}{\mathbb{V}\text{\rm ar}}

\newcommand{\1}{\mathds{1}}

\newcommand{\oldU}{i}

\DeclareMathOperator{\cum}{cum}

\DeclareMathOperator{\arcosh}{{\rm arcosh}}

\newcommand{\Hd}{\mathbb H^{d}}

\newcommand{\defeq}{\mathrel{\mathop:}=}

\renewcommand{\to}{\rightarrow}

\newcounter{variablenzaehler}

\begin{document}
\title{Does a central limit theorem hold for the $k$-skeleton\\ of Poisson hyperplanes in hyperbolic space?}

\author{Felix Herold\footnote{Karlsruhe Institute of Technology, Karlsruhe, Germany, felix.herold@kit.edu}, Daniel Hug\footnote{Karlsruhe Institute of Technology, Karlsruhe, Germany, daniel.hug@kit.edu} and Christoph Thäle\footnote{Fakult\"at f\"ur Mathematik, Ruhr-Universit\"at Bochum, Bochum, Germany, christoph.thaele@rub.de}}
\date{\today}

\maketitle

\begin{abstract}
\noindent Poisson processes in the space of $(d-1)$-dimensional totally geodesic subspaces (hyperplanes) in a $d$-dimensional hyperbolic space of constant curvature $-1$ are studied. The $k$-dimensional Hausdorff measure of their $k$-skeleton is considered. Explicit formulas for first- and second-order quantities restricted to bounded observation windows are obtained. The central limit problem for the $k$-dimensional Hausdorff measure of the $k$-skeleton is approached in two different set-ups: (i) for a fixed window and growing intensities, and (ii) for fixed intensity and growing spherical windows. While in case (i) the central limit theorem is valid for all $d\geq 2$, it is shown that in case (ii) the central limit theorem holds for $d\in\{2,3\}$ and fails if $d\geq 4$ and $k=d-1$ or if $d\geq 7$ and for general $k$. Also rates of convergence are studied and multivariate central limit theorems are obtained. Moreover, the situation in which the intensity and the spherical window are growing simultaneously is discussed. In the background are the Malliavin-Stein method for normal approximation and the combinatorial moment structure of Poisson U-statistics as well as tools from hyperbolic integral geometry.
\\[0.5em]
{\bf Keywords}. Central limit theorem, Crofton formula, Malliavin-Stein method, hyperbolic integral geometry, hyperbolic stochastic geometry, $K$-function, pair correlation function, Poisson hyperplane process, random measures, U-statistics.\\
{\bf MSC}. Primary: 60D05, 53C65, 52A22,   Secondary: 52A55, 60F05.
\end{abstract}

\tableofcontents

\section{Introduction}

Random tessellations in $\R^d$ form a class of mathematical objects that have been under intensive investigation in stochastic geometry during the last decades. In addition to intrinsic mathematical curiosity, a major reason for continuing interest in random tessellations is  that they provide highly relevant models for practical applications, for example, in telecommunication or materials science  \cite{Gloaguen,OhserMuecklich,Okabe,Redenbach09}. One of the principal random tessellation models in Euclidean space is induced by a
Poisson process of hyperplanes. In $\R^d$ with $d\geq 2$ and in the stationary and isotropic case, the construction of a Poisson hyperplane tessellation can be described as follows. Fix a parameter $t>0$ and consider a stationary Poisson point process on the real line with intensity $t$. To each point $p_i$ of the Poisson process we attach independently of each other and independently of the underlying Poisson process a random vector $u_i$ which is uniformly distributed on the unit sphere $\S^{d-1}$ of $\R^d$. Then to each pair $(p_i,u_i)\in\R\times\S^{d-1}$ we associate the hyperplane $H_i:=\{x\in\R^d:\langle x,u_i\rangle=p_i\}$ and call the random collection of all such hyperplanes a (stationary and isotropic) Poisson hyperplane process in $\R^d$ with intensity $t$. The random hyperplanes $H_i$ almost surely divide the space $\R^d$ into countably many random convex polytopes. The collection of all these polytopes is a (stationary and isotropic) Poisson hyperplane tessellation in $\R^d$ with intensity $t$. We remark that the intensity parameter $t$, roughly speaking, controls the expected surface content of the Poisson hyperplane tessellation per unit volume. More precisely, $t=\E\cH^{d-1}(Z\cap[0,1]^d)$, where $Z=\bigcup_{i=1}^\infty H_i$ is the random union set induced by the Poisson hyperplane process and $\cH^{d-1}$ stands for the $(d-1)$-dimensional Hausdorff measure.

For Poisson hyperplane tessellations many first- and second-order quantities are explicitly available for a broad class of functionals and also a comprehensive central limit theory has been developed over the last 15 years, cf. \cite{HeinrichCLTforPHT,HeinrichSchmidtSchmidtCLT,LPST,ReitznerSchulteCLT} and \cite[Chapter 10]{SW} as well as the many references cited therein. In the literature, central limit theorems for functionals of Poisson hyperplanes have been considered in two different set-ups. In a first setting the tessellation is restricted to a fixed (usually convex) observation window and the asymptotic behaviour is explored when the intensity $t$ of the underlying Poisson process is increased. Alternatively, the intensity is kept fixed, while the size of the observation window is increased. By a simple scaling relation both set-ups are equivalent when homogeneous functionals (such as intrinsic volumes, positive powers of intrinsic volumes or integrals with respect to support measures) of the tessellation are considered, see \cite[Corollary 6.2]{LPST}.

While the spherical analogues of Poisson hyperplane tessellations, namely Poisson great hypersphere tessellations, were investigated, for example, in \cite{ArbeiterZahle,HuThspstitt,HugReichenbacher,HugSchneiderConicalTessellations,MilesSphere}, only few results seem to be available for such tessellations in standard spaces of  constant negative curvature, see \cite{BenjaminiEtAl,PorretBlanc,SantaloYanetz,TykessonCalka}. The spherical space of constant positive curvature especially features  by its compactness, which in turn implies that Poisson great hypersphere tessellations almost surely consist of only finitely many spherical random polytopes, In contrast, Poisson hyperplane tessellations in a standard space of constant negative curvature display a number of striking new phenomena that cannot be observed in their Euclidean or spherical counterparts. It is the purpose of the present paper to initiate a systematic study of intersection processes of Poisson hyperplane tessellations in the $d$-dimensional hyperbolic space $\Hd$ and to uncover some of the anticipated and remarkable new phenomena. We confine ourselves to the study of the total volume (in the appropriate dimension) of the intersection processes induced by Poisson hyperplanes in a (hyperbolic convex) test set. We explicitly identify the expectation and the covariance structure of these functionals by making recourse to general formulas for and structural properties of Poisson U-statistics and to Crofton-type formulas from hyperbolic integral geometry. In addition and more importantly, we study probabilistic limit theorems for these functionals in the two asymptotic regimes described above for the Euclidean set-up. While the central limit theorems for growing intensity and fixed observation window are a direct consequence of general central limit theorems for Poisson U-statistics \cite{LPST,ReitznerSchulteCLT,SchulteDissertation,SchulteKolmogorov}, it will turn out that the limit theory in the other regime, that is, when the intensity is kept fixed and the size of the observation window is increased, is fundamentally different. We will prove that here a central limit theorem in fact holds in space dimensions $d=2$ and $d=3$. On the other hand, we will show that a central limit theorem fails for all space dimensions $d\geq 4$ if the total $(d-1)$-volume of the union of all hyperplanes is considered. For the total volume of intersection processes of arbitrary order this will be proved for technical reasons only for dimensions $d\geq 7$. We emphasize that this remarkable and surprising new feature is a consequence of the negative curvature of the underlying space and has no counterpart in the Euclidean or spherical set-up. Another interesting and unexpected feature is observed in this regime for the asymptotic covariance matrix of the vector of  $k$-volumes of the $k$-skeletons, $k=0,\ldots,d-1$. This matrix turns out to have full rank for $d=2$, but it has rank one in dimension $d\geq 3$. In addition, we will study the situation in which the intensity and the size of the observation window are increased \textit{simultaneously}. In this case it will turn out that in all situations where the central limit theorem fails for fixed intensity, the Gaussian fluctuations are in fact preserved as soon as the intensity tends to infinity, independently of the behaviour of the size of the observation window (as long as it is bounded from below).

As anticipated above, the proofs of our results concerning first- and second-order properties of the total volume of intersection processes rely on general formulas for U-statistics of Poisson point processes as presented in \cite{LP} and on tools from hyperbolic integral geometry as developed in \cite{Brothers,GallegoSolanes,Santalo,Solanes}. The central limit theorems we consider will be of quantitative nature, that is, we will provide explicit bounds on the quality (speed) of normal approximation measured in terms of both the Wasserstein and the Kolmogorov distance. Their proofs are based on general normal approximation bounds that have been derived in \cite{EichelsbacherThaele14,ReitznerSchulteCLT,SchulteKolmogorov} using the Malliavin-Stein technique on Poisson spaces (see collection \cite{ReitznerPeccati} for a representative overview concerning this method). This directly implies the central limit theorem for fixed windows and growing intensities. On the other hand, for fixed intensity and when the window is a hyperbolic ball $B_r$ of radius $r$ around a fixed point in $\Hd$, crucial building blocks of these bounds are Crofton-type integrals of the form $$
\int_{A_h(d,k)}\cH^k(H\cap B_r)^l\,\mu_k(dH),
$$
 where $A_h(d,k)$ denotes the space of $k$-dimensional totally geodesic subspaces of $\Hd$ and $\mu_k$ is the suitably normalized invariant measure on $A_h(d,k)$ (all terms will be explained in detail below). While in the Euclidean case the asymptotic behaviour of such integrals, as $r\to\infty$, is quite straightforward, this is not the case in the hyperbolic set-up. In contrast to the Euclidean case, it will turn out that their behaviour crucially depends on whether $l(k-1)$ is less than, greater than or equal to $d-1$ (see Lemma \ref{lem:lines_intersecting_ball_inequality}). In essence, the latter is an effect of the negative curvature, which in turn causes an exponential growth of volume of linearly expanding balls in $\Hd$. To show that a central limit theorem fails in higher space dimensions is arguably the most technical part of this paper. We do this by showing that the fourth cumulant of the centred and normalized total volume of the intersection processes does not converge to $0$, which in turn is the fourth cumulant of a standard Gaussian distribution. However, to bring this in contradiction with a central limit theorem we need to argue that the fourth power of the total volume is uniformly integrable, which in turn will be established by consideration of their fifths moments. This requires a fine analysis of combinatorial moment formulas for U-statistics of Poisson processes. In essence and in contrast to the lower dimensional cases $d=2$ and $d=3$, the failure of the central limit theorem for space dimensions $d \geq 4$ is due to the fact that in these dimensions the contribution of single hyperplanes is asymptotically not negligible anymore.

We emphasize that the present paper contributes to a recent and active line of current mathematical research in stochastic geometry on models in non-Euclidean spaces. As concrete examples we mention here the studies about spherical convex hulls and convex hulls on half-spheres in \cite{BaranyHugReitznerSchneider,KabluchkoMarynychTemesvariThaele,MaeharaMartini18}. Central limit theorems for the volume of random convex hulls in spherical space, hyperbolic spaces and Minkowski geometries were obtained in \cite{BesauThaele},  asymptotic normality of very general so-called stabilizing functionals of Poisson point processes on manifolds was considered in \cite{PenroseYukichMf}. Again more specifically, the papers \cite{BodeFoun...,FountulakisYukich,MullerStaps,TakashiYogesh} study various aspects of random geometric graphs in hyperbolic spaces, including central limit theorems for a number of parameters. Random tessellations of the unit sphere by great hyperspheres are the content of \cite{ArbeiterZahle,HugReichenbacher,HugSchneiderConicalTessellations,MilesSphere}, while so-called random splitting tessellations in spherical spaces were introduced and investigated in \cite{DeussHoerrmannThaele,HuThspstitt}. The paper \cite{CalkaChapronEnriquez} is concerned with properties of Poisson-Voronoi tessellations on general Riemannian manifolds. Finally, the geometry of random fields on the sphere is studied in the monograph \cite{MarinucciPeccati} and invariant random fields on
spaces with a group action are described in \cite{Malyarenko2013}. In a similar vein, it is pointed out in \cite{Liao} that a
systematic study of the invariance properties of probability distributions under a general group action is missing. The  book \cite{Liao}  therefore explores Markov processes whose distributions are
invariant under the action of a Lie group.

\medspace

The remaining parts of this paper are structured as follows. In the next section we formally define Poisson hyperplane tessellations in $\Hd$ and present our main results. We start in Section \ref{subsec:FirstOrder} with expectations and continue in Section \ref{subsec:SecondOrder} with second-order characteristics associated with the total volume of intersections processes. Our limit theorems will be discussed in Section \ref{subsec:LimitTheorems}. The necessary background material on hyperbolic geometry and hyperbolic integral geometry is collected in Section \ref{sec:3.1},  the background material on Poisson U-statistics is the content of Sections \ref{subsec:NormalApproxUstatistics} and \ref{subsec:3.3}. All remaining sections are devoted to the proofs of our results. In Section \ref{sec:4} we present the proofs for first- and second-order parameters and also carry out a detailed covariance analysis, which is needed for our multivariate central limit theory. Our results on generalizations of the K-function and the pair-correlation function are established in Section \ref{sec:ProofKg}. All univariate limit theorems are proved in Section \ref{sec:ProofUni}, while the arguments for the multivariate central limit theorems are provided in the final Section \ref{sec:Multi}.

\section{Main results}\label{sec:MainResults}

\subsection{First-order quantities}\label{subsec:FirstOrder}

We denote by $\Hd$, for $d\ge 2$, the $d$-dimensional hyperbolic space of constant curvature $-1$, which is supplied with the hyperbolic metric $d_h(\,\cdot\,,\,\cdot\,)$. We refer to Section \ref{sec:3.1} below for further background material on hyperbolic geometry and for a description of the conformal ball model for $\Hd$.  Let $p\in\Hd$ be an arbitrary (fixed) point, also referred to as the origin. For $r\ge 0$ we denote by $B_r=\{x\in\Hd:d_h(x,p)\leq r\}$ the hyperbolic ball around $p$ with radius $r$. A set $K\subset\Hd$ is called a hyperbolic convex body, provided that $K$ is non-empty, compact and if with each pair of points $x,y\in K$ the (unique) geodesic connecting $x$ and $y$ is contained in $K$. The space of hyperbolic convex bodies is denoted by $\cK_h^d$. Recall that for $k\in\{0,1,\ldots,d-1\}$ a $k$-dimensional totally geodesic subspace of $\Hd$ is called a $k$-plane and especially $(d-1)$-planes are called hyperplanes. The space of $k$-planes in $\Hd$ is denoted by $A_h(d,k)$. The space $A_h(d,k)$ carries a measure $\mu_k$, which is invariant under isometries of $\Hd$ (see Section \ref{sec:3.1} for the present normalization of this measure). For $s\geq 0$ we denote by $\cH^s$ the $s$-dimensional Hausdorff measure with respect to the intrinsic metric of $\Hd$ as a Riemannian manifold. Finally, we write $\omega_k={2\pi^{k/2}/\Gamma({k/ 2})}$, $k\in\N$, for the surface area of the $k$-dimensional unit ball in the Euclidean space $\R^k$.

\begin{figure}[t]
    \centering
	\includegraphics[width=0.9\columnwidth]{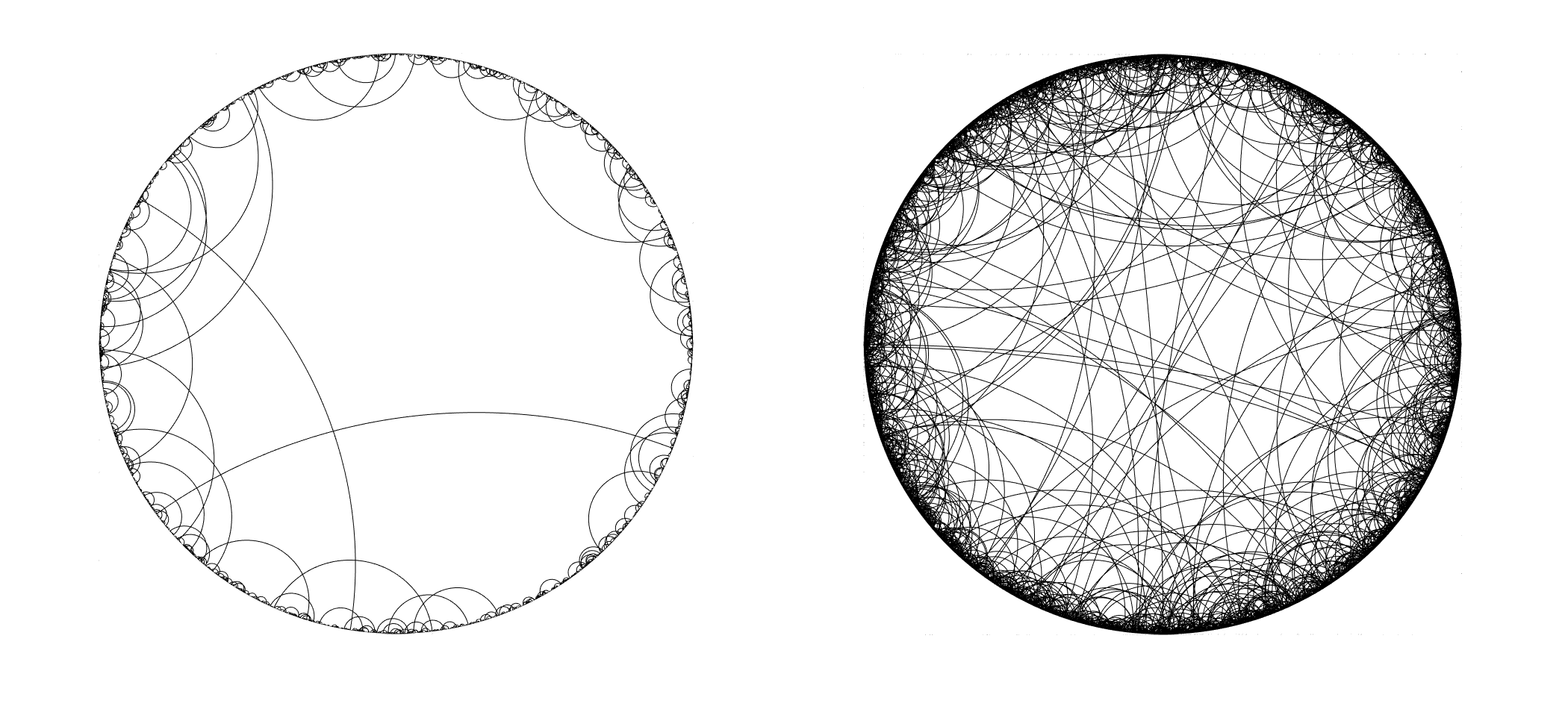}
        \caption{Two realizations of a Poisson hyperplane tessellation in $\mathbb{H}^2$ of different intensities represented in the conformal disc model.}
\label{fig:int_both}
\end{figure}

For $t>0$, let $\eta_t$ be a Poisson process on the space $A_h(d,d-1)$ of hyperplanes in $\Hd$ with intensity measure $t\mu_{d-1}$. We refer to $\eta_t$ as a (hyperbolic) Poisson hyperplane process with intensity $t$. It induces a Poisson hyperplane tessellation in $\Hd$, i.e., a subdivision of $\Hd$ into (possibly unbounded) hyperbolic cells (generalized polyhedra), see Figure \ref{fig:int_both}. For $i\in\{0,\ldots,d-1\}$ we consider the intersection process $\xi_t^{(i)}$ of order $d-i$ of the Poisson hyperplane process $\eta_t$ given by
$$
\xi_t^{(i)} := \frac{1}{(d-i)!}\sum_{(H_1,\ldots,H_{d-i})\in\eta_{t,\neq}^{d-i}}\delta_{H_1\cap\ldots\cap H_{d-i}}\,{\bf 1}\{\dim(H_1\cap\ldots\cap H_{d-i})=i\},
$$
where $\eta_{t,\neq}^{d-i}$ is the set of $(d-i)$-tuples of different hyperplanes supported by $\eta_t$, $\delta_{(\,\cdot\,)}$ denotes the Dirac measure and $\dim(\,\cdot\,)$ stands for the dimension of the set in the argument. In this paper we are interested in random variables of the form
\begin{equation}\label{eq:DefFwti}
\begin{split}
F_{W,t}^{(i)} :&= \int \cH^i(E\cap W)\,\xi^{(i)}_t(dE) \\
& = \frac{1}{(d-i)!}\sum_{(H_1,\ldots,H_{d-i})\in\eta_{t,\neq}^{d-i}}\cH^i(H_1\cap\ldots\cap H_{d-i} \cap W)\,{\bf 1}\{\dim(H_1\cap\ldots\cap H_{d-i})=i\},
\end{split}
\end{equation}
where $W\subset\Hd$ is a (fixed) Borel set in $\Hd$ . In other words, $F_{W,t}^{(i)}$ measures the total $i$-volume (i.e., the $i$-dimensional Hausdorff measure) of the intersection process $\xi_t^{(i)}$ within $W$. For example,
$$
F_{W,t}^{(d-1)} = \sum_{H\in\eta_t}\cH^{d-1}(H\cap W) = \cH^{d-1}\Big(\bigcup_{H\in\eta_t}H\cap W\Big)
$$
is the total surface content of the union of all hyperplanes of $\eta$ within $W$. On the other hand,
$$
F_{W,t}^{(0)} = {1\over d!}\sum_{(H_1,\ldots,H_d)\in\eta_{t,\neq}^d}{\bf 1}\{H_1\cap\ldots\cap H_d\in W,{\rm dim}(H_1\cap\ldots\cap H_d)=0\}
$$
is the total number of vertices in $W$ of the Poisson hyperplane tessellation, i.e., the total number of intersection points induced by the hyperplanes of $\eta_t$. In the Euclidean case these random variables have received particular attention in the literature, see e.g.\ \cite{GH,HeinrichCLTforPHT,HTW,Kallenberg76,Kallenberg80,LPST,Mecke91,ReitznerSchulteCLT,SW} and the references cited therein. As in the Euclidean case, we will start by investigating the expectation of $F_{W,t}^{(i)}$.

\begin{theorem}[Expectation]\label{thm:Expectation}
If $W\subset\Hd$ is a Borel set, $t>0$ and $i\in\{0,1,\ldots,d-1\}$, then
$$
\mathbb{E} F_{W,t}^{(i)}= \frac{\omega_{i+1}}{\omega_{d+1}}\left(\frac{\omega_{d+1}}{\omega_d}\right)^{d-i} \frac{t^{d-i}}{(d-i)!}\, \mathcal{H}^d(W).
$$
\end{theorem}

\begin{remark}\rm
In comparison with the Euclidean and spherical case we observe that precisely the same formula holds in these spaces. This is not surprising, since the proof of Theorem \ref{thm:Expectation} is based only on the multivariate Mecke formula for Poisson processes and a recursive application of Crofton's formula from integral geometry, see Section \ref{sec:4}. Since the latter holds for any standard space of constant curvature $\kappa\in\{-1,0,1\}$ with the same constant (cf.\ \cite{Brothers,Santalo}), independently of the curvature $\kappa$, the result of Theorem \ref{thm:Expectation} holds simultaneously for all standard spaces of constant curvature $\kappa\in\{-1,0,1\}$. In other words this means that the expectation $\mathbb{E} F_{W,t}^{(i)}$ is not an appropriate quantity to `feel' or to `detect' the curvature of the underlying space. For this we will use second-order characteristics.
\end{remark}

\subsection{Second-order quantities}\label{subsec:SecondOrder}

In a next step, we describe the covariance structure of the functionals $F_{W,t}^{(i)}$, $i\in\{0,1,\ldots,d-1\}$, introduced in \eqref{eq:DefFwti}. The following explicit representation
for the covariances will be derived from the Fock space representation of Poisson U-statistics.

\begin{theorem}[Covariances]\label{thm:Variance}
Let $W\subset\Hd$ be a Borel set, let $t>0$, and let $i,j\in\{0,1,\ldots,d-1\}$. Then
$$
\Cov(F_{W,t}^{(i)},F_{W,t}^{(j)}) =\sum_{n=1}^{\min\{d-i,d-j\}} c_{i,j,n,d}\,t^{2d-i-j-n}\int_{A_h(d,d-n)}\cH^{d-n}(E\cap W)^{2}\,\mu_{d-n}(dE)
$$
with
$$
c_{i,j,n,d} = \frac{1}{n!}\, \frac{1}{\omega_{d+1} \, \omega_{d-n+1}}\frac{\omega_{i+1}}{(d-i-n)!}
\frac{\omega_{j+1}}{(d-j-n)!} \left(\frac{\omega_{d+1}}{\omega_d} \right)^{2d-i-j-n}.
$$
\end{theorem}

\begin{remark}\rm
Since Theorem \ref{thm:Variance} follows from the general Fock space representation of Poisson U-statistics, the formula for $\Cov(F_{W,t}^{(i)},F_{W,t}^{(j)})$ is formally the same for all spaces of constant curvature $\kappa\in\{-1,0,1\}$. However, the curvature properties of the underlying space are hidden in the integral-geometric expression
$$
J_k(W):=\int_{A_h(d,k)}\cH^{k}(E\cap W)^{2}\,\mu_{k}(dE),
$$
for $k\in\{0,\ldots,d-1\}$. In fact, if $\kappa\in\{-1,0\}$ and if we replace $W$ by a ball $B_r$ of radius $r$ around an arbitrary fixed point, we can consider the asymptotic behaviour of $J_k(B_r)$, as $r\to\infty$, which is quite different in these two cases (note that in spherical spaces with constant curvature $\kappa=1$ the range of $r$ is bounded). While in the Euclidean case $\kappa=0$, $J_k(B_r)$ behaves like a constant multiple of $r^{d+k}$ for all choices of $k$, in the hyperbolic case $\kappa=-1$ we will show that $J_k(B_r)$ behaves like a constant multiple of $e^{(d-1)r}$ if $2k-1<d$, like a constant multiple of $re^{(d-1)r}$ if $2k-1=d$ and like a constant multiple of $e^{2(k-1)r}$ if $2k-1>d$, see Lemma \ref{lem:lines_intersecting_ball_inequality} below. In this sense we can say that second-order properties of the functionals $F_{W,t}^{(i)}$ are sensitive to the curvature of the underlying space.
\end{remark}

Continuing the discussion of second-order properties of Poisson hyperplane tessellations in $\Hd$, we now introduce and describe the K-function and the pair-correlation function of the $i$-dimensional Hausdorff measure restricted to the $i$-skeleton of the tessellation. In the Euclidean case these two functions have turned out to be essential tools in the second-order analysis of stationary random measures (see the original paper \cite{Ripley} and the recent monograph \cite{BaddeleyTurner} as well as the references cited therein). To be precise, for $i\in\{0,1,\ldots,d-1\}$ and fixed $t>0$, we first consider the $i$-skeleton of the Poisson hyperplane tessellation in $\Hd$ with intensity $t$, which is defined as the random closed set
\begin{align*}
{\rm skel}_i := \bigcup_{(H_1,\ldots,H_{d-i})\in\eta_{t,\neq}^{d-i}\atop {\rm dim}(H_1\cap \ldots\cap H_{d-i})=i}H_1\cap\ldots\cap H_{d-i}.
\end{align*}
The $i$-dimensional Hausdorff measure on ${\rm skel}_i$ is denoted by $\mathbf{M}_i$. It is a stationary random measure on $\Hd$, that is, its distribution is invariant under isometries of $\Hd$.  Its intensity is defined by $\lambda_i=\E F_{B,t}^{(i)}$, where $B\subset\Hd$ is an arbitrary Borel set with $\cH^d(B)=1$.  It follows from Theorem \ref{thm:Expectation} that
\begin{equation}\label{lambdak}
\lambda_i=\frac{\omega_{i+1}}{\omega_{d+1}}\left(\frac{\omega_{d+1}}{\omega_d}\right)^{d-i} \frac{t^{d-i}}{(d-i)!}.
\end{equation}
 The K-function of the random measure $\mathbf{M}_i$ is defined by
\begin{equation}\label{Kfdh}
K_i(r) : = \frac{1}{\lambda_i^2}\,\E\int_{{\Hd}}\int_B\1\{0<d_h(x,y)\leq r\}\,\mathbf{M}_i(dy)\,\mathbf{M}_i(dx),\qquad r>0.
\end{equation}
The condition $d_h(x,y)>0$ is usually omitted in the definition of the K-function of a diffuse stationary random measure. For $i\in\{1,\ldots,d-1\}$, the proof of the following more general
Theorem \ref{thm:KFunctionAndPCF} will show that $K_i(r)$  remains indeed unchanged if we drop the condition $d_h(x,y)>0$. For $i=0$, however,
the random measure $\mathbf{M}_i$ is a stationary point process in $\Hd$ and then the restriction $d_h(x,y)>0$ is common. The proof of
Theorem \ref{thm:KFunctionAndPCF} will also show that the summands corresponding to indices $n\in\{0,\ldots,d-1\}$ in \eqref{ctdh1} are not affected by the restriction, but the summand with $n=d$ will be zero.

If we define $K_i(B,r)$ as in \eqref{Kfdh}, but for a general measurable set $B\subset\Hd$, it follows from the stationarity of $\eta_t$
that the measure $K_i(\,\cdot\,,r)$ is isometry invariant and hence a constant multiple of $\mathcal{H}^d(\, \cdot \, )$, provided it is locally finite. In Theorem \ref{thm:KFunctionAndPCF}, this will be shown and the constant will be determined.
We will also see that $K_i(r)$ is differentiable, which allows us to consider the pair-correlation function
$$
g_i(r) := {1\over \omega_{d}\sinh^{d-1}(r)}\,\frac{dK_i}{ dr}(r),\qquad r> 0.
$$
Roughly speaking it describes the probability of finding a point on the $i$-skeleton at geodesic distance $r$ from another point belonging to ${\rm skel}_i$.

More generally and in analogy to the covariances considered in Theorem \ref{thm:Variance}, we will consider the
mixed K-function $K_{ij}$ for $ i, j \in \{0,\ldots,d-1\}$.
For  $r>0$ and a measurable set $B\subset\Hd$  with $\cH^d(B)=1$ it is defined by
\begin{align*}
K_{ij}(r) & = {1\over\lambda_i \lambda_j}\,\E\int_{{\Hd}}\int_B\1\{0<d_h(x,y)\leq r\}\,\mathbf{M}_j(dy)\,\mathbf{M}_i(dx)\\
& = {1\over\lambda_i \lambda_j}\,\E\int_{{\rm skel}_i}\int_{{\rm skel}_j\cap B}\1\{0<d_h(x,y)\leq r\}\,\cH^j(dy)\,\cH^i(dx)
\end{align*}
and describes the random measure $\mathbf{M}_i$ as seen from a typical point of $\mathbf{M}_j$, in the sense of Palm distribution. In particular, we retrieve the ordinary K-function by the special choice $j=i$. The mixed pair-correlation function $g_{ij}$ is then defined in the obvious way by differentiation of $K_{ij}$, namely,
$$
g_{ij}(r) := {1\over \omega_{d}\sinh^{d-1}(r)}\,\frac{dK_{ij}}{ dr}(r),\qquad r> 0.
$$
As in the case of the K-function, the condition that $0<d_h(x,y)$ can be omitted if $i\ge 1$ or $j\ge 1$.

\begin{theorem}[Mixed K-function and mixed pair-correlation function]\label{thm:KFunctionAndPCF}
If  $i,j\in\{0,1,\ldots,d-1\}$, $t>0$ and $r>0$, then
\begin{align}
K_{ij}(r) &= \sum_{n=0}^{m(d,i,j)}n!{d-i\choose n}{d-j\choose n}{\omega_{d+1}\omega_{d-n}\over\omega_{d-n+1}}\bigg({\omega_d\over\omega_{d+1}}{1\over t}\bigg)^{n}\int_0^r\sinh^{d-n-1}(s)\,ds,\label{ctdh1}\\
g_{ij}(r) &= 1+\sum_{n=1}^{m(d,i,j)}n!{d-i\choose n}{d-j\choose n}{\omega_{d-n}\over\omega_{d-n+1}}\bigg({\omega_d\over\omega_{d+1}}\bigg)^{n-1}{1\over (t\sinh(r))^{n}},
\nonumber
\end{align}
where $m(d,i,j):=\min\{d-i,d-j,d-1\}$.
\end{theorem}

In \eqref{ctdh1} we restrict the summation to $n\le d-1$ in order to avoid an undefined expression which arises for
  $i=j=0$ and $n=d$. Alternatively, for $n=d$ the factor $\omega_{d-n}=\omega_0$ is $\omega_0=2/\Gamma(0)=0$ and the product
with the infinite integral can be defined to be zero.

In the special case $d=2$ and for $i=j$ we thus obtain
\begin{align*}
g_0(r) = 1+{4\over \pi\,t}\,{1\over \sinh(r)}\qquad\text{and}\qquad
g_1(r) = 1+{1\over \pi\,t}\,{1\over \sinh(r)},
\end{align*}
and for $d=3$  and  again $i=j$ we get
\begin{align*}
g_0(r) & = 1 + {9\over 2\,t}\,{1\over\sinh(r)} + {36\over \pi^2\,t^2}\,{1\over\sinh^2(r)},\\
g_1(r) & = 1 + {2\over t}\,{1\over\sinh(r)} + {4\over\pi^2\,t^2}\,{1\over\sinh^2(r)},\\
g_2(r) & = 1 + {1\over 2\,t}\,{1\over\sinh(r)},
\end{align*}
see Figure \ref{fig:PCF}.
\begin{figure}[t]
\begin{center}
\begin{minipage}{0.45\columnwidth}
  \begin{tikzpicture}[scale=0.8]
    \begin{axis}[
     clip=false,
     xmin=0,xmax=1,
     xtick={0,0.5,1},
     xticklabels={$0$, $0.5$, $1$}
     ]
      \addplot[domain=0.07:1,samples=200,black]{1+4/(3.1415*sinh(x))};
      \addplot[domain=0.025:1,samples=200,blue]{1+1/(3.1415*sinh(x))};
    \end{axis}
  \end{tikzpicture}
\end{minipage}
\begin{minipage}{0.45\columnwidth}
  \begin{tikzpicture}[scale=0.8]
    \begin{axis}[
     clip=false,
     xmin=0,xmax=1,
     xtick={0,0.5,1},
     xticklabels={$0$, $0.5$, $1$}
     ]
      \addplot[domain=0.2:1,samples=200,black]{1+9/(2*sinh(x))+36/(3.1415^2*sinh(x)^2)};
      \addplot[domain=0.08:1,samples=200,blue]{1+2/(sinh(x))+4/(3.1415^2*sinh(x)^2)};
      \addplot[domain=0.01:1,samples=200,red]{1+1/(2*sinh(x))};
    \end{axis}
  \end{tikzpicture}
\end{minipage}
\end{center}
\caption{Left panel: The pair-correlation functions $g_{0}$ (black curve) and $g_1$ (blue curve) for $d=2$ and $t=1$. Right panel: The pair-correlation functions $g_{0}$ (black curve), $g_1$ (blue curve) and $g_2$ (red curve) for $d=3$ and $t=1$.}
\label{fig:PCF}
\end{figure}
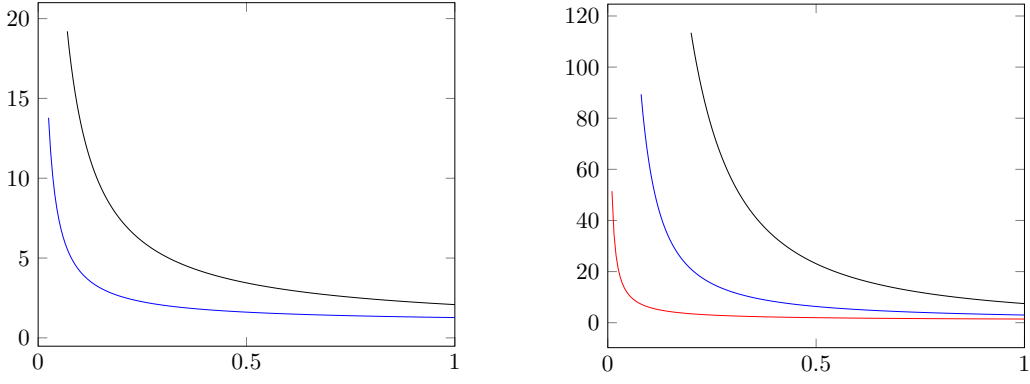


\begin{remark}
An inspection of the proof shows that Theorem \ref{thm:KFunctionAndPCF} is based only on Crofton's formula and Lemma \ref{lem:mu_almost_surely_d-n_hyperplane}, which in turn is also based on Crofton's formula. However, since the latter holds for any space of constant curvature $\kappa\in\{-1,0,1\}$ with the same constant (cf.\ \cite{Brothers,Santalo}), independently of the curvature $\kappa$, Theorem \ref{thm:KFunctionAndPCF} remains valid also in spherical and Euclidean spaces of curvature $\kappa=1$ and $\kappa=0$, respectively. Namely, defining the modified sine function
$$
{\rm sn}_\kappa(r) := \begin{cases}
\sin(r) &: \kappa=1,\\
r &: \kappa=0,\\
\sinh(r) &: \kappa=-1,
\end{cases}
$$
we obtain
\begin{align*}
K_{ij}(r) = \sum_{n=0}^{m(d,i,j)}n!{d-i\choose n}{d-j\choose n}{\omega_{d+1}\omega_{d-n}\over\omega_{d-n+1}}\bigg({\omega_d\over\omega_{d+1}}{1\over t}\bigg)^{n}\int_0^r{\rm sn}_\kappa^{d-n-1}(s)\,ds
\end{align*}
and
\begin{align*}
g_{ij}(r) = 1+\sum_{n=1}^{m(d,i,j)}n!{d-i\choose n}{d-j\choose n}{\omega_{d-n}\over\omega_{d-n+1}}\bigg({\omega_d\over\omega_{d+1}}\bigg)^{n-1}{1\over (t\,{\rm sn}_\kappa(r))^{n}}
\end{align*}
for $r> 0$ if $\kappa\in\{-1,0\}$ and $0< r<\pi$ if $\kappa=1$. For $i=j=d-1$ and $\kappa=1$ these formulas have been proved in
{\cite[Section 6.2]{HuThspstitt}} based on a different normalization. Moreover, for $\kappa=0$ the formula for $g_0(r)$ appears as the identity (3.15) in \cite{HeinrichMuchePVT}, while $g_{d-1}(r)$ can be found in {\cite[Section 7]{SchreiberThaele2ndOrder}}. As already explained in \cite{HeinrichSchmidtSchmidtCLT}, for general $i\in\{0,1,\ldots,d-1\}$ it can in principle be deduced from an explicit formula for the second-order moments of the total volume of intersection processes, see \cite[p.~164]{Matheron}.
\end{remark}

\subsection{Limit theorems}\label{subsec:LimitTheorems}

Our next result is a central limit theorem for $F_{W,t}^{(i)}$, for a fixed hyperbolic convex body $W$, when the intensity parameter $t$ tends to infinity. We will measure the distance between (the laws of) two random variables by the Wasserstein and the Kolmogorov distance. For their definitions we refer to Section \ref{subsec:NormalApproxUstatistics} below.

\begin{theorem}[CLT, growing intensity]\label{thm:CLTtToInfinity}
Let $d\geq 2$, $i\in\{0,1,\ldots,d-1\}$ and let $W\in\cK_h^d$ be a fixed hyperbolic convex body with non-empty interior. Let $N$ be a standard Gaussian random variable, and let $d(\,\cdot\,,\,\cdot\,)$ denote either the Wasserstein or the Kolmogorov distance. Then there exists a constant $c\in(0,\infty)$ such that
$$
d\left(\frac{F_{W,t}^{(i)}- \mathbb{E}F_{W,t}^{(i)}}{\sqrt{\Var F_{W,t}^{(i)}}},N \right) \leq c\,t^{-1/2}
$$
for all $t\geq 1$.
\end{theorem}

As already explained in the introduction, the central limit problem for $F_{W,t}^{(i)}$ can also be approached in another set-up, which in the Euclidean case is equivalent to the one just discussed, but turns out to be fundamentally different in hyperbolic space. More precisely, we turn now to the case, where the intensity $t$ is fixed, while the size of the observation window is increased. We do this only in the case of spherical windows in $\Hd$. In other words, we choose for $W$ the hyperbolic ball $B_r$ (around the origin $p$) and write $F_{r,t}^{(i)}$ instead of $F_{B_r,t}^{(i)}$ in this case. Our next result is a central limit theorem for $F_{r,t}^{(i)}$ for dimension $d=2$ in part (a) and for $d=3$ in part (b). Moreover, it turns out that a central limit theorem for $F_{r,t}^{(i)}$ is no longer valid in any space dimension $d\geq 4$, see part (c). We emphasize that this surprising phenomenon is in sharp contrast to the Euclidean case \cite{HeinrichCLTforPHT,LPST,ReitznerSchulteCLT} and is an effect of the negative curvature.

\begin{theorem}[CLT, growing spherical window]\label{thm:CLTrToInfinity}
Let $t\geq 1$, let $N$ be a standard Gaussian random variable, and let  $d(\,\cdot\,,\,\cdot\,)$ denote either the Wasserstein or the Kolmogorov distance.
\begin{itemize}
\item[{\rm (a)}] If $d=2$, then there is a constant $c_2\in(0,\infty)$ only depending on $t$ such that
$$
d\left(\frac{F_{r,t}^{(i)}- \mathbb{E}F_{r,t}^{(i)}}{\sqrt{\Var F_{r,t}^{(i)}}},N \right) \leq c_2\,r^{1-i}\,e^{-r/2}
$$
for  $i\in\{0,1\}$ and $r\geq 1$.
\item[{\rm (b)}]If $d=3$, then there is a constant $c_3\in(0,\infty)$ only depending on $t$ such that
$$
d\left(\frac{F_{r,t}^{(i)}- \mathbb{E}F_{r,t}^{(i)}}{\sqrt{\Var F_{r,t}^{(i)}}},N \right) \leq \begin{cases}
c_3\,r^{-1} &: i=2,\\
c_3\,r^{-1/2} &: i\in\{0,1\},\\
\end{cases}
$$
for $r\geq 1$.
\item[{\rm (c)}] If $d\geq 4$ and $i=d-1$ or if $d\ge 7$ and $i\in\{0,1,\ldots,d-1\}$, then a central limit theorem for $(F_{r,t}^{(i)}- \mathbb{E}F_{r,t}^{(i)})/\sqrt{\Var   F_{r,t}^{(i)}}$ does not hold for $r\to\infty$.
\end{itemize}
\end{theorem}

\begin{remark}\rm
\begin{itemize}
\item[(i)] The restriction imposed on the parameters $d,i$ in Theorem \ref{thm:CLTrToInfinity} (c) is the result of a number of technical obstacles one needs to overcome in its proof. We strongly believe that a central limit theorem in fact fails for all $d\geq 4$ and all choices of $i\in\{0,1,\ldots,d-1\}$. However, we have to leave this as an open problem for future work. For some remarks about the potential limiting distribution in Theorem \ref{thm:CLTrToInfinity} (c) we refer to Remark \ref{rem:LimitDistribution}.

\item[(ii)] It is instructive to rewrite the normal approximation bounds in Theorem \ref{thm:CLTrToInfinity} (a) and (b) as follows. For $d=2$ and $i\in\{0,1\}$ we have that
$$
d\left(\frac{F_{r,t}^{(i)}- \mathbb{E}F_{r,t}^{(i)}}{\sqrt{\Var F_{r,t}^{(i)}}},N \right) \leq \hat{c}_2 \, {\log^{1-i}\cH^2(B_r)\over\sqrt{\cH^2(B_r)}},\qquad r\geq 1,
$$
and for $d=3$ we have, again for $r\geq 1$,
$$
d\left(\frac{F_{r,t}^{(i)}- \mathbb{E}F_{r,t}^{(i)}}{\sqrt{\Var F_{r,t}^{(i)}}},N \right) \leq \hat{c}_3 \begin{cases}
{1\over\log\cH^3(B_r)} &: i=2,\\
{1\over\sqrt{\log\cH^3(B_r)}} &: i\in\{0,1\}.\\
\end{cases}
$$
Here $\hat{c}_2,\hat{c}_3\in(0,\infty)$ are again constants only depending on $t$. This means that in dimension $d=2$ and for $i=0$ the speed of convergence is the same as in the Euclidean case, up to the logarithmic factor. Moreover, it shows that $d=3$ is the critical dimension for the central limit theorem, which only holds in this case with a rate of convergence which is very much slowed down.
\end{itemize}
\end{remark}

Theorem \ref{thm:CLTtToInfinity} shows that for fixed radius $r$ and increasing intensity $t$ a central limit theorem for $F_{r,t}^{(i)}$ with $i\in\{0,1,\ldots,d-1\}$ holds. On the other hand, according to Theorem \ref{thm:CLTrToInfinity} (c) the central limit theorem breaks down for dimensions $d\geq 4$ (if the total surface area is considered) or $d\geq 7$ (for general $i\in\{0,1,\ldots,d-1\}$) if the intensity $t$ stays fixed and $r\to\infty$. Against this background the question arises whether in these cases the central limit behaviour can be preserved if the intensity $t$ and the radius $r$ tend to infinity \textit{simultaneously}. In fact, the following result states that this is indeed the case. More precisely, it says that, independently of the behaviour of $r$, the central limit theorem holds as soon as $t\to\infty$ (and $r$ is bounded from below by $1$).

\begin{theorem}[CLT for simultaneous growth of intensity and window]\label{thm:RandTtoInfinity}
Let $d\geq 4$ and $i=d-1$ or $d\geq 7$ and $i\in\{0,1,\ldots,d-1\}$. Also, let $N$ be a standard Gaussian random variable. Then there is a constant $c\in(0,\infty)$ such that
$$
d\left(\frac{F_{r,t}^{(i)}- \mathbb{E}F_{r,t}^{(i)}}{\sqrt{\Var F_{r,t}^{(i)}}},N \right) \leq {c\over\sqrt{t}}
$$
for all $r\geq 1$ and $t\geq 1$, where $d(\,\cdot\,,\,\cdot\,)$ denotes either the Wasserstein or the Kolmogorov distance.
\end{theorem}

\begin{remark}
In dimensions $d=2$ and $d=3$ we also have normal approximation bounds that simultaneously involve the two parameters $t$ and $r$. In fact, for $d=2$ the bounds \eqref{eq:Boundd=2i=1} and \eqref{eq:Boundd=2i=0} below show that
$$
d\left(\frac{F_{r,t}^{(i)}- \mathbb{E}F_{r,t}^{(i)}}{\sqrt{\Var F_{r,t}^{(i)}}},N \right) \leq c\,t^{-1/2}\,r^{1-i}e^{-r/2}
$$
holds for all $t\geq 1$, $r\geq 1$ and $i\in\{0,1\}$. Similarly, for $d=3$ the estimates \eqref{eq:Boundd=3i=2}, \eqref{eq:Boundd=3i=1} and \eqref{eq:Boundd=3i=0} prove that
$$
d\left(\frac{F_{r,t}^{(i)}- \mathbb{E}F_{r,t}^{(i)}}{\sqrt{\Var F_{r,t}^{(i)}}},N \right) \leq c\cdot\begin{cases}
t^{-1/2}r^{-1} &: i=2,\\
t^{-1/2}r^{-1/2} &: i\in\{0,1\},
\end{cases}
$$
for all $t\geq 1$ and $r\geq 1$. In both cases, $d(\,\cdot\,,\,\cdot\,)$ stands for either the Wasserstein or the Kolmogorov distance. This way we recover Theorem \ref{thm:CLTtToInfinity} for $d=2$ and $d=3$ in the special case where $W=B_r$ with $r$ fixed and we recover Theorem \ref{thm:CLTrToInfinity} (a) and (b) by fixing $t$.
\end{remark}

Finally, let us turn to the multivariate set-up. To compare the distance between the distributions of (the laws of) two random vectors we use what is known as the $d_2$- and the $d_3$-distance; for their  definition we refer to Section \ref{subsec:3.3} below. We approach the multivariate central limit theorem by considering, as above, two different settings. To handle the central limit problem for a fixed window $W\in\cK_h^d$ and growing intensities we define for $t>0$ the $d$-dimensional random vector
\begin{align*}
\textbf{F}_{W,t} := \Bigg({F_{W,t}^{(0)}-\E F_{W,t}^{(0)}\over t^{d-1/2}},\ldots,{F_{W,t}^{(i)}-\E F_{W,t}^{(i)}\over t^{d-i-1/2}},\ldots,{F_{W,t}^{(d-1)}-\E F_{W,t}^{(d-1)}\over t^{1/2}}\Bigg).
\end{align*}
Moreover, for $i,j\in\{0,1,\ldots,d-1\}$ we introduce the asymptotic covariances and the asymptotic covariance matrix of the random vector $\textbf{F}_{W,t}$, as $t\to\infty$, by
$$
\tau^{i,j}_W := \lim_{t\to\infty}\Cov\Bigg({F_{W,t}^{(i)}-\E F_{W,t}^{(i)}\over t^{d-i-1/2}},{F_{W,t}^{(j)}-\E F_{W,t}^{(j)}\over t^{d-j-1/2}}\Bigg),\qquad T_W:=\left(\tau^{i,j}_W\right)_{i,j=0}^{d-1}.
$$
The existence of the limit and the precise value of  $\tau^{i,j}_W$  follows from \eqref{eq:covariance_general} below. It is easy to
see that $T_W$ has rank one, as in  Euclidean space.

In view of Theorem \ref{thm:CLTrToInfinity}, for fixed intensity $t>0$ and a sequence of growing spherical windows, taking $W=B_r$ for $r>0$ we put
\begin{align*}
\textbf{F}_{r,t} := \begin{cases}
\Bigg({F_{r,t}^{(0)}-\E F_{r,t}^{(0)}\over e^{r/2}},{F_{r,t}^{(1)}-\E F_{r,t}^{(1)}\over e^{r/2}}\Bigg) &: d=2,\\
\Bigg({F_{r,t}^{(0)}-\E F_{r,t}^{(0)}\over \sqrt{r}\,e^{r}},{F_{r,t}^{(1)}-\E F_{r,t}^{(1)}\over \sqrt{r}\,e^{r}},{F_{r,t}^{(2)}-\E F_{r,t}^{(2)}\over \sqrt{r}\,e^{r}}\Bigg) &: d=3,\\
\Bigg({F_{r,t}^{(0)}-\E F_{r,t}^{(0)}\over e^{r(d-2)}},\ldots,{F_{r,t}^{(d-1)}-\E F_{r,t}^{(d-1)}\over e^{r(d-2)}}\Bigg) &: d\geq 4,
\end{cases}
\end{align*}
and define the asymptotic covariance matrix $\Sigma_d=\left(\sigma^{i,j}_d\right)_{i,j=0}^{d-1}$ of the
random vector $\textbf{F}_{r,t}$,  as $r\to\infty$, for $d \geq 2$ by
$$
\sigma^{i,j}_d: = \begin{cases}
\lim\limits_{r\to\infty}\Cov\Bigg({F_{r,t}^{(i)}-\E F_{r,t}^{(i)}\over e^{r/2}},{F_{r,t}^{(j)}-\E F_{r,t}^{(j)}\over e^{r/2}}\Bigg) &: d=2,\\
\lim\limits_{r\to\infty}\Cov\Bigg({F_{r,t}^{(i)}-\E F_{r,t}^{(i)}\over \sqrt{r}\,e^{r}},{F_{r,t}^{(j)}-\E F_{r,t}^{(j)}\over \sqrt{r}\,e^{r}}\Bigg) &: d=3,\\
\lim\limits_{r\to\infty}\Cov\Bigg({F_{r,t}^{(i)}-\E F_{r,t}^{(i)}\over e^{r(d-2)}},{F_{r,t}^{(j)}-\E F_{r,t}^{(j)}\over e^{r(d-2)}}\Bigg) &: d\geq 4.
\end{cases}
$$
The covariance matrices $\Sigma_d$ are explicitly given by \eqref{eq:Sigma_2} for $d=2$, \eqref{eq:Sigma_3} for $d=3$ and \eqref{eq:Sigma_4} for $d \geq 4$ below. Moreover, in Section \ref{sec:4.5} we determine convergence rates.
In particular, we will show that $\Sigma_2$ has full rank (is positive definite) and $\Sigma_d$ has rank one for $d \geq 3$. We remark that this is in sharp contrast to the corresponding result in Euclidean spaces, where the asymptotic covariance matrix has rank one for all $d\geq 2$, see \cite[Theorem 5.1 (ii)]{HeinrichCLTforPHT}. Note that the dependence of these limits on the fixed intensity $t>0$ is not made explicit by our notation, but this dependence is shown in Lemmas 20, 21 and 23.

In order to state the multivariate central limit theorem, we use the $d_2$ and the $d_3$ distance for random vectors (see Section \ref{subsec:3.3} for explicit definitions).

\begin{theorem}[Multivariate CLT]\label{thm:CLTMultivariate}
\begin{itemize}
\item[{\rm (a)}] Let $d\geq 2$ and $W\in\cK_h^d$. Let $N_{T_W}$ be a $d$-dimensional centred Gaussian random vector with covariance matrix $T_W$. Then there exists a constant $c\in(0,\infty)$ such that
$$
d_3(\textbf{F}_{W,t},N_{T_W}) \leq c\,t^{-1/2}
$$
for all $t\geq 1$.

\item[{\rm (b)}] Fix $t\geq 1$ and let $d=2$. Let $N_{\Sigma_2}$ be a $2$-dimensional centred Gaussian random vector with covariance matrix $\Sigma_2$. Then there exists a constant $c_2\in(0,\infty)$ such that
  $$
  d_j(\textbf{F}_{r,t},N_{\Sigma_2}) \leq c_2\,r\,e^{-r/2}
  $$
  for all $r\geq 1$ and $j\in\{2,3\}$.
\item[{\rm (c)}] Fix $t\geq 1$ and let $d=3$. Let $N_{\Sigma_3}$ be a $3$-dimensional centred Gaussian random vector with covariance matrix $\Sigma_3$. Then there exists a constant $c_3\in(0,\infty)$ such that
  $$
  d_3(\textbf{F}_{r,t},N_{\Sigma_3}) \leq c_3\,r^{-1/2}
  $$
  for all $r\geq 1$.
\end{itemize}
\end{theorem}

\begin{remark}
After having seen that in the univariate case the central limit theorem for $d\geq 4$ can be preserved by a simultaneous growth of the intensity $t$ and the radius $r$, the question arises whether such a phenomenon also holds in the multivariate set-up. This is in fact the case, but we decided not to present the details for brevity.
\end{remark}

\section{Background material and preparations}\label{sec:Background}

\subsection{More hyperbolic geometry}\label{sec:3.1}

Recall that by $\Hd$ we denote the hyperbolic space of dimension $d$. For concreteness we may take as a specific model for $\Hd$ the $d$-dimensional Euclidean unit ball $B_{\rm euc}^d$ supplied with the Poincar\'e metric $d_h$ given by
$$
\cosh d_h(x,y) := 1+{2\|x-y\|_{\rm euc}^2\over (1-\|x\|_{\rm euc}^2)(1-\|y\|_{\rm euc}^2)},\qquad x,y\in B_{\rm euc}^d,
$$
where $\|\,\cdot\,\|_{\rm euc}$ stands for the usual Euclidean norm. This is known as the conformal ball model for $\Hd$, see {\cite[Chapter 4.5]{Ratcliffe}}. However, it should be emphasized that
our arguments are independent of the special choice of a model for a simply connected, geodesically complete space  of constant negative curvature $\kappa=-1$.
We write $B(z,r)=\{x\in\Hd:d_h(x,z)\leq r\}$ for the hyperbolic ball with centre $z\in\Hd$ and radius $r\ge 0$ and put $B_r=B(p,r)$, where $p$ is a fixed reference point. In this paper the $s$-dimensional Hausdorff measure $\cH^s$, $s\geq 0$, is understood with respect to the metric space $(\Hd,d_h)$.

For later reference we need a formula for the surface area of a hyperbolic ball $B(z,r)$. It is given by
$$
\mathcal{H}^{d-1}(\partial B(z,r))=\omega_d \sinh^{d-1}(r),
$$
where $\omega_d=d \kappa_d ={2 \pi^{d/2}}/{\Gamma(d/2)}$
is the surface area of a $d$-dimensional unit ball in the Euclidean space $\R^d$ and $\kappa_d$ is its volume. Moreover, the volume of a hyperbolic ball of radius $r$ is given by
\begin{equation}\label{volbalr}
\mathcal{H}^{d}(B(z,r))= \omega_d \int_{0}^{r} \sinh^{d-1}(s) \ ds.
\end{equation}
We refer to Sections 3.3 and 3.4 and especially to formulas (3.25) and (3.26) in the monograph \cite{Chavel}.
For the special case $d=2$, we thus get $\mathcal{H}^2(B(z,r))= 2 \pi(\cosh(r)-1)$. Here, $\cosh$ and $\sinh$ are the hyperbolic cosine and sine, which are given by
$$\cosh(x)=\frac{e^x+e^{-x}}{2}\qquad\text{and}\qquad\sinh(x)=\frac{e^x-e^{-x}}{2},\qquad x\in\R,
$$
respectively. We will frequently make use of the fact that $\cosh(x), \sinh(x) \in \Theta(e^x)$, as $x\to \infty$, where $\Theta(\,\cdot\,)$ stands for the usual Landau symbol. Additionally we will use the following inequalities.

\begin{Lemma}\label{lem:inequalities}
The function $\sinh$ satisfies the inequalities
\begin{align*}
  &{\rm (a)}\ \sinh(x)\geq e^{x-3} \quad \text{for } x \geq 0.1, && {\rm (b)} \ \sinh(x)\ge x \quad \text{for } x\ge 0.
\end{align*}
\end{Lemma}

\begin{proof}
\begin{itemize}
  \item [(a)] By the definition of the hyperbolic sine function, we get
  $$\frac{2 \sinh(x)}{e^{x-3}}=e^3-e^{-2x+3}=e^3(1-e^{-2x}) \geq 2 \quad\text{for }x\ge 0.1,$$
since $\exp(2x)\ge (1-2\exp(-3))^{-1}$ for $x\ge 0.1$.
  \item [(b)] This follows from the definition of  $\sinh$ by basic calculus.
\end{itemize}
\end{proof}

Let $I(\Hd)$ denote the isometry group of $\Hd$ and let $I(\Hd,p)$ denote the subgroup of isometries which fix $p$. We remark that in the conformal ball model, $I(\Hd)$ can be identified with the group of M\"obius transformations of $B_{\rm euc}^d$, see {\cite[Corollary 4.5.1]{Ratcliffe}}. We denote by $G_h(d,k)$ the compact space of $k$-dimensional totally geodesic subspaces containing the origin $p$. In the conformal ball model, all elements of $G_h(d,k)$ arise as follows. If $p$ coincides with the centre $o$ of $B_{\rm euc}^d$, then an element of $G_h(d,k)$ is the intersection of $B_{\rm euc}^d$ with a $k$-dimensional Euclidean linear subspace of $\R^d$. If otherwise $p\neq o$, then an element of $G_h(d,k)$ is the intersection of $B_{\rm euc}^d$ with a $k$-dimensional Euclidean sphere in $\R^d$ through $p$ which is orthogonal to the boundary of $B_{\rm euc}^d$, cf.\ {\cite[Theorem 4.5.3]{Ratcliffe}}. Up to a scaling factor, $G_h(d,k)$ carries a regular Borel measure $\nu_k$ which is invariant under $I(\Hd,p)$. Since $G_h(d,k)$ is compact we can normalize $\nu_k$ such that $\nu_k(G_h(d,k))=1$. Recall that $A_h(d,k)$ is the space of $k$-dimensional planes in $\mathbb{H}^{d}$. In the conformal ball model all elements of $A_h(d,k)$ can be represented as intersections with $B_{\rm euc}^d$ of either $k$-dimensional Euclidean linear subspace of $\R^d$ or $k$-dimensional Euclidean spheres in $\R^d$ that are orthogonal to the boundary of $B_{\rm euc}^d$. On $A_h(d,k)$ there exists a unique (up to scaling) $I(\Hd)$-invariant measure. In contrast to $G_h(d,k)$, the larger space $A_h(d,k)$ is not compact. Each $k$-plane $H \in A_h(d,k)$ is uniquely determined by its orthogonal subspace $L_{d-k}$ passing through the origin $p$ and the intersection point $\{x\}=H \cap L_{d-k}$. Using these facts, Santal\'o {\cite[Equation (17.41)]{Santalo}} (see also \cite[Proposition 2.1.6]{Solanes}, \cite[Equation (9)]{GallegoSolanes})  provides a useful representation of an isometry invariant measure on $A_h(d,k)$, which we use here with a different normalization. For a Borel set $B \subset A_h(d,k)$, it is given by
\begin{equation}\label{eq:Croftonmass}
\mu_k(B)= \int_{G_h(d,d-k)} \int_{L}\cosh^{k}(d_h(x,p)) \, \mathds{1}\{H(L,x) \in B\} \ \mathcal{H}^{d-k}(dx) \ \nu_{d-k}(d L),
\end{equation}
where $H(L,x)$ is the $k$-plane orthogonal to $L$ passing through $x$.

\begin{remark}
The current normalization of the measure $\mu_k$ differs from the normalization of the measure $dL_k$ used in \cite{Santalo} by the constant  $ {\omega_d \cdots \omega_{d-k+1}}/({\omega_k \cdots \omega_1})$. This also affects the constants in the formulas from hyperbolic integral geometry taken from \cite{Santalo}. The reason for the present normalization is to simplify a comparison of our results to corresponding results in Euclidean and  spherical space.
\end{remark}

According to {\cite[Equation (14.69)]{Santalo}} the measure $\mu_k$  satisfies  the following Crofton-type formula. In fact, the discussion in {\cite[Section 7]{Brothers}} allows us to state the result not only for sets bounded by smooth submanifolds (as in \cite{Santalo}), but for much more general sets, which include arbitrary convex sets as a very special case. The following lemma holds for $\mathcal{H}^{d+i-k}$ measurable sets $W\subset\Hd$ which are Hausdorff  $(d+i-k)$-rectifiable. Following {\cite[Definition 5.13]{Brothers}}, we say that a set $W\subset\Hd$ is
$\ell$-rectifiable if  $\ell$ is an integer with $0< \ell\le d$
and $W$ is the image of some bounded subset of $\R^\ell$ under a Lipschitz map from $\R^\ell$ to $\Hd$. A set $W\subset\Hd$ is Hausdorff
$\ell$-rectifiable provided that $\mathcal{H}^\ell(W)<\infty$ and if there exist $\ell$-rectifiable subsets $B_1,B_2,\ldots$ of $\Hd$ such that $\mathcal{H}^\ell(W\setminus \bigcup_{i\ge 1}B_i)=0$. Clearly, any Borel set $W$  which is contained in an $\ell$-dimensional plane is Hausdorff $\ell$-rectifiable if it satisfies $\mathcal{H}^\ell(W)<\infty$.

\begin{Lemma}\label{lem:Crofton}
Let  $0 \leq i \leq k \leq d-1$, and let $W\subset\Hd$ be a Borel set which is Hausdorff $(d+i-k)$-rectifiable.  Then
\begin{align}\label{eq:Crofton_volume}
\int_{A_h(d,k)} \mathcal{H}^{i}(W \cap E) \ \mu_k(d E) = \frac{\omega_{d+1} \, \omega_{i+1}}{\omega_{k+1} \, \omega_{d-k+i+1}}\,\mathcal{H}^{d+i-k}(W).
\end{align}
\end{Lemma}

\begin{remark}
Strictly speaking the case $k=i$ is not covered by \cite{Brothers}. Although the framework in \cite{Brothers} should extend to this marginal case, we prefer to provide an elementary direct argument for the case $k=i$. In this case, the left side of \eqref{eq:Crofton_volume} defines an isometry invariant Borel measure on $\Hd$. Therefore, in order to confirm \eqref{eq:Crofton_volume} in this case, it is
sufficient to show that the equality holds for $W=B_r$, $r\ge 0$. Since equality holds for $r=0$ and in view of \eqref{volbalr}, it is sufficient to show that $\omega_d\sinh^{d-1}(r)$ is the derivative with respect to
$r$ of the function defined by
\begin{align*}
h(r):&= \int_{A_h(d,k)} \mathcal{H}^{i}(B_r \cap E) \ \mu_k(d E)\\
&=\omega_k\omega_{d-k}\int_0^r\sinh^{d-k-1}(t)\cosh^k(t) \int_0^{\arcosh \left(\frac{\cosh(r)}{\cosh(t)}\right)}\sinh^{k-1}(s)\, ds\, dt,
\end{align*}
where we used \eqref{eq:Croftonmass} and \eqref{eq:volume_ball} for the equality. The differential of $h$ can be determined by basic rules of calculus. Using that $\arcosh(\cosh(r)/\cosh(r))=0$, we thus obtain
$$
h'(r)=\omega_k\omega_{d-k}\int_0^r\sinh^{d-k-1}(t)\sinh(r)\big(\cosh^2(r)-\cosh^2(t)\big)^{\,k-2}\, dt.
$$
The substitution $\sinh(t)=\sinh(r)\cdot x$ leads to
\begin{align*}
h'(r)&= \omega_k\omega_{d-k}\int_0^1 x^{d-k-1}\big({1-x^2}\big)^{\,k-2}\, dx \, \sinh^{d-1}(r)=\omega_d\sinh^{d-1}(r),
\end{align*}
as was to be shown.
\end{remark}

\begin{remark}
Although both sides of \eqref{lem:Crofton} define measures with respect to their dependence on a Borel set $W\subset\Hd$, for $k\neq i$ the
equality in \eqref{lem:Crofton} in general does not extend from $(d+i-k)$-rectifiable sets to general Borel sets. This is due to deep classical
results in the structure theory of geometric measure theory, see \cite[p.~2]{Federer69} or \cite[Chapter 3]{Morgan} for an introduction and
\cite[Theorem 3.3.13]{Federer69} for the general treatment. In fact, in the Euclidean setting, for $i=0$,
$k\in\{1,\ldots,d-1\}$ and for a general Borel set $W\subset\R^d$, the right side of \eqref{lem:Crofton} is always as large as the left side with equality
if and only if $W$ is $(d-k)$-rectifiable.
\end{remark}

We will frequently make use of the following transformation formula.

\begin{Lemma}\label{lem:mu_almost_surely_d-n_hyperplane}
Let $k \in \{0,\ldots,d-1\}$, and let $f:A_h(d,d-1)^{d-k}\to\R$ be a non-negative measurable function satisfying $f(H_1,\ldots,H_{d-k})=0$ if $\textrm{\rm dim}(H_1\cap\ldots\cap H_{d-k})\neq k$. Then
$$
\int_{A_{h}(d,d-1)^{d-k}} f(H_1 \cap \ldots \cap H_{d-k}) \,\mu_{d-1}^{d-k}(d(H_1,\ldots,H_{d-k}))=c(d,k) \int_{A_{h}(d,k)} f(E) \ \mu_{k}(d E)
$$
with
$$
c(d,k)=\frac{\omega_{k+1}}{\omega_{d+1}}\left(\frac{\omega_{d+1}}{\omega_d}\right)^{d-k}.
$$
\end{Lemma}
\begin{proof}
  Let $\varphi \in I(\mathbb{H}^d)$ be an arbitrary isometry, and let $B$ a measurable subset of $A_{h}(d,k)$. Then we have
  \begin{align*}
    &\mu_{d-1}^{d-k}(\{(H_1,\ldots,H_{d-k}) \in A_{h}(d,d-1)^{{d-k}}: \ H_1 \cap \ldots \cap H_{d-k} \in \varphi B\}) \\
    &\qquad = \mu_{d-1}^{d-k}(\{(H_1,\ldots,H_{d-k}) \in A_{h}(d,d-1)^{{d-k}}: \  H_1 \cap \ldots \cap  H_{d-k} \in  B\})
  \end{align*}
by the isometry invariance of $\mu_{d-1}$. Since up to a multiplicative constant, $\mu_{k}$ is the only isometry invariant measure on $A_h(d,k)$, the formula follows up to the determination of the constant, which is independent of the function $f$. We do this by choosing
$$f(H_1,\ldots,H_{d-k})=
\begin{cases}
   \cH^k(H_1\cap\ldots\cap H_{d-k}\cap W) & : \text{dim}(H_1\cap\ldots\cap H_{d-k})=k, \\
   0 & :\mbox{otherwise},
\end{cases}
$$
where $W\in\cK_h^d$ is a fixed convex body with $\cH^d(W)=1$.
We compute
    \begin{align*}
      & \int_{A_{h}(d,d-1)^{d-k}} f(H_1 \cap \ldots \cap H_{d-k} \cap W) \ \mu_{d-1}^{d-k}(d(H_1,\ldots,H_{d-k})) \\
&\qquad= \left(\prod_{i=k}^{d-1} \frac{\omega_{d+1}}{\omega_d} \frac{\omega_{i+1}}{\omega_{i+2}} \right) \mathcal{H}^d(W) = \frac{\omega_{k+1}}{\omega_{d+1}}\left(\frac{\omega_{d+1}}{\omega_d}\right)^{d-k}
    \end{align*}
    by a $(d-k)$-fold application of the Crofton formula \eqref{eq:Crofton_volume} with the choice $k=d-1$ and (successively) $i=k,k+1,\ldots,d-1$ there. On the other hand, applying directly the Crofton formula with $i=k$, we get
    $$\int_{A_h(d,k)} \mathcal{H}^{k}(W \cap E) \ \mu_k(d E) = \mathcal{H}^{d}(W)=1.  $$
    A comparison yields the constant and proves the assertion of the lemma.
\end{proof}

In what follows we use the convention that $\text{dim}(\varnothing)=-1$.

\begin{Lemma}\label{lem:mu_d_measure}
Fix $d\geq 2$ and let $n\in\{1,\ldots,d\}$. Then $\text{\rm dim}(H_1 \cap \ldots \cap H_{n})\in\{-1,d-n\}$ holds for $\mu_{d-1}^n$-almost all $(H_1,\ldots,H_n)\in A_h(d,d-1)^n$.
\end{Lemma}
\begin{proof}
We apply induction over $n$ and start by observing that for $n=1$ there is nothing to show. For $n \geq 2$ we have
  $$\mu_{d-1}^{n-1}(\{(H_1,\ldots,H_{n-1}) \in A_{h}(d,d-1)^{n-1}: \ \text{dim}(H_1 \cap \ldots \cap H_{n-1}) \not \in \{-1,d-(n-1)\}\})=0$$
  by the induction hypothesis. Let us introduce the abbreviation $L_{d-k} \defeq H_1 \cap \ldots \cap H_k$ for $H_1,\ldots,H_k\in A_h(d,d-1)$ and $k\in\{1,\ldots,d\}$. We obtain
  \begin{align*}
     & \mu_{d-1}^{n}(\{(H_1,\ldots,H_{n}) \in A_{h}(d,d-1)^{n}: \ \text{dim}(H_1 \cap \ldots \cap H_{n}) \not \in \{-1,d-n\}\}) \\
     & \qquad = \int_{A_{h}(d,d-1)^n} \mathbf{1} \{\text{dim}(L_{d-n})\not \in \{-1,d-n\} \} \ \mu_{d-1}^{n}(d(H_1,\ldots,H_n)).
\end{align*}
We decompose the indicator function as follows:
\begin{equation}\label{eq:040919}
\begin{split}
&\mathbf{1} \{\text{dim}(L_{d-n})\not \in \{-1,d-n\} \}\\
&\qquad = \mathbf{1} \{\text{dim}(L_{d-n})\not \in \{-1,d-n\}, \ \text{dim}(L_{d-(n-1)})=d-(n-1) \} \\
&\qquad\qquad +\mathbf{1} \{\text{dim}(L_{d-n})\not \in \{-1,d-n\}, \ \text{dim}(L_{d-(n-1)})=-1 \}\\
&\qquad\qquad +\mathbf{1} \{\text{dim}(L_{d-n})\not \in \{-1,d-n\}, \ \text{dim}(L_{d-(n-1)}) \not \in \{-1,d-(n-1)\} \}.
\end{split}
\end{equation}
Since the second indicator function on the right-hand side is identically equal to zero, we arrive at
\begin{align*}
& \mu_{d-1}^{n}(\{(H_1,\ldots,H_{n}) \in A_{h}(d,d-1)^n: \ \text{dim}(H_1 \cap \ldots \cap H_{n}) \not \in \{-1,d-n\}\}) \\
      &\qquad \le  \int_{A_{h}(d,d-1)^n} \mathbf{1} \{\text{dim}(L_{d-n})\not \in \{-1,d-n\}, \ \text{dim}(L_{d-(n-1)})=d-(n-1) \} \\
     &\qquad\qquad \qquad \qquad + \mathbf{1} \{\text{dim}(L_{d-(n-1)}) \not \in \{-1,d-(n-1)\} \} \ \mu_{d-1}^{n}(d(H_1,\ldots,H_n)).
  \end{align*}
  By the induction hypothesis and Fubini's theorem we get
  \begin{align*}
    \int_{A_{h}(d,d-1)^n} \mathbf{1} \{\text{dim}(L_{d-(n-1)}) \not \in \{-1,d-(n-1)\} \} \ \mu_{d-1}^{n}(d(H_1,\ldots,H_n))
    = 0,
  \end{align*}
  which covers the case of the third indicator function on the right-hand side of \eqref{eq:040919}. Finally, we write
   $c(H_1,\ldots,H_{n-1})$ for an arbitrary point chosen on $H_1 \cap \ldots \cap H_{n-1}$ (in a measurable way). Then, again by Fubini's theorem, we conclude for the first indicator function on the right-hand side of \eqref{eq:040919} that
  \begin{align*}
  & \mu_{d-1}^{n}(\{(H_1,\ldots,H_{n}) \in A_{h}(d,d-1)^n: \ \text{dim}(L_{d-n}) \notin \{-1,d-n \}, \\
   & \hspace{6.932cm} \text{dim}(L_{d-(n-1)})=d-(n-1)\}) \\
       &\qquad  \leq  \int_{A_{h}(d,d-1)^{n-1}} \int_{A_{h}(d,d-1)} \mathbf{1} \{H_1 \cap \ldots \cap H_{n-1} \subseteq H_n, \ H_1 \cap \ldots \cap H_{n-1} \neq \emptyset \} \\
       & \hspace{6.932cm} \times\, \mu_{d-1}(d H_n) \ \mu_{d-1}^{n-1}(d(H_1,\ldots,H_{n-1}))  \\
     & \qquad \leq \int_{A_{h}(d,d-1)^{n-1}} \int_{A_{h}(d,d-1)} \mathbf{1} \{c(H_1, \ldots, H_{n-1}) \in H_n \} \ \mu_{d-1}(d H_n) \ \mu_{d-1}^{n-1}(d(H_1,\ldots,H_{n-1})) \\
      &\qquad   =  \int_{A_{h}(d,d-1)^{n-1}}  0 \ \mu_{d-1}^{n-1}(d(H_1,\ldots,H_{n-1}))  = 0.
  \end{align*}
  This completes the proof.
\end{proof}

\subsection{Poisson U-statistics}\label{subsec:NormalApproxUstatistics}

Let $(\mathbb{X},\mathcal{X})$ be a measurable space, which is supplied with a $\sigma$-finite measure $\mu$. Let $\eta$ be a Poisson process on $\mathbb{X}$ with intensity measure $\mu$ (we refer to \cite{LP} for a formal construction). Further, fix $m\in\N$ and let $h:\mathbb{X}^m\to\R$ be a non-negative, measurable and symmetric function, which is integrable with respect to $\mu^m$, the $m$-fold product measure of $\mu$. By a Poisson U-statistic (of order $m$ and with kernel $h$) we understand a random variable of the form
$$
\mathscr{U} = \sum_{(x_1,\ldots,x_m)\in\eta_{\neq}^m}h(x_1,\ldots,x_m),
$$
where $\eta_{\neq}^m$ is the collection of all $m$-tuples of distinct points of $\eta$, see \cite{LP}. Functionals of this type have received considerable attention in the literature, especially in connection with applications in stochastic geometry, see, for example, \cite{EichelsbacherThaele14,HTW,LaPecc0,LaPecc,LPST,ReitznerPeccati,ReitznerSchulteCLT,SchulteDissertation,SchulteKolmogorov}. In the following, we will frequently use the following consequence of the multivariate Mecke equation for Poisson functionals \cite[Theorem 4.4]{LP}. Namely, the expectation $\E\mathscr{U}$ of the Poisson U-statistic $\mathscr U$ is given by
\begin{equation}\label{eq:Mecke}
\E \mathscr{U} = \E \sum_{(x_1,\ldots,x_m)\in\eta_{\neq}^m}h(x_1,\ldots,x_m)= \int_{\mathbb{X}^m}h(x_1,\ldots,x_m)\,\mu^m(d(x_1,\ldots,x_m)).
\end{equation}

In the present paper we need a formula for the centred moments of the Poisson U-statistics $\mathscr{U}$ as well as a bound for the Wasserstein and the Kolmogorov distance of a normalized version of $\mathscr{U}$ and a standard Gaussian random variable. To state such results, we need some more notation. Following {\cite[Chapter 12]{LP}}, for an integer $n\in\N$ we let $\Pi_n$ and $\Pi_n^*$ be the set of partitions and sub-partitions of $[n]:=\{1,\ldots,n\}$, respectively. We recall that by a sub-partition of $\{1,\ldots,n\}$ we understand a family of non-empty disjoint subsets (called blocks) of $\{1,\ldots,n\}$ and that a sub-partition $\sigma$ is called a partition if $\bigcup_{J\in\sigma}J=\{1,\ldots,n\}$. For $\sigma\in\Pi_n^*$ we let $|\sigma|$ be the number of blocks of $\sigma$ and $\|\sigma\|=\big|\bigcup_{J\in\sigma}J\big|$ be the number of elements of $\bigcup_{J\in\sigma}J$. In particular, a partition $\sigma\in\Pi_n$ satisfies $\|\sigma\|=n$. For $\ell\in\N$ and  $n_1,\ldots,n_\ell\in\N$, let $n:=n_1+\ldots+n_\ell$ and define
$$
J_i:=\{j\in\N:n_1+\ldots+n_{i-1}<j\leq n_1+\ldots+n_i\},\qquad i\in\{1,\ldots,\ell\},
$$
and $\pi:=\{J_i:i\in\{1,\ldots,\ell\}\}$. Next, we introduce two classes of sub-partitions of $[n]$ by
\begin{align*}
\Pi^*(n_1,\ldots,n_\ell) &:= \{\sigma\in\Pi_n^*:|J\cap J'|\leq 1\text{ for all }J\in\sigma\text{ and }J'\in\pi\},\\
\Pi_{\geq 2}^*(n_1,\ldots,n_\ell) &:= \{\sigma\in\Pi^*(n_1,\ldots,n_\ell):|J|\geq 2\text{ for all }J\in\sigma\}.
\end{align*}
In the same way the two classes of partitions $\Pi(n_1,\ldots,n_\ell)$ and $\Pi_{\geq 2}(n_1,\ldots,n_\ell)$ of $[n]$ are defined (just by omitting the upper index $\!\!\!\!\phantom{x}^*$ in the above definition). From now on we assume that $n_1=\ldots=n_\ell=m\in\N$ and define, for $\sigma\in\Pi^*(m,\ldots,m)$ (where here and below $m$ appears $\ell$ times),
$$
[\sigma] := \{i\in[\ell]:\text{there exists a block }J\in\sigma\text{ such that }J\cap\{m(i-1)+1,\ldots,mi\}\neq\emptyset\}
$$
as well as
$$
\Pi_{\geq 2}^{**}(m,\ldots,m) := \{\sigma\in\Pi_{\geq 2}^*(m,\ldots,m):[\sigma]=[\ell]\}.
$$
The sub-partitions $\sigma\in\Pi_{\geq 2}^{**}(m,\ldots,m)$ of $[m\ell]$ are easy to visualize as diagrams (cf.\ \cite[Chapter 4]{PTbook}). The $m\ell$ elements of $[m\ell]$ are arranged in an array of $\ell$ rows and $m$ columns, where $1,\ldots,m$ form the first row, $m+1,\ldots,2m$ the second etc. The blocks of $\sigma$ are indicated by closed curves, where the elements enclosed by a curve are meant to belong to the same block. Then the condition that $\sigma\in\Pi_{\geq 2}^{**}(m,\ldots,m)$ can be expressed by the following three requirements:
\begin{itemize}
\item[(i)] all blocks of $\sigma$ have at least two elements,
\item[(ii)] each block of $\sigma$ contains at most one element from each row,
\item[(iii)] in each row there is at least one element that belongs to some block of $\sigma$.
\end{itemize}
For an example and a counterexample we refer to Figure \ref{fig:Diag1}.

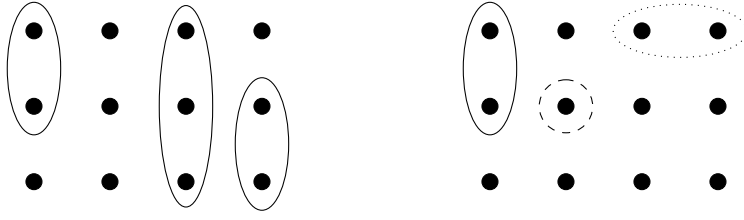
\begin{figure}[t]
\begin{center}
\begin{tikzpicture}
\filldraw
(0,0) circle (3pt)
(1,0) circle (3pt)
(2,0) circle (3pt)
(3,0) circle (3pt)
(0,1) circle (3pt)
(1,1) circle (3pt)
(2,1) circle (3pt)
(3,1) circle (3pt)
(0,2) circle (3pt)
(1,2) circle (3pt)
(2,2) circle (3pt)
(3,2) circle (3pt)
(0+6,0) circle (3pt)
(1+6,0) circle (3pt)
(2+6,0) circle (3pt)
(3+6,0) circle (3pt)
(0+6,1) circle (3pt)
(1+6,1) circle (3pt)
(2+6,1) circle (3pt)
(3+6,1) circle (3pt)
(0+6,2) circle (3pt)
(1+6,2) circle (3pt)
(2+6,2) circle (3pt)
(3+6,2) circle (3pt);

\draw[rotate=90] (1.5,0) ellipse (25pt and 10pt);
\draw[rotate=90] (1,-2) ellipse (38pt and 10pt);
\draw[rotate=90] (0.5,-3) ellipse (25pt and 10pt);

\draw[rotate=90] (1.5,-6) ellipse (25pt and 10pt);
\draw[rotate=90,dashed] (1,-7) ellipse (10pt and 10pt);
\draw[rotate=90,dotted] (2,-8.5) ellipse (10pt and 25pt);

\end{tikzpicture}
\end{center}
\caption{Left panel: Sub-partition from $\Pi_{\geq 2}^{**}(4,4,4)$. Right panel: Example of a sub-partition not belonging to $\Pi_{\geq 2}^{**}(4,4,4)$. In fact, the block indicated by the dashed curve contradicts condition (i), the block indicated by the dotted curve contradicts condition (ii) and since no element from the last row is contained in any block also condition (iii) is violated.}
\label{fig:Diag1}
\end{figure}

For two functions $g_1:\mathbb{X}^{\ell_1}\to\R$ and $g_2:\mathbb{X}^{\ell_2}\to\R$ with $\ell_1,\ell_2\in\N$), we denote by $g_1\otimes g_2:\mathbb{X}^{\ell_1+\ell_2}\to\R$ their usual tensor product. We are now in the position to rephrase the following formula for the centred moments of the Poisson U-statistic $\mathscr{U}$ (see \cite[Proposition 12.13]{LP}):
\begin{align}\label{eq:MomentsUstatistic}
\E[(\mathscr{U}-\E\mathscr{U})^\ell] = \sum_{\sigma\in\Pi_{\geq 2}^{**}(m,\ldots,m)}\int_{\mathbb{X}^{m\ell+|\sigma|-\|\sigma\|}}(h^{\otimes \ell})_\sigma\,d\mu^{m\ell+|\sigma|-\|\sigma\|},
\end{align}
where $h^{\otimes \ell}$ is the $\ell$-fold tensor product of $h$ with itself and $(h^{\otimes \ell})_\sigma:\mathbb{X}^{m\ell+|\sigma|-\|\sigma\|}\to\R$ stands for the function that arises from $h^{\otimes \ell}$ by replacing all variables that are in the same block of $\sigma$ by a new, common variable. Here, we implicitly assume that the function $h$ is such that all integrals that appear on the right-hand side are well-defined. This formula will turn out to be a crucial tool in the proof of Theorem \ref{thm:CLTrToInfinity} (c).

\subsection{Normal approximation bounds}\label{subsec:3.3}

In this section, we continue to use the notation and the set-up of the preceding section. But since we turn to normal approximation bounds for Poisson U-statistics, some further notation is required. For $u,v\in\{1,\ldots,m\}$ we let $\Pi_{\geq 2}^{{\rm con}}(u,u,v,v)$ be the class of partitions in $\Pi_{\geq 2}(u,u,v,v)$ whose diagram is connected, which means that the rows of the diagram cannot be divided into two subsets, each defining a separate diagram (cf.\ \cite[page 47]{PTbook}). More formally, there are no sets $A,B\subset[4]$ with $A\cup B=[4]$, $A\cap B=\emptyset$ and such that each block either consists of elements from rows in $A$ or of elements from rows in $B$, see Figure \ref{fig:Diag2} for an example and a counterexample. We can now introduce the quantities
\begin{equation}\label{eq:defMij}
M_{u,v}(h) := \sum_{\sigma\in\Pi_{\geq 2}^{{\rm con}}(u,u,v,v)} \int_{\mathbb{X}^{|\sigma|}}(h_u\otimes h_u\otimes h_v\otimes h_v)_\sigma\,\ d\mu^{|\sigma|},
\end{equation}
where
\begin{align}\label{eq:ChaosKernels}
h_u(x_1,\ldots,x_u) = {m\choose u}\int_{\mathbb{X}^{m-u}}h(x_1,\ldots,x_u,\tilde{x}_1,\ldots,\tilde{x}_{m-u})\,\mu^{m-u}(d(\tilde{x}_1,\ldots,\tilde{x}_{m-u}))
\end{align}
for $u\in\{1,\ldots,m\}$ (again, we implicitly assume that $h$ is such that the integrals appearing in \eqref{eq:defMij} are well-defined). To measure the distance between two real-valued random variables $X,Y$ (or, more precisely, their laws), the Kolmogorov distance
$$
d_K(X,Y) \defeq \sup_{s \in \mathbb{R}} |\mathbb{P}(X \leq s)- \mathbb{P}(Y \leq s)|
$$
and the Wasserstein distance
$$
d_W(X,Y) \defeq \sup_{\varphi \in \text{Lip}(1)} | \mathbb{E}\varphi(X)-\mathbb{E}\varphi(Y)|
$$
are used, where $\text{Lip}(1)$ denotes the space of Lipschitz functions $\varphi:\R\to\R$ with a Lipschitz constant less than or equal to one. It is well known that convergence with respect to the Kolmogorov or the Wasserstein distance implies convergence in distribution. We are now in the position to rephrase a quantitative central limit theorem for Poisson U-statistics. Namely, {\cite[Theorem 4.7]{ReitznerSchulteCLT}} and {\cite[Therorem 4.2]{SchulteKolmogorov}} state that there exists a constant $c_m\in(0,\infty)$, depending only on $m$ (the order of the Poisson U-statistic), such that
\begin{align}\label{eq:Kolmogorov}
d\left(\frac{\mathscr{U}-\mathbb{E}\mathscr{U}}{\sqrt{\Var(\mathscr{U})}},N \right) \leq c_m \sum_{u,v=1}^{m} \frac{\sqrt{M_{u,v}(h)}}{\Var(\mathscr{U})},
\end{align}
where $d(\,\cdot\,,\,\cdot\,)$ stands for either the Wasserstein or the Kolmogorov distance. Here, one can choose $c_m=2 m^{7/2}$ for the Wasserstein distance and $c_m=19 m^{5}$ for the Kolmogorov distance.

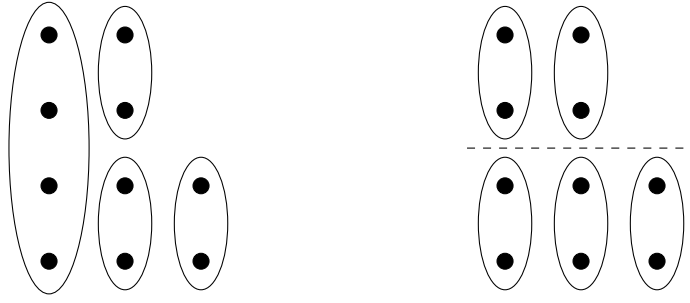
\begin{figure}[t]
\begin{center}
\begin{tikzpicture}
\filldraw
(0,0) circle (3pt)
(1,0) circle (3pt)
(2,0) circle (3pt)
(0,1) circle (3pt)
(1,1) circle (3pt)
(2,1) circle (3pt)
(0,2) circle (3pt)
(1,2) circle (3pt)
(0,3) circle (3pt)
(1,3) circle (3pt)
(0+6,0) circle (3pt)
(1+6,0) circle (3pt)
(2+6,0) circle (3pt)
(0+6,1) circle (3pt)
(1+6,1) circle (3pt)
(2+6,1) circle (3pt)
(0+6,2) circle (3pt)
(1+6,2) circle (3pt)
(0+6,3) circle (3pt)
(1+6,3) circle (3pt);

\draw[rotate=90] (1.5,0) ellipse (55pt and 15pt);
\draw[rotate=90] (2.5,-1) ellipse (25pt and 10pt);
\draw[rotate=90] (0.5,-1) ellipse (25pt and 10pt);
\draw[rotate=90] (0.5,-2) ellipse (25pt and 10pt);

\draw[rotate=90] (2.5,-1-6) ellipse (25pt and 10pt);
\draw[rotate=90] (0.5,-1-6) ellipse (25pt and 10pt);
\draw[rotate=90] (0.5,-2-6) ellipse (25pt and 10pt);
\draw[rotate=90] (2.5,-1-5) ellipse (25pt and 10pt);
\draw[rotate=90] (0.5,-1-5) ellipse (25pt and 10pt);
\draw[dashed] (5.5,1.5)--(8.5,1.5);

\end{tikzpicture}
\end{center}
\caption{Left panel: Partition from $\Pi_{\geq 2}^{{\rm con}}(2,2,3,3)$. Right panel: Example of a partition not belonging to $\Pi_{\geq 2}^{{\rm con}}(2,2,3,3)$. In fact, the diagram is not connected as indicated by the dashed line.}
\label{fig:Diag2}
\end{figure}

Finally, we turn to a multivariate normal approximation for Poisson U-statistics. For integers $p\in\N$ and  $m_1,\ldots,m_p\in\N$, and for each $\oldU \in\{1,\ldots,p\}$, let
$$
\mathscr{U}_\oldU = \sum_{(x_1,\ldots,x_{m_\oldU})\in\eta_{\neq}^{m_\oldU}}h^{(\oldU)}(x_1,\ldots,x_{m_\oldU})
$$
be a Poisson U-statistic of order $m_\oldU$ based on a kernel function $h^{(\oldU)}:\mathbb{X}^{m_\oldU}\to\R$ satisfying the same assumptions as above. We form the $p$-dimensional random vector $\textbf{U}:=(\mathscr{U}_1,\ldots,\mathscr{U}_p)$ and our goal is to compare $\textbf{U}$ with a $p$-dimensional Gaussian random vector $\textbf{N}$. To do this, we use the so-called $d_2$- and $d_3$-distance, which are defined as
\begin{align*}
d_2(\textbf{U},\textbf{N}) &:= \sup_{h\in C^2}\big|\E \varphi(\textbf{U})-\E \varphi(\textbf{N})\big|\\
d_3(\textbf{U},\textbf{N}) &:= \sup_{h\in C^3}\big|\E \varphi(\textbf{U})-\E \varphi(\textbf{N})\big|,
\end{align*}
respectively. Here, $C^2$ is the space of function $\varphi:\R^p\to\R$ which are twice partially continuously differentiable and satisfy
$$
\sup_{x\neq y}{|\varphi(x)-\varphi(y)|\over\|x-y\|}\leq 1\qquad\text{and}\qquad \sup_{x\neq y}{\|\nabla\varphi(x)-\nabla\varphi(y)\|_{\rm op}\over\|x-y\|}\leq 1,
$$
where $\|\,\cdot\,\|$ denotes the Euclidean norm in $\R^p$ and $\|\,\cdot\,\|_{\rm op}$ stands for the operator norm. Moreover, $C^3$ is the space of functions $\varphi:\R^p\to\R$ which are thrice partially continuously differentiable and satisfy
$$
\max_{1\leq i\leq j\leq p}\sup_{x\in\R^p}\Big|{\partial^2 \varphi(x)\over\partial x_i\partial x_j}\Big| \leq 1\qquad\text{and}\qquad \max_{1\leq i\leq j\leq k\leq p}\sup_{x\in\R^p}\Big|{\partial^3 \varphi(x)\over\partial x_i\partial x_j\partial x_k}\Big| \leq 1.
$$
Moreover, similarly to the quantities $M_{u,v}(h)$ introduced in \eqref{eq:defMij}, for $i,j\in\{1,\ldots,p\}$, $u\in\{1,\ldots,m_i\}$ and $v\in\{1,\ldots,m_j\}$ we define
$$
M_{u,v}(h^{(i)},h^{(j)}) := \sum_{\pi\in\Pi_{\geq 2}^{{\rm con}}(u,u,v,v)}\int_{\mathbb{X}^{|\pi|}}(h_u^{(i)}\otimes h_u^{(i)}\otimes h_v^{(j)}\otimes h_v^{(j)})_\pi\,d\mu^{|\pi|},
$$
where $h_u^{(i)}$ and $h_v^{(j)}$ are given by \eqref{eq:ChaosKernels}. This allows us to state the following multivariate normal approximation bound from {\cite[Theorem 6.3]{SchulteDissertation}} (see also \cite[Equation (5.1)]{ReitznerSchulteThaele}). Namely, if $\textbf{N}$ is a centred Gaussian random vector with covariance matrix $\Sigma=(\sigma_{i,j})_{i,j=1}^p$, then
\begin{equation}\label{eq:d_3}
\begin{split}
 d_3(\textbf{U}-\mathbb{E}\textbf{U},\textbf{N}) & \leq \frac{1}{2} \sum_{i,j=1}^{p} |\sigma_{i,j}- \Cov(\mathscr{U}_i,\mathscr{U}_j)| \\
  & \qquad  +\frac{p}{2} \left(\sum_{n=1}^{p} \sqrt{\Var(\mathscr{U}_n)}+1\right) \sum_{i,j=1}^{p} \sum_{u=1}^{m_i} \sum_{v=1}^{m_j} m_i^{7/2} \sqrt{M_{u,v}(h^{(i)},h^{(j)})}.
\end{split}
\end{equation}
If the covariance matrix $\Sigma$ is positive definite then also
\begin{equation}\label{eq:d_2}
\begin{split}
 & d_2(\textbf{U}-\mathbb{E}\textbf{U},\textbf{N}) \leq \|\Sigma^{-1}\|_{\rm op}\|\Sigma\|_{\rm op}^{1/2}\sum_{i,j=1}^{p} |\sigma_{i,j}- \Cov(\mathscr{U}_i,\mathscr{U}_j)| +\\
  & \qquad  +\frac{p\sqrt{2\pi}}{4}\|\Sigma^{-1}\|_{\rm op}^{3/2}\|\Sigma\|_{\rm op}\left(\sum_{\oldU=1}^{p} \sqrt{\Var(\mathscr{U}_\oldU)}+1\right) \sum_{i,j=1}^{p} \sum_{u=1}^{m_i} \sum_{v=1}^{m_j} m_i^{7/2} \sqrt{M_{u,v}(h^{(i)},h^{(j)})},
\end{split}
\end{equation}
where again $\|\,\cdot\,\|_{\rm op}$ stands for the operator norm. We remark that although the bound for $d_2(\textbf{U}-\mathbb{E}\textbf{U},\textbf{N})$ is not explicitly stated in the literature, it directly follows from \cite[Theorem 3.3]{PeccatiZheng} together with the computations in \cite[Chapters 5 and 6]{SchulteDissertation} for the $d_3$-distance.

\section{Proofs I -- Expectations and variances}\label{sec:4}

\subsection{Representation as a Poisson U-statistic}\label{sec:4.1}

We recall that $\eta_t$, for $t>0$, is a Poisson hyperplane process in $\Hd$ with intensity measure $t\mu_{d-1}$. Moreover, for a Borel set $W\subset\Hd$ and $i\in\{0,1,\ldots,d-1\}$ we recall from \eqref{eq:DefFwti} the definition of the functional $F_{W,t}^{(i)}$. To shorten our notation we put
$$
f^{(i)}(H_1,\ldots,H_{d-i}) \defeq \begin{cases}
                            \frac{1}{(d-i)!} \mathcal{H}^{i}(H_1 \cap \ldots \cap H_{d-i} \cap W) &: \text{dim}(H_1 \cap \ldots \cap H_{d-i})=i, \\
                            0 &: \mbox{otherwise},
                          \end{cases}
$$
which allows us to rewrite $F_{W,t}^{(i)}$ as
$$
F_{W,t}^{(i)}= \sum_{(H_1,\ldots,H_{d-i}) \in \eta_{t,\neq}^{d-i}} f^{(i)}(H_1,\ldots,H_{d-i}).
$$
In other words, $F_{W,t}^{(i)}$ is a Poisson U-statistic of order $d-i$ and with kernel $f^{(i)}$. It is well known (see \cite{LP,LaPecc0,LaPecc,LPST,ReitznerSchulteCLT}) that Poisson U-statistics admit a Fock space representation having only a finite number of terms. This leads to the variance and covariance representations
\begin{align}\label{eq:variance_general}
\Var(F_{W,t}^{(i)})&= \sum_{n=1}^{d-i} t^{2(d-i)-n}n!\|f_n^{(i)}\|_n^{2},
\end{align}
where the functions $f_n^{(i)}: A_h(d,d-1)^n \rightarrow [0, \infty)$ are given by
\begin{align*}
f_n^{(i)}(H_1,\dots,H_n)
      \defeq \binom{d-i}{n} \int_{A_h(d,d-1)^{d-i-n}} &f^{(i)}(H_1,\ldots,H_n,\tilde{H}_1,\ldots,\tilde{H}_{d-i-n}) \\ &\qquad\times\mu_{d-1}^{d-i-n}(d(\tilde{H}_1,\ldots,\tilde{H}_{d-i-n})),
\end{align*}
recall \eqref{eq:ChaosKernels}, and we write $\|\, \cdot \,\|_n$ for the norm in the $L^2$-space $L^2(\mu_{d-1}^n)$ with respect to the $n$-fold product measure of $\mu_{d-1}$. Similarly, for $i,j \in \{0,1, \ldots, d-1\}$ the covariance $\Cov(F_{W,t}^{(i)},F_{W,t}^{(j)})$ can be represented as
\begin{align}\label{eq:covariance_general}
\Cov(F_{W,t}^{(i)},F_{W,t}^{(j)})&= \sum_{n=1}^{\min \{d-i,d-j\}} t^{2d-i-j-n}n!\langle f_n^{(i)},f_n^{(j)} \rangle_{n},
\end{align}
where $\langle \, \cdot \, ,  \, \cdot \, \rangle_{n}$ denotes the standard scalar product in $L^2(\mu_{d-1}^n)$.

\subsection{Expectations: Proof of Theorem \ref{thm:Expectation}}\label{sec:Expectations}

Theorem \ref{thm:Expectation} is a consequence of the transformation formula in Lemma \ref{lem:mu_almost_surely_d-n_hyperplane} and the Crofton formula in Lemma \ref{lem:Crofton} with $k=i$ there. In fact, using \eqref{eq:Mecke} we obtain
\begin{align*}
\mathbb{E}F_{W,t}^{(i)} &= t^{d-i} \int_{A_h(d,d-1)^{d-i}} f^{(i)}(H_1,\ldots,H_{d-i}) \ \mu_{d-1}^{d-i}(d(H_1,\ldots,H_{d-i})) \\
&= \frac{t^{d-i}}{(d-i)!} \int_{A_h(d,d-1)^{d-i}} \mathcal{H}^{i}(H_1\cap \ldots \cap H_{d-i} \cap W) \ \mu_{d-1}^{d-i}(d(H_1,\ldots,H_{d-i})) \\
&= c(d,i)\,\frac{t^{d-i}}{(d-i)!} \int_{A_h(d,i)} \mathcal{H}^{i}(E \cap W)\,\mu_i(dE)\\
&=c(d,i)\,\frac{t^{d-i}}{(d-i)!}\, \mathcal{H}^{d}(W)\\
&= \frac{\omega_{i+1}}{\omega_{d+1}}\left(\frac{\omega_{d+1}}{\omega_d}\right)^{d-i} \frac{t^{d-i}}{(d-i)!} \ \mathcal{H}^d(W),
\end{align*}
and the proof is complete.  \hfill $\Box$

\begin{remark}{\rm
The  measure $W\mapsto\mathbb{E}F_{W,t}^{(i)}$ is isometry invariant. One could argue that it must be a constant multiple of  $\mathcal{H}^{d}$,
if one knows that it is also locally finite.  Theorem \ref{thm:Expectation} shows that this is indeed the case and also yields the constant.}
\end{remark}

\subsection{Variances: Proof of Theorem \ref{thm:Variance}}\label{sec:variance}

To investigate the variance of $F_{W,t}^{(i)}$ we use the representation as a Poisson U-statistic, especially \eqref{eq:variance_general}. We start by simplifying the kernel functions $f_n^{(i)}$.

\begin{Lemma}\label{lem:function_f_n}
Let $n\in\{1,\ldots,d-i\}$. Let $W\subset\Hd$ be a bounded Borel set. If $H_1,\ldots,H_n \in A_h(d,d-1)$ are $n$ hyperplanes satisfying $\textup{dim}(H_1 \cap \ldots \cap H_n)=d-n$, then
$$
f_n^{(i)}(H_1,\ldots,H_n)=  c(i,n,d)\, \mathcal{H}^{d-n}(H_1 \cap \ldots \cap H_n \cap W)
$$
with
$$
c(i,n,d):={\binom{d-i}{n}\over (d-i)!}\frac{\omega_{i+1} }{\omega_{d-n+1}} \left( \frac{\omega_{d+1}}{\omega_d} \right)^{d-n-i}.
$$
\end{Lemma}
\begin{proof}
We use the definition of $f_n^{(i)}$, the transformation formula in Lemma \ref{lem:mu_almost_surely_d-n_hyperplane} and the Crofton formula (\ref{eq:Crofton_volume}) (in the general form indicated before the statement of Lemma \ref{lem:Crofton}). Putting $L_{d-n}:=H_1\cap\ldots\cap H_{n}$, this gives
  \begin{align*}
    & \binom{d-i}{n}^{-1}f_n^{(i)}(H_1,\ldots,H_n)   \\
     &=   \frac{1}{(d-i)!} \int_{A_h(d,d-1)^{d-i-n}} \mathcal{H}^{i}(L_{d-n} \cap \tilde{H}_1 \cap \ldots \cap \tilde{H}_{d-i-n} \cap W) \ \mu_{d-1}^{d-i-n}(d(\tilde{H}_1,\ldots,\tilde{H}_{d-i-n})) \\
     &=\frac{c(d,i+n)}{(d-i)!} \int_{A_h(d,i+n)}\cH^i(L_{d-n}\cap W \cap E)\,\mu_{i+n}(dE)\\
     & =   \frac{c(d,i+n)}{(d-i)!} \frac{\omega_{d+1} \, \omega_{i+1}}{\omega_{i+n+1} \, \omega_{d-n+1}}\mathcal{H}^{d-n}(L_{d-n} \cap W)  \\
     & =  {1\over (d-i)!}\frac{\omega_{i+1} }{\omega_{d-n+1}} \left( \frac{\omega_{d+1}}{\omega_d} \right)^{d-n-i} \mathcal{H}^{d-n}(H_1 \cap \ldots \cap H_n \cap W).
  \end{align*}
  Here we used that since $L_{d-n}$ is $(d-n)$-dimensional, the intersection $L_{d-n}\cap W$ is Hausdorff $(d-n)$-rectifiable.
\end{proof}

For the variances and covariances, we need the $L^{2}$-norms and the scalar products of these functions.

\begin{Corollary}\label{cor:VarianceExact}
Let $W\subset\Hd$ be a bounded Borel set. If $n\in\{1,\ldots,\min\{d-i,d-j\}\}$, then
$$
\langle f_n^{(i)},f_n^{(j)} \rangle_n = c(d,n,i,j) \int_{A_h(d,d-n)}\cH^{d-n}(E\cap W)^{2}\,\mu_{d-n}(dE).
$$
Especially, the choice $W=B_r$ for some $r>0$ yields
$$
\langle f_n^{(i)},f_n^{(j)}  \rangle_n = c(d,n,i,j)  \, \omega_{n} \int_{0}^{r} \cosh^{d-n}(s) \sinh^{n-1}(s) \,\mathcal{H}^{d-n}(L_{d-n}(s) \cap B_r)^2 \ ds,
$$
where $ c(d,n,i,j):=  c(d,d-n)\, c(i,n,d)\, c(j,n,d)$ and $L_{d-n}(s)$ for $s\in[0,r]$ is an arbitrary $(d-n)$-dimensional totally geodesic subspace which satisfies $d_h(L_{d-n}(s),p)=s$.
\end{Corollary}

\begin{proof}
The first claim is a direct consequence of the previous lemma and the transformation formula from Lemma \ref{lem:mu_almost_surely_d-n_hyperplane}.

The second claim follows by  combining  the previous result with \eqref{eq:Croftonmass} and using geodesic spherical coordinates in the $(d-n)$-dimensional planes $L_{d-n}$ through $p$ (see \cite[Proposition 3.1 and Equation (3.22)]{Chavel}).
\end{proof}

\begin{proof}[Proof of Theorem \ref{thm:Variance}]
This is now a direct consequence of \eqref{eq:variance_general} and Corollary \ref{cor:VarianceExact}.
\end{proof}

\subsection{Variance: Asymptotic behaviour}

In this section we look at the variance of $F_{r,t}^{(i)}=F_{B_r,t}^{(i)}$ in the asymptotic regime, as $r\to\infty$. We divide our analysis into the three different cases $d=2$, $d=3$ and $d\geq 4$. Before, we start with a number of preprations.

\subsubsection{Preliminaries}

The following lemma will be repeatedly applied below.

\begin{Lemma}\label{lem:CoshBound}
	If $r>0$ and $ 0\le s \le r $, then
	$$
	0\leq\arcosh\left (\frac{\cosh(r)}{\cosh(s)} \right)-(r-s)\leq \log(2).
	$$
\end{Lemma}
\begin{proof}
	We start by proving the lower bound which is equivalent to $\cosh(r)-\cosh(s)\cosh(r-s) \geq 0$. By definition of $\cosh$, $\sinh$ and since $0 \leq s \leq r$ we have
$$\cosh(r)-\cosh(s)\cosh(r-s)=\sinh(s)\sinh(r-s) \geq 0.$$
This yields the lower bound.
	Next, we turn to the upper bound. We use the logarithmic representation $\arcosh(x)=\log(x+\sqrt{x^{2}-1})$ of the $\arcosh$-function and the fact that $\cosh(r)/\cosh(s) \geq 1$ for $r\ge s\ge 0$. Then we get
	\begin{align*}
	\arcosh\left (\frac{\cosh(r)}{\cosh(s)} \right)-(r-s) & = \log \left(\frac{\cosh(r)}{\cosh(s)}+\sqrt{\frac{\cosh^2(r)}{\cosh^2(s)}-1} \right)-(r-s) \\
	&=\log \left(\frac{e^s \cosh(r)}{e^r \cosh(s)}+\sqrt{\frac{e^{2s} \cosh^2(r)}{e^{2r} \cosh^2(s)}-\frac{e^{2s}}{e^{2r}}} \right) \\
	&=\log \left( \frac{e^s (e^r+e^{-r})}{e^r (e^s+e^{-s})}+\sqrt{\frac{e^{2s} (e^r+e^{-r})^2}{e^{2r} (e^s+e^{-s})^2}-\frac{e^{2s}}{e^{2r}}} \right)\\
	&=\log \left(\frac{1+e^{-2r}}{1+e^{-2s}}+\sqrt{\frac{e^{2s} (e^{2r}+2+e^{-2r}-e^{2s}-2-e^{-2s})}{e^{2r} (e^{2s}+2+e^{-2s})}} \right) \\
	&=\log \left(\frac{1+e^{-2r}}{1+e^{-2s}}+\sqrt{\frac{1+e^{-4r}-e^{2s-2r}-e^{-4s}}{1+2e^{-2s}+e^{-2s-2r}}} \right)\\
	& \leq \log(2),
	\end{align*}
	where the last inequality holds because both terms in the argument of the $\log$ function are bounded from above by $1$ for $s \in [0,r]$.
\end{proof}

Moreover, we frequently apply the following upper and lower bounds for $\mathcal{H}^{i}(L_i(s)\cap B_r)$. As before, let  $L_{i}(s)$ denote an arbitrary  measurable choice of an $i$-dimensional totally geodesic subspace satisfying $d_h(L_{i}(s),p)=s$, $i\in\{1,\ldots,d-1\}$. The following lemma concerns the case
$i\in\{2,\ldots,d-1\}$.

\begin{Lemma}\label{lem:H_d_bounds}
If $i \in\{2,\ldots,d-1\}$ and $0 \leq s \leq r$, then
$$
\mathcal{H}^{i}(L_i(s)\cap B_r)\leq \frac{\omega_i}{i-1}  e^{(r-s)(i-1)}.
$$
If, in addition, $0\le s \leq r-1/2$ then
$$
\frac{\omega_i}{e^{3(i-1)}(i-1)}e^{(r-s)(i-1)} \leq \mathcal{H}^{i}(L_i(s)\cap B_r).
$$
\end{Lemma}
\begin{proof}
We start by noting that $L_{i}(s) \cap B_r$ is an $i$-dimensional hyperbolic ball of radius $\arcosh \big(\frac{\cosh(r)}{\cosh(s)} \big)$ for $i\in \{1,\ldots,d-1\}$, see {\cite[Theorem 3.5.3]{Ratcliffe}}. Thus we get
\begin{align}\label{eq:volume_ball}
 \mathcal{H}^{i}(L_i(s)\cap B_r) & = \omega_i \int_{0}^{\arcosh \left( \frac{\cosh(r)}{\cosh(s)} \right)} \sinh^{i-1}(u) \ du
\end{align}
for $i\in \{1,\ldots,d-1\}$. Hence, using Lemma \ref{lem:CoshBound} and for $i \in\{2,\ldots,d-1\}$ we get
  \begin{align*}
    \mathcal{H}^{i}(L_i(s)\cap B_r)      & \leq \omega_i \int_{0}^{r-s+\log(2)} \sinh^{i-1}(u) \ du \\
     & \leq \frac{\omega_i}{2^{i-1}} \int_{0}^{r-s+\log(2)} e^{u(i-1)} \ du \\
     & \leq \frac{\omega_i}{2^{i-1}(i-1)}  2^{i-1}e^{(r-s)(i-1)} \\
     & = \frac{\omega_i}{i-1} e^{(r-s)(i-1)}.
  \end{align*}
  On the other hand, Lemma \ref{lem:CoshBound} and Lemma \ref{lem:inequalities} imply that
  \begin{align*}
    \mathcal{H}^{i}(L_i(s)\cap B_r)     & \geq \omega_i \int_{0}^{r-s} \sinh^{i-1}(u) \ du \\
    & =\omega_i \left( \int_{1/2}^{r-s} \sinh^{i-1}(u) \ du +\int_{0}^{1/2} \sinh^{i-1}(u) \ du \right) \\
    & \geq \omega_i \left( \int_{1/2}^{r-s} \left(\frac{e^{u}}{e^3} \right)^{i-1} \ du +\int_{0}^{1/2} u^{i-1} \ du \right) \\
    & = \frac{\omega_i}{e^{3(i-1)}(i-1)} \left(e^{(r-s)(i-1)}-e^{(i-1)/2} \right)+ \frac{1}{2^{i}} \, \frac{\omega_i}{i } \\
    & \geq \frac{\omega_i}{e^{3(i-1)}(i-1)} e^{(r-s)(i-1)},
  \end{align*}
  where we used that $s \leq r-1/2$ to obtain the equality in the third line. The last inequality is due to
  \begin{align*}
    \frac{1}{2^{i}} \, \frac{\omega_i}{i }- \frac{\omega_i}{e^{(5/2)(i-1)}(i-1)}
    &  =  \omega_i \left(\frac{1}{i \, 2^{i}} - \frac{1}{e^{(5/2)(i-1)}(i-1)} \right) \geq 0.
  \end{align*}
  The positivity of the last term holds for $i \geq 2$, since $2^{i+1}  \leq e^{(5/2)(i-1)}$ implies that
  $$2^{i}  \leq \frac{i-1}{i} \, e^{(5/2)(i-1)},$$
  which is equivalent to the desired inequality.
\end{proof}

We will need later the following lemma.

\begin{Lemma}\label{lem:lines_intersecting_ball_inequality}
Let $r\geq 1 $. For $k\in\{0,1,\ldots,d-1\}$ and $ 0 \leq s \leq r$, let $L_k(s) \in A_h(d,k)$ be a $k$-dimensional totally geodesic subspace such that $d_h(L_k(s),p)=s$. Then for any $ l \in \mathbb{N}$ there exist constants $c,C>0$,  depending only on $k, l$ and $d$, such that
\begin{align*}
c \,g(k,l,d,r)&\leq \omega_{d-k} \int_{0}^{r} \cosh^{k}(s) \sinh^{d-1-k}(s) \, \mathcal{H}^{k}(L_k(s) \cap B_r)^{l}\, ds\\
&\qquad \qquad = \int_{A_h(d,k)}\mathcal{H}^k (H\cap B_r)^l\, \mu_k(dH)
 \leq C \, g(k,l,d,r)
 \end{align*}
 with
 $$g(k,l,d,r)= \begin{cases}
              \exp(r(d-1)) &: l(k-1)<d-1, \\
                 r \,  \exp(r(d-1)) &: l(k-1)=d-1, \\
                 \exp(r \, l(k-1)) &: l(k-1)>d-1.
               \end{cases}$$
\end{Lemma}
\begin{proof}
The asserted equality of the two integral expressions is clear from the argument for the second claim in Corollary \ref{cor:VarianceExact}.

For $k=0$ the integral is just the volume of a geodesic ball of radius $r$ which can be bounded from above and below
by a positive constant times $\exp(r(d-1))$.

In the following, we repeatedly use that the intersection $L_k(s) \cap B_r$ is a $k$-dimensional hyperbolic ball of radius $\arcosh(\cosh(r)/\cosh(s))$. The constants $c$ and $C$ used in the calculations below only depend on $k,l,d$ and may vary from line to line. Suppose that $k\ge 2$. The substitution $u=r-s$ and an application of Lemma \ref{lem:CoshBound} yield
\begin{align*}
    &\int_{0}^{r} \cosh^{k}(s) \sinh^{d-1-k}(s) \, \mathcal{H}^{k}(L_k(s) \cap B_r)^{l} \, ds\\
    &\qquad=  \int_{0}^{r} \cosh^{k}(r-u)\, \sinh^{d-1-k}(r-u) \, \mathcal{H}^{k}(L_k(r-u) \cap B_r)^{l} \, du  \\
    &\qquad=  \int_{0}^{r} \cosh^{k}(r-u)\,\sinh^{d-1-k}(r-u) \, \left( \omega_k \int_{0}^{\arcosh \left(\frac{\cosh(r)}{\cosh(r-u)}\right)} \sinh^{k-1}(s) \ ds \right)^{l} \, du  \\
    &\qquad \leq  C \int_{0}^{r} e^{k(r-u)}\,e^{(d-1-k)(r-u)} \, \left( 2^{-(k-1)}\int_{0}^{u+\log(2)} e^{(k-1)s} \ ds \right)^{l} \, du  \\
    &\qquad \leq  C \int_{0}^{r} e^{(d-1)(r-u)} \, \left(\frac{1}{k-1} e^{(u +\log(2))(k-1)}\right)^{l} \, du  \\
    &\qquad \leq  C e^{r(d-1)}\int_{0}^{r} e^{u(l(k-1)-(d-1))} \, du \\
    &\qquad \leq  C g(k,l,d,r).
\end{align*}
To obtain the lower bound, we first show for $u \geq 0.2$ that
  \begin{align*}
    \int_{0}^{u} \sinh^{k-1}(s) &\geq \int_{0.1}^{u} \sinh^{k-1}(s) \ ds  \geq \int_{0.1}^{u} e^{(k-1)(s-3)} \ ds \\
    & \geq \frac{e^{-3(k-1)}}{k-1}\left(e^{(k-1)u}-e^{0.1(k-1)}\right)\\
    & \geq \frac{e^{0.1(k-1)}}{k-1} \, e^{-3(k-1)}\left(e^{(k-1)(u-0.1)}-1\right)\\
    & \geq \frac{e^{0.1(k-1)}}{k-1} \, e^{-3(k-1)} \frac{1}{20}e^{(k-1)(u-0.1)}.
  \end{align*}
Now we  substitute again $u=r-s$. An application of Lemma \ref{lem:inequalities} and the lower bound from Lemma \ref{lem:CoshBound} then yield
  \begin{align*}
    &\int_{0}^{r} \cosh^{k}(s) \sinh^{d-1-k}(s) \, \mathcal{H}^{k}(L_k(s) \cap B_r)^{l} \, ds\\
    &\qquad=  \int_{0}^{r} \cosh^{k}(r-u)\, \sinh^{d-1-k}(r-u) \, \mathcal{H}^{k}(L_k(r-u) \cap B_r)^{l} \, du  \\
    &\qquad=  \int_{0}^{r} \cosh^{k}(r-u)\,\sinh^{d-1-k}(r-u) \, \left( \omega_k \int_{0}^{\arcosh \left(\frac{\cosh(r)}{\cosh(r-u)}\right)} \sinh^{k-1}(s) \ ds \right)^{l} \, du  \\
    &\qquad \geq  c\int_{0}^{r-0.1} e^{k(r-u)}\,e^{(d-1-k)(r-u-3)} \, \left( \int_{0}^{u} \sinh^{k-1}(s) \ ds \right)^{l} \, du  \allowdisplaybreaks\\
    &\qquad \geq  c e^{r(d-1)}\int_{0.2}^{r-0.1} e^{-u(d-1)} \, e^{l(k-1)(u-0.1)} \, du  \\
    &\qquad =  c e^{r(d-1)}\int_{0.2}^{r-0.1} e^{u(l(k-1)-(d-1))} \, du  \\
    &\qquad \geq  c g(k,l,d,r).
  \end{align*}
For $k=1$, the proof is almost the same except that we simply use that $\int_0^a\sinh^{k-1}(s)\, ds=a$ for $a\ge 0$.
\end{proof}

\subsubsection{The planar case $d=2$}

Although we present a very detailed covariance analysis in Section \ref{sec:4.5} we will separately investigate the asymptotic behaviour of the variances in Lemmas \ref{lem:VarianceBoundd=2} -- \ref{lem:VarianceBoundd>=4}. In fact while the results of Section \ref{sec:4.5} are necessary for the multivariate central limit theorems, the variance analysis we carry out here is already sufficient for the unvariate cases. In this and the following two sections, $c_i$ will denote a positive constant only depending on the dimension and a counting parameter $i \in \N_0$. If it additionally depends on another parameter $n \in \N_0$, we indicate this by writing, for instance, $c_{i,n}$ or $c_i(n)$.  The value of this constant may change from occasion to occasion.

\begin{Lemma}\label{lem:VarianceBoundd=2}
Let $d=2$, $i\in\{0,1\}$ and $t\ge t_0>0$. Then there are constants $c^{(i)}(2,t_0), C^{(i)}(2,t_0)\in(0,\infty)$ such that for all $r\geq 1$,
\begin{align*}
  c^{(i)}(2,t_0)\,t^{3-2i}\,e^{r} \leq \Var(F_{r,t}^{(i)}) \leq C^{(i)}(2,t_0)\,t^{3-2i} \,e^{r}.
\end{align*}
\end{Lemma}
\begin{proof}
For $i\in\{0,1\}$ and $n=1$, Corollary \ref{cor:VarianceExact} and  Lemma \ref{lem:lines_intersecting_ball_inequality} yield
  \begin{align*}
 c_i\, e^r\le \|f_1^{(i)}\|_1^{2}= c_i \int_{0}^{r} \cosh(s)\, \mathcal{H}^{1}(L_{1}(s) \cap B_r)^2 \ ds \leq C_i \, e^{r}.
  \end{align*}
Similarly, for $i=0$ and $n=2$ we obtain
  \begin{align*}
  \|f_2^{(0)}\|_2^{2}&= c_0 \int_{0}^{r} \sinh(s) \mathcal{H}^{0}(L_{1}(s) \cap B_r)^2 \ ds
 = c_0  \int_{0}^{r} \sinh(s) \ ds
 = c_0 \big(\cosh(r)-1\big).
  \end{align*}
From (\ref{eq:variance_general}) we now deduce that
$$
c(t^2+t^3)e^r\le c_1t^3e^r+c_2t^2e^r\le \Var(F_{r,t}^{(0)})\le c_1t^3e^r+c_2t^2e^r\le C(t^2+t^3)e^r.
$$
Using that $t\ge t_0>0$, the assertion follows for $i=0$. The case $i=1$ is obtained in the same way, but requires
bounds for only one summand in (\ref{eq:variance_general}).
\end{proof}

\subsubsection{The spatial case $d=3$}

\begin{Lemma}\label{lem:var_d=3}\label{lem:VarianceBoundd=3}
Let $d=3$, $i\in\{0,1,2\}$ and $t\ge t_0>0$. Then there are constants $c^{(i)}(3,t_0),C^{(i)}(3,t_0)\in(0,\infty)$ such that for all $r\geq 1$,
\begin{align*}
  c^{(i)}(3,t_0)\,t^{5-2i} \,re^{2r} \leq \Var(F_{r,t}^{(i)}) \leq C^{(i)}(3,t_0)\,t^{5-2i} \,re^{2r}.
\end{align*}
\end{Lemma}
\begin{proof}
Corollary \ref{cor:VarianceExact} and Lemma \ref{lem:lines_intersecting_ball_inequality} imply the upper bound
  \begin{align*}
   \Var(F_{r,t}^{(i)})-\sum_{n=2}^{3-i} t^{6-2i-n} n! \|f_n^{(i)}\|_n^{2}
    &=  t^{5-2i} \|f_1^{(i)}\|_1^{2} \le c_i \, t^{5-2i}  \, re^{2r}.
  \end{align*}
It remains to determine the asymptotic behaviour in $r$ of the terms $\|f_{2}^{(i)}\|_{2}^2$ and $\|f_{3}^{(i)}\|_{3}^2$.  Corollary \ref{cor:VarianceExact} and  Lemma \ref{lem:lines_intersecting_ball_inequality} yield
$$
 c_i\,  e^{2r}\le    \|f_{2}^{(i)}\|_{2}^2\leq C_i \, e^{2r}\qquad\text{and}\qquad c_i\,  e^{2r}\le\|f_{3}^{(i)}\|_{3}^2
      \leq C_i \, e^{2r}.
$$
To deduce the lower bound, it is sufficient to derive a lower bound for the term $\|f_{1}^{(i)}\|_{1}^2$. But
$$
   \Var(F_{r,t}^{(i)})\ge t^{5-2i} \|f_1^{(i)}\|_1^{2}\ge c_i \, t^{5-2i}  \, re^{2r}.
$$
This completes the proof.
\end{proof}

\subsubsection{The higher dimensional case $d\geq 4$}

\begin{Lemma}\label{lem:VarianceBoundd>=4}
Let $d\geq 4$, $i\in\{0,1,\ldots,d-1\}$, and $t\ge t_0>0$. Then there are positive constants $c^{(i)}(d,t_0),C^{(i)}(d,t_0)\in(0,\infty)$ such that for all $r\geq 1$,
\begin{align*}
  c^{(i)}(d,t_0)\, t^{2(d-i)-1} \, e^{2r(d-2)} \leq \Var(F_{r,t}^{(i)}) \leq C^{(i)}(d,t_0)\,t^{2(d-i)-1} \,e^{2r(d-2)}.
\end{align*}
\end{Lemma}
\begin{proof}
Combining Corollary \ref{cor:VarianceExact} with Lemma  \ref{lem:lines_intersecting_ball_inequality}, we obtain
 $$
   \Var(F_{r,t}^{(i)})- \sum_{n=d-1}^{d-i} t^{2(d-i)-n} n! \|f_n^{(i)}\|_n^{2}
  \le \sum_{n=1}^{\min \{d-2, \ d-i\} } c_{i,n} \, t^{2(d-i)-n}\, g(d-n,2,d,r).
 $$
For $n=1\le \min \{d-2, \ d-i\}$, we have $g(d-1,2,d,r)\le C_i\,\exp(r2(d-2))$, since $2(d-2)-(d-1)=d-3>0$.
If $2(d-n-1)-(d-1)=d-1-2n>0$, then $g(d-n,2,d,r)\le g(d-1,2,d,r)$. For the remaining cases, we use that
  $\exp(r(d-1))$   is of lower than $\exp(2r(d-2))$ for $d\ge 4$. Moreover, as in the case $d=3$ it follows that $\|f_{d-1}^{(i)}\|_{d-1}^2$ and $\|f_{d}^{(i)}\|_{d}^2$ are of order at most $e^{r(d-1)}$. This yields the upper bound.

  The lower bound is again derived by just taking into account $\|f_{1}^{(i)}\|_{1}^2$ and by applying the lower bound   $g(d-1,2,d,r)\ge c_i\, \exp(r2(d-2))$ from Lemma  \ref{lem:lines_intersecting_ball_inequality}.
\end{proof}

\subsection{Covariance analysis}\label{sec:4.5}

In this section we prepare the proof of Theorem \ref{thm:CLTMultivariate} by an asymptotic analysis of the covariance structure of the random vector ${\textbf{F}}_{r,t}$ in dimensions $d=2$ and $d=3$.

\subsubsection{The planar case $d=2$}

The following lemma describes the rate of convergence, as $r\to\infty$, of the suitably scaled covariances to the asymptotic covariance matrix $\Sigma_d=(\sigma^{i,j}_d)_{i,j=0}^{d-1}$ for $d=2$.

\begin{Lemma}\label{lem:dist_cov_d=2}
Let $d=2$ and $t\ge t_0>0$. There is a positive constant $c_{\rm cov}(2,t_0)\in(0,\infty)$ such that if $r\geq 1$, then
$$
\Bigg|\sigma^{i,j}_2-\Cov\Bigg({F_{r,t}^{(i)}-\E F_{r,t}^{(i)}\over e^{r/2}},{F_{r,t}^{(j)}-\E F_{r,t}^{(j)}\over e^{r/2}}\Bigg)\Bigg| \leq c_{\rm cov}(2,t_0)\, t^{3-i-j} r^{2}e^{-r},\qquad i,j\in\{0,1\}.
$$
Moreover,
\begin{align}\label{eq:Sigma_2}
 \Sigma_2= \begin{pmatrix}
    t^2\left(\left(\frac{4}{\pi}\right)^2ta+\frac{1}{\pi}\right) & \frac{8}{\pi}\,t^2a \\
    \frac{8}{\pi}\,t^2a & 4ta
  \end{pmatrix}
\end{align}
and $a=4\cdot\text{\rm G}$ with  Catalan's constant  $\text{\rm G}\approx 0.915965594$. In particular, $\Sigma_2$ is positive definite with $\det(\Sigma_2)=\frac{4}{\pi}t^3a$.
\end{Lemma}
\begin{proof}
Since ${\textbf{F}}_{r,t}$ is a vector of Poisson U-statistics the covariance representation \eqref{eq:covariance_general} shows that, for $i,j\in\{0,1\}$,
$$
\Cov\Bigg({F_{r,t}^{(i)}-\E F_{r,t}^{(i)}\over e^{r/2}},{F_{r,t}^{(j)}-\E F_{r,t}^{(j)}\over e^{r/2}}\Bigg) = e^{-r}\sum_{n=1}^{\min \{2-i,2-j\}} t^{4-i-j-n}n!\langle f_n^{(i)},f_n^{(j)} \rangle_{n}
$$
and it remains to compute the scalar products. Using \eqref{eq:volume_ball} and Corollary \ref{cor:VarianceExact} we get
\begin{align*}
\langle {f}_1^{(i)},{f}_1^{(j)} \rangle_{1}
    &= c(2,1,i,j)\cdot 2\cdot 4 \int_{0}^{r} \cosh(s) \, \arcosh^{2} \left(\frac{\cosh(r)}{\cosh(s)}\right) \, ds \\
    &= c(2,1,i,j)\cdot 2\cdot 4 \int_{0}^{r} \cosh(r-s) \, \arcosh^{2} \left(\frac{\cosh(r)}{\cosh(r-s)}\right)  \, ds\\
    &= c_1^{(i,j)} \int_{0}^{r} (e^{r-s}+e^{s-r}) \arcosh^{2} \left(e^{s}\left(\frac{1+e^{-2r}}{1+e^{2(s-r)}}\right)\right)\,ds
  \end{align*}
with $c_1^{(i,j)}=4\cdot c(i,1,2)c(j,1,2)$. We have $c(0,1,2)=2/\pi$ and $c(1,1,2)=1$, and hence
$$
c_1^{(0,0)}=4\left(\frac{2}{\pi}\right)^2=\left(\frac{4}{\pi}\right)^2,\quad  c_1^{(1,1)}=4,\quad c_1^{(1,0)}=c_1^{(0,1)}=4\cdot \frac{2}{\pi}=\frac{8}{\pi}.
$$
Furthermore, again by Corollary \ref{cor:VarianceExact}
\begin{align*}
\langle {f}_2^{(i)},{f}_2^{(j)} \rangle_{2}
    &= c(2,2,i,j)\cdot 2 \int_{0}^{r} \sinh(s)\, ds
    = c_2^{(i,j)} (e^r+e^{-r}-2)
  \end{align*}
with $c_2^{(i,j)}=(2/\pi)c(i,2,2)c(j,2,2)$. In particular, $c_2^{(0,0)}=1/(2\pi)$.

In the following, we use that
\begin{equation}\label{absch}
\arcosh \left(e^{s}\left(\frac{1+e^{-2r}}{1+e^{2(s-r)}}\right)\right)\le \arcosh \left(e^{s}\right)\le s+\log(2).
\end{equation}
For $(i,j) \in \{(0,1),(1,0),(1,1)\}$ we then deduce from the dominated convergence theorem that
  \begin{align*}
    \sigma^{i,j}_2 & = \lim_{r \rightarrow \infty} c_1^{(i,j)}t^{3-i-j} \int_{0}^{r} ({e^{-s}+e^{-2r+s}})\arcosh^{2} \left(e^{s}\left(\frac{1+e^{-2r}}{1+e^{2(s-r)}}\right)\right) \ ds\\
    &=c_1^{(i,j)} t^{3-i-j}\int_{0}^{\infty} e^{-s} \arcosh^{2}(e^{s}) \ ds=:c_1^{(i,j)} t^{3-i-j}\cdot a
  \end{align*}
  and, in addition we have
  \begin{align*}
    \sigma^{0,0}_2 & =  c_1^{(0,0)} t^3\cdot a+2t^2\,c_2^{(0,0)}.
  \end{align*}
Since $a=4\cdot \text{G}$ by the following Remark \ref{remarkn7},
we obtain the specific values of $\sigma^{i,j}_2$ for $i,j\in\{0,1\}$,
and hence the determinant of the asymptotic covariance matrix $\Sigma_2$ given in \eqref{eq:Sigma_2}.

Next we prove the asserted rates of convergence. For $(i,j) \in \{(0,1),(1,0),(1,1)\}$, we
   get
  \begin{align}
    & \Bigg|\sigma^{i,j}_2-\Cov\Bigg({F_{r,t}^{(i)}-\E F_{r,t}^{(i)}\over e^{r/2}},{F_{r,t}^{(j)}-\E F_{r,t}^{(j)}\over e^{r/2}}\Bigg)\Bigg|\nonumber \\
     &\qquad =   \left |c_1^{(i,j)}t^{3-i-j}\cdot a-c_1^{(i,j)}t^{3-i-j} \int_{0}^{r} (e^{-s}+e^{-2r+s}) \arcosh^{2} \left(e^{s}\left(\frac{1+e^{-2r}}{1+e^{2(s-r)}}\right)\right) \ ds \right| \nonumber  \\
     &\qquad \leq  c_1^{(i,j)}t^{3-i-j}\int_{0}^{r}  e^{-s} \left(\arcosh^{2}(e^{s})- \arcosh^{2} \left(e^{s}\left(\frac{1+e^{-2r}}{1+e^{2(s-r)}}\right)\right) \right)  \ ds  \label{Daa}\\
     &\qquad\qquad + c_1^{(i,j)}t^{3-i-j}\int_{0}^{r}  e^{-2r+s} \arcosh^{2} \left(e^{s}\left(\frac{1+e^{-2r}}{1+e^{2(s-r)}}\right)\right)  \ ds  \label{Dab}\\
     &\qquad\qquad + c_1^{(i,j)} t^{3-i-j}\int_{r}^{\infty} e^{-s} \arcosh^{2}(e^{s}) \ ds .\label{Dac}
  \end{align}
  Applying \eqref{absch} to  the expression in \eqref{Dac} we get
  \begin{align}\label{eq:dist_cov_3}
    \int_{r}^{\infty} {e^{-s}} \arcosh^{2}(e^{s}) \ ds
      \leq  \int_{r}^{\infty} e^{-s} (\log(2)+s)^{2} \ ds \leq c \,r^{2} e^{-r}.
  \end{align}
 Using \eqref{absch} for the expression in \eqref{Dab} we obtain
  \begin{align}\label{eq:dist_cov_2}
     \int_{0}^{r}  e^{-2r+s} \arcosh^{2} \left(e^{s}\left(\frac{1+e^{-2r}}{1+e^{2(s-r)}}\right)\right)  \, ds
    &\leq  \int_{0}^{r}  e^{-2r+s} \arcosh^{2} \left(e^{s}\right)  \, ds  \nonumber\\
    &\leq  \int_{0}^{r}  e^{-2r+s} (s+\log(2))^{2}  \, ds  \nonumber\\
    &\leq  c\,r^{2}e^{-r}.
  \end{align}
  Finally, we treat the expression in \eqref{Daa}. An application of the mean value theorem in the first and \eqref{absch} in the second to last step yields
  \begin{align}\label{eq:dist_cov_1}
     &\int_{0}^{r}  e^{-s} \left(\arcosh^{2}(e^{s})- \arcosh^{2} \left(e^{s}\left(\frac{1+e^{-2r}}{1+e^{2(s-r)}}\right)\right) \right)  \ ds \nonumber\allowdisplaybreaks\\
    &\qquad\leq \int_{0}^{r}  e^{-s} \left( \left(e^{s}- e^{s}\left(\frac{1+e^{-2r}}{1+e^{2(s-r)}}\right) \right) \, \underset{z \in \left[e^{s}\left(\frac{1+e^{-2r}}{1+e^{2(s-r)}}\right),e^{s}\right]}\max \frac{d}{d z} (\arcosh^{2}(z)) \right)  \ ds  \nonumber\allowdisplaybreaks\\
    &\qquad\leq \int_{0}^{r}   \left(\frac{e^{2(s-r)}-e^{-2r}}{1+e^{2(s-r)}} \right) \frac{2 \arcosh(e^{s})}{\sqrt{\left(\frac{e^{s}+e^{-2r+s}}{1+e^{2(s-r)}} \right)^{2}-1}}  \ ds  \nonumber\allowdisplaybreaks\\
    &\qquad=  \int_{0}^{r}    e^{-2r}\left(e^{2s}-1 \right) \frac{2 \arcosh(e^{s})}{\sqrt{e^{2s}-1+e^{2(s-2r)}-e^{-4(r-s)}}}  \ ds \nonumber \allowdisplaybreaks\\
    &\qquad\leq \frac{1}{\sqrt{1-e^{-2r}}}\int_{0}^{r}    e^{-2r}\left(e^{2s}-1 \right) \frac{2 \arcosh(e^{s})}{\sqrt{e^{2s}-1}}   \ ds  \nonumber\allowdisplaybreaks\\
    &\qquad\leq c e^{-2r}\int_{0}^{r}    \sqrt{e^{2s}-1} \,\arcosh(e^{s})   \ ds \nonumber \allowdisplaybreaks\\
    &\qquad\leq c e^{-2r}\int_{0}^{r}    e^{s} (s+\log(2))   \ ds  \nonumber\allowdisplaybreaks\\
    &\qquad\leq c\,re^{-r}.
  \end{align}
  Thus, a combination of \eqref{eq:dist_cov_3}, \eqref{eq:dist_cov_2} and \eqref{eq:dist_cov_1} yields the result for $(i,j)\in\{(0,1),(1,0),(1,1)\}$. Finally, if $(i,j)=(0,0)$ we obtain the result by additionally taking into account that
  $$|c_2^{(0,0)} (1+e^{-2r}-2e^{-r})-c_2^{(0,0)}| \leq c \,e^{-r}.$$
  This completes the proof.
\end{proof}

\begin{remark}\label{remarkn7}
The relation $a=4\text{G}$ can be confirmed by Maple. It is not clear to us how Maple verifies this relation.
Since we could not find the current integral representation of the Catalan constant in one of the lists available to us, we provide
a short derivation. We first use
the substitution $t=\exp(-\arcosh(e^s))$ or $e^s=\frac{1}{2}(t^{-1}+t)$ and then expand $(1+t^2)^{-2}$ into a Taylor series under the integral sign. This leads to
\begin{align*}
a&=\int_{0}^{\infty} e^{-s} \arcosh^{2}(e^{s})=\int_0^1\frac{1-t^2}{(1+t^2)^2}(\ln t)^2\, dt=2\int_0^1\sum_{i=0}^\infty (-1)^i(i+1)t^{2i}(1-t^2)(\ln t)^2\, dt.
\end{align*}
By the substitution $t=e^y$ we obtain
$$
\int_0^1 t^{2i}(\ln t)^2\, dt=\frac{2}{(2i+1)^3}.
$$
Hence we can interchange summation and integration to get
\begin{align*}
a&=4\left(\sum_{i=0}^\infty (-1)^i(i+1)\frac{1}{(2i+1)^3}-\sum_{i=0}^\infty (-1)^i(i+1)\frac{1}{(2i+3)^3}\right)\\
&=4\left(\frac{1}{2}\text{G}+\frac{1}{2}\sum_{i=0}^\infty (-1)^i \frac{1}{(2i+1)^3}-\frac{1}{2}(-\text{G}+1)+
\frac{1}{2}\sum_{i=0}^\infty (-1)^i \frac{1}{(2i+3)^3}\right)\\
&=4\left(\frac{1}{2}\text{G}+\frac{1}{2}\text{G} +\frac{1}{2}-\frac{1}{2}\right)=4\text{G}.
\end{align*}
\end{remark}

\subsubsection{The spatial case $d=3$}

Now we turn to the case $d=3$ and again describes the rate of convergence, as $r\to\infty$, of the suitably scaled covariances to the asymptotic covariance matrix $\Sigma_d=(\sigma^{i,j}_d)_{i,j=0}^{d-1}$.

\begin{Lemma}\label{lem:dist_cov_d=3}
Let $d=3$ and $t\ge t_0>0$.  There exists a positive constant $c_{\rm cov}(3,t_0)\in(0,\infty)$ such that
$$
\Bigg|\sigma^{i,j}_3-\Cov\Bigg({F_{r,t}^{(i)}-\E F_{r,t}^{(i)}\over \sqrt{r}\,e^{r}},{F_{r,t}^{(j)}-\E F_{r,t}^{(j)}\over \sqrt{r}\,e^{r}}\Bigg)\Bigg| \leq c_{\rm cov}(3,t_0)\, t^{5-i-j} \, r^{-1}, \qquad i,j\in\{0,1,2\},$$
for $r\geq 1$. The matrix $\Sigma_3$ has rank one and is explicitly given by
\begin{equation}\label{eq:Sigma_3}
  \Sigma_3= 2 \pi^2\, \begin{pmatrix}
                \frac{\pi^{2}}{2^8}t^{5} & \frac{\pi^{2}}{2^6}t^{4} & \frac{\pi}{2^4}t^{3} \\
                \frac{\pi^{2}}{2^6}t^{4} & \frac{\pi^{2}}{2^4}t^{3} & \frac{\pi}{2^2}t^{2} \\
                \frac{\pi}{2^4}t^{3} & \frac{\pi}{2^2}t^{2} & t
              \end{pmatrix}.
  \end{equation}
\end{Lemma}

\begin{proof}
For  $i,j\in\{0,1,2\}$, the covariance formula for Poisson U-statistics yields that
$$
\Cov\Bigg({F_{r,t}^{(i)}-\E F_{r,t}^{(i)}\over \sqrt{r}\,e^{r}},{F_{r,t}^{(j)}-\E F_{r,t}^{(j)}\over \sqrt{r}\,e^{r}}\Bigg) = r^{-1}e^{-2r}\sum_{n=1}^{\min \{3-i,3-j\}} t^{6-i-j-n}n!\langle f_n^{(i)},f_n^{(j)} \rangle_{n}.
$$
As in the planar case $d=2$ we compute the scalar products. We let $L_2(s)$ be a $2$-dimensional subspace in $\mathbb{H}^3$ having distance $s\geq 0$ from the origin $p$. For $n=1$ Corollary \ref{cor:VarianceExact} and Equation \eqref{eq:volume_ball} yield
\begin{align*}
\langle {f}_1^{(i)},{f}_1^{(j)} \rangle_{1}
    &= \omega_1 c(3,1,i,j)  \int_{0}^{r} \cosh^{2}(s) \,\mathcal{H}^{2}(L_2(s)\cap B_r)^{2} \ ds  \\
    &= \omega_2^{2} \omega_1  c(3,1,i,j) \int_{0}^{r} \cosh^{2}(s) \left( \int_{0}^{\arcosh \left(\frac{\cosh(r)}{\cosh(s)}\right)}\sinh(u) \ du \right)^{2}  \ ds  \\
    &= \omega_2^{2} \omega_1 c(3,1,i,j) \int_{0}^{r} \cosh^{2}(s) \left( \frac{\cosh(r)}{\cosh(s)}-1 \right)^{2} \ ds \\
    &= \omega_2^{2} \omega_1 c(3,1,i,j)  \int_{0}^{r} \ \left( \cosh(r)-\cosh(s)\right)^{2} \ ds \\
    &= {\omega_2^{2} \omega_1 c(3,1,i,j)}\frac{1}{2} \,\big(r+2r\cosh^2(r)-3\sinh(r)\cosh(r)\big).
  \end{align*}
In addition, using Lemma \ref{lem:mu_almost_surely_d-n_hyperplane} and Lemma \ref{lem:lines_intersecting_ball_inequality}, we obtain
\begin{align}
\langle {f}_2^{(i)},{f}_2^{(j)} \rangle_{2} &\leq c \,e^{2r} \qquad\text{and}\qquad \langle {f}_3^{(i)},{f}_3^{(j)} \rangle_{3} \leq c \,e^{2r}. \notag
\end{align}
Since $c(3,2)=1$, $c(0,1,3)=\pi/16$, $c(1,1,3)=\pi/4$ and $c(2,1,3)=1$, we obtain $c(3,1,0,0)=\pi^2/2^8$, $c(3,1,0,1)=\pi^2/2^6$, $c(3,1,0,2)=\pi/2^4$, $c(3,1,1,1)=\pi^2/2^4$, $c(3,1,1,2)=\pi/2^2$ and $c(3,1,2,2)=1$.
Moreover, we have
  \begin{align*}
    \sigma^{i,j}_3 & = \lim_{r \rightarrow \infty} t^{5-i-j}\,{\omega_2^{2} \omega_1 c(3,1,i,j)}\frac{1}{2} \,r^{-1}e^{-2r} \,\big(r+2r\cosh^2(r)-3\sinh(r)\cosh(r)\big)\\
    &= t^{5-i-j}\, {\omega_2^{2} \omega_1 c(3,1,i,j)}\frac{1}{4} \\
    &=t^{5-i-j}\, 2 \pi^{2} c(3,1,i,j).
  \end{align*}
  Therefore we conclude that the asymptotic covariance matrix $\Sigma_3$ is given by \eqref{eq:Sigma_3}.
  Clearly, $\Sigma_3$ has rank one. Moreover, we obtain
  \begin{align*}
    & \Bigg|\sigma^{i,j}_3-\Cov\Bigg({F_{r,t}^{(i)}-\E F_{r,t}^{(i)}\over \sqrt{r}\,e^{r}},{F_{r,t}^{(j)}-\E F_{r,t}^{(j)}\over \sqrt{r}\,e^{r}}\Bigg)\Bigg| \\
      &\qquad \leq t^{5-i-j}\, 4 \pi^{2}c(3,1,i,j) \left| 1/2-r^{-1}e^{-2r}\big(r+2r\cosh^2(r)-3\sinh(r)\cosh(r)\big)\right| \\
      &\hspace{4cm}+ r^{-1}e^{-2r}\sum_{n=2}^{\min\{3-i,3-j \}} t^{6-i-j-n}\,n! \langle {f}_n^{(i)},{f}_n^{(j)} \rangle_{n} \\
     &\qquad \leq  c_{\rm cov}(3,t_0)\,t^{5-i-j}\,  r^{-1},
  \end{align*}
  where we used that $|1/2-r^{-1}e^{-2r}\big(r+2r\cosh^2(r)-3\sinh(r)\cosh(r)\big)|$ is bounded from above by a constant multiple of $r^{-1}$ as $r\to\infty$. This completes the proof.
\end{proof}

\subsubsection{The case $d\geq 4$}

In order to describe explicitly the limit covariance matrix $\Sigma(d)$ for $d \geq 4$  we  need the following lemma.

\begin{Lemma}\label{lem:integral_calculation}
For $\alpha >0$,
$$\int_{0}^{\infty} \cosh^{-\alpha}(x) \ dx=\frac{\sqrt{\pi}}{2} \frac{\Gamma(\frac{\alpha}{2})}{\Gamma(\frac{\alpha+1}{2})}.$$
\end{Lemma}
\begin{proof}
  Substituting first $u=e^{x}$ and then $\tan(z)=u$, and using $(\tan^{2}(x)+1)^{-1}=\cos^{2}(x)$, we get
  \begin{align*}
    \int_{0}^{\infty} \cosh^{-\alpha}(x) \ dx
    = 2^{\alpha} \int_{1}^{\infty} \frac{u^{\alpha-1}}{(u^{2}+1)^{\alpha}}  \ du
    = 2^{\alpha} \int_{\pi/4}^{\pi/2} \sin^{{\alpha-1}}(z) \cos^{\alpha-1}(z)  \ dz=:I_\alpha.
  \end{align*}
The trigonometric identity $2\sin\alpha \cos\alpha=\sin(2\alpha)$ and the substitution $y=2z$ yield
  \begin{align*}
    I_\alpha
    =2 \int_{\pi/4}^{\pi/2} \sin^{{\alpha-1}}(2z) \ dz
    = \int_{0}^{\pi/2} \sin^{{\alpha-1}}(y) \ dy
    =\frac{\sqrt{\pi}}{2}\frac{\Gamma(\frac \alpha 2)}{\Gamma(\frac{\alpha+1}{2})}.
  \end{align*}
  This completes the argument.
\end{proof}

Depending on the dimension, we will bound the speed of convergence by means of the function
$$
h(d,r)=\begin{cases}
          e^{-r} &: d=4,\\
re^{-2r} &: d=5,\\
e^{-2r} &: d\geq 6.
         \end{cases}
$$

\begin{Lemma}\label{lem:dist_cov_d>=3}
Let $d\geq 4$ and $t\ge t_0>0$.  There exists a positive constant $c_{\rm cov}(d,t_0)\in(0,\infty)$ such that
$$
\Bigg|\sigma^{i,j}_d-\Cov\Bigg({F_{r,t}^{(i)}-\E F_{r,t}^{(i)}\over e^{r(d-2)}},{F_{r,t}^{(j)}-\E F_{r,t}^{(j)}\over e^{r(d-2)}}\Bigg)\Bigg| \leq c_{\rm cov}(d,t_0)\, t^{2d-1-i-j} \, h(d,r), \qquad i,j\in\{0,\ldots
,d-1\},$$
for $r\geq 1$.
The matrix $\Sigma_d$ has rank one and its entries are explicitly given by
\begin{align}\label{eq:Sigma_4}
  \sigma^{i,j}_d=t^{2d-1-i-j} \, c(i,1,d) \,c(j,1,d) \, \frac{\omega_{d-1}\omega_d }{ 4^{d-2}(d-3)(d-2) }, \qquad i,j \in \{0,\ldots,d-1\},
\end{align}
where the constants $c(i,1,d), c(j,1,d)$ are introduced in Lemma \ref{lem:function_f_n}.
\end{Lemma}
\begin{proof}
Recall that
\begin{equation}\label{eq:290919}
\Cov\Bigg({F_{r,t}^{(i)}-\E F_{r,t}^{(i)}\over e^{(d-2)r}},{F_{r,t}^{(j)}-\E F_{r,t}^{(j)}\over e^{(d-2)r}}\Bigg)=e^{-2(d-2)r} \sum_{n=1}^{\min\{d-i,d-j\}} t^{2d-i-j-n}n! \langle f_n^{(i)},f_{n}^{(j)} \rangle_n
\end{equation}
for $r\ge 1$.
In a first step, we bound from above the summands with $n\ge 2$. For this, let $n \in \{2, \ldots , \min\{d-i,d-j\}\}$.  Lemma \ref{lem:lines_intersecting_ball_inequality} implies that
\begin{align*}
  e^{-2(d-2)r} \,\langle f_n^{(i)},f_{n}^{(j)} \rangle_n \leq c  \, e^{-2(d-2)r} \, g(d-n,2,d,r)
\end{align*}
with some constant $c$, not depending on $r$. For $2(d-n-1) < d-1$ we obtain from Lemma \ref{lem:lines_intersecting_ball_inequality} that
$$c \,  e^{-2(d-2)r}  g(d-n,2,d,r) \leq c \, e^{r(-2d+4)}\, e^{r(d-1)}\le c  \, e^{r(-d+3)}\le c\, h(d,r).$$
Note that  $2(d-n-1) = d-1$ implies that $d$ is odd, hence $d\ge 5$, and therefore
$$c \,  e^{-2(d-2)r}  g(d-n,2,d,r) \leq c \, e^{r(-2d+4)}\, r \, e^{r(d-1)}\le c \, r \, e^{r(-d+3)}\le c\, h(d,r).$$
For $2(d-n-1)>d-1$ we get
$$c \,  e^{-2(d-2)r} g(d-n,2,d,r) \leq c \, e^{r(-2d+4)} \, e^{2r(d-n-1)}\le c \, e^{r(-2n+2)}\le c\, h(d,r),$$
since $n\geq 2$.

Now we examine the remaining term corresponding to $n=1$ in \eqref{eq:290919}. By Corollary \ref{cor:VarianceExact} and \eqref{eq:volume_ball} we get
\begin{equation}\label{eq:dist_cov_d>=3_1}
\begin{split}
  &e^{-2(d-2)r} \, \langle f_1^{(i)},f_{1}^{(j)} \rangle_1 \\
  & \quad = \frac{c(d,1,i,j) \, \omega_1}{e^{2(d-2)r}} \int_{0}^{r} \cosh^{d-1}(s) \, \mathcal{H}^{d-1}(L_{d-1}(s)\cap B_r)^{2} \ ds  \\
  & \quad = \frac{c(d,1,i,j) \, \omega_1}{e^{2(d-2)r}} \int_{0}^{r} \cosh^{d-1}(s) \, \left(\omega_{d-1} \int_{0}^{\arcosh\left(\frac{\cosh(r)}{\cosh(s)} \right)} \sinh^{d-2}(u) \ du \right)^{2} \ ds  \\
  & \quad = \frac{c(d,1,i,j) \, \omega_1 \, \omega_{d-1}^{2}}{e^{2(d-2)r}}  \int_{0}^{r} \cosh^{d-1}(s) \\
  &\qquad\qquad\qquad\qquad\times\left( \int_{0}^{\arcosh\left(\frac{\cosh(r)}{\cosh(s)} \right)} \sum_{k=0}^{d-2} \frac{(-1)^{k}}{2^{d-2}} \, \binom{d-2}{k} \, e^{u(d-2-2k)} \ du \right)^{2} \ ds  \\
  & \quad = \frac{c(d,1,i,j) \, \omega_1 \, \omega_{d-1}^{2}}{4^{d-2}e^{2(d-2)r}}  \int_{0}^{r} \cosh^{d-1}(s) \\
  &\qquad\qquad\qquad\qquad\times \left(\sum_{k=0}^{d-2} (-1)^{k} \, \binom{d-2}{k} \, \int_{0}^{\arcosh\left(\frac{\cosh(r)}{\cosh(s)} \right)}  e^{u(d-2-2k)} \ du \right)^{2} \ ds.
\end{split}
\end{equation}
The quadratic term in brackets in \eqref{eq:dist_cov_d>=3_1} is given by
\begin{align*}
  \sum_{(k_1,k_2) \in \{0, \ldots,d-2\}^{2}}(-1)^{k_1+k_2} \, \binom{d-2}{k_1} \, \binom{d-2}{k_2} \,  &\int_{0}^{\arcosh\left(\frac{\cosh(r)}{\cosh(s)} \right)}  e^{u_1(d-2-2 k_1)} \ du_1 \\
& \times\int_{0}^{\arcosh\left(\frac{\cosh(r)}{\cosh(s)} \right)}  e^{u_2(d-2-2k_2)} \ du_2.
\end{align*}
Next, we provide and upper bound for the summands obtained for  $(k_1,k_2) \in \{0, \ldots,d-2\}^{2} \setminus\{(0,0)\} $.
Without loss of generality we assume $k_2 \geq 1$. Then we get
\begin{align}\label{eq:dist_cov_d>=3_2}
  &e^{-2(d-2)r}  \int_{0}^{r} \cosh^{d-1}(s) \,
  \int_{0}^{\arcosh\left(\frac{\cosh(r)}{\cosh(s)} \right)}  e^{u_1(d-2-2k_1)} \ du_1
  \int_{0}^{\arcosh\left(\frac{\cosh(r)}{\cosh(s)} \right)}  e^{u_2(d-2-2k_2)} \ du_2  \ ds\\
  & \qquad \leq c \, e^{-2r(d-2)}  \int_{0}^{r} e^{s(d-1)} \,
  \int_{0}^{r-s+\log(2)}  e^{u_1(d-2-2k_1)} \ du_1
  \int_{0}^{r-s+\log(2)}  e^{u_2(d-2-2k_2)} \ du_2 \ ds \notag \\
  & \qquad \leq c \, e^{-2r(d-2)}  \int_{0}^{r} e^{s(d-1)} \,
  e^{(r-s)(d-2)} \, e^{(r-s)(d-4)} \ ds \notag \\
  & \qquad \leq c \, e^{-2r}  \int_{0}^{r} e^{s(-d+5)} \ ds\le c\, h(d,r). \notag
\end{align}
for $d\ge 5$. For $d=4$ the third line is
$$c \, e^{-4r}  \int_{0}^{r} e^{3s} \, e^{2(r-s)} (r-s+\log(2)) \ ds
   = c  \, e^{-2r}  \int_{0}^{r} (r-s+\log(2)) \, e^{s} \ ds\le c\, h(4,r).$$
Therefore we can concentrate on the summand corresponding to $k=0$ in \eqref{eq:dist_cov_d>=3_1}.
In the following we will make use of the logarithmic representation $\arcosh(x)=\log(x+\sqrt{x^{2}-1})$ of the $\arcosh$-function in order to evaluate the inner integral. Then we get
\begin{align}
&   \cosh^{-2(d-2)}(r)  \int_{0}^{r} \cosh^{d-1}(s) \, \left( \int_{0}^{\arcosh\left(\frac{\cosh(r)}{\cosh(s)} \right)}  e^{u(d-2)} \ du \right)^{2} \ ds \notag \\
  & \qquad = \frac{\cosh^{-2(d-2)}(r)}{(d-2)^{2}}   \int_{0}^{r} \cosh^{d-1}(s) \, \left( \left(\frac{\cosh(r)}{\cosh(s)}+\sqrt{\frac{\cosh^{2}(r)}{\cosh^{2}(s)}-1} \right)^{d-2}-1 \right)^{2} \ ds \label{eq:dist_cov_d>=3_3}\\
  & \qquad = (d-2)^{-2} \, \int_{0}^{r} \cosh^{-(d-3)}(s) \, \left( \left(1+\sqrt{1-\frac{\cosh^{2}(s)}{\cosh^{2}(r)}} \right)^{d-2}-\left(\frac{\cosh(s)}{\cosh(r)} \right)^{d-2} \right)^{2} \ ds. \notag
\end{align}
For $r\to\infty$ this expression converges to a constant. To get the correct rate stated in the lemma we observe that
\begin{align*}
    & \Bigg|\sigma^{i,j}_d-\Cov\Bigg({F_{r,t}^{(i)}-\E F_{r,t}^{(i)}\over e^{r(d-2)}},{F_{r,t}^{(j)}-\E F_{r,t}^{(j)}\over e^{r(d-2)}}\Bigg)\Bigg| \\
      &\qquad \leq \left|\sigma^{i,j}_d-e^{-2(d-2)r} \, t^{2d-1-i-j}\langle f_1^{(i)},f_{1}^{(j)} \rangle_1\right|
      +e^{-2(d-2)r} \sum_{n=2}^{\min\{d-i,d-j\}} t^{2d-i-j-n}n! \langle f_n^{(i)},f_{n}^{(j)} \rangle_n.
  \end{align*}
  We have already shown that the second summand satisfies the asserted upper bound. It follows from \eqref{eq:dist_cov_d>=3_2}  that it remains to consider
  \begin{align}
    & \left|\sigma^{i,j}_d- \frac{\beta}{e^{2r(d-2)}}  \int_{0}^{r} \cosh^{d-1}(s) \, \left( \int_{0}^{\arcosh\left(\frac{\cosh(r)}{\cosh(s)} \right)}  e^{u(d-2)} \ du \right)^{2} \ ds\right|\notag \\
    & \qquad\leq \left|\sigma^{i,j}_d- \frac{\beta}{4^{d-2}\cosh^{2(d-2)}(r)}  \int_{0}^{r} \cosh^{d-1}(s) \, \left( \int_{0}^{\arcosh\left(\frac{\cosh(r)}{\cosh(s)} \right)}  e^{u(d-2)} \ du \right)^{2} \ ds\right|  \label{eq:dist_cov_d>=3_4}\\
    & \qquad \qquad  + \beta
    \left|\frac{1}{4^{d-2} \, \cosh^{2(d-2)}(r)}- \frac{1}{e^{2r(d-2)}} \right|
    \int_{0}^{r} \cosh^{d-1}(s) \, \left( \int_{0}^{\arcosh\left(\frac{ \cosh(r)}{\cosh(s)} \right)}  e^{u(d-2)} \ du \right)^{2} \ ds, \notag
  \end{align}
  where we set
$$\beta \defeq t^{2d-1-i-j} \, \frac{c(d,1,i,j) \, \omega_1 \, \omega_{d-1}^{2}}{4^{d-2}}.$$
For the second summand, observe that
  \begin{align*}
    \left|\frac{1}{4^{d-2} \, \cosh^{2(d-2)}(r)}- \frac{1}{e^{2r(d-2)}} \right|
    \leq e^{-2r(d-2)} \left(1-(1+e^{-2r})^{-2(d-2)} \right)
    \leq c \, e^{-2r(d-1)}.
  \end{align*}
Since by \eqref{eq:dist_cov_d>=3_3} the integral in the second summand of \eqref{eq:dist_cov_d>=3_4} is of the order $e^{2r(d-2)}$,
the second summand is at most of the order $\beta\, e^{-2r}$.

It remains to show the decay of the first summand in \eqref{eq:dist_cov_d>=3_4}. This is done by using the same steps as in \eqref{eq:dist_cov_d>=3_3} and by splitting up the limit covariance $\sigma^{i,j}_d$. Lemma \ref{lem:integral_calculation} and basic calculus show that the asserted entries of the asymptotic covariance matrix can be written in the form
$$
   \sigma^{i,j}_d
   = \frac{\beta}{(d-2)^{2}} \int_{0}^{\infty} \cosh^{-(d-3)}(s) \ ds.
$$
Then we get
 $$
   \left|\sigma^{i,j}_d-  \frac{\beta}{4^{d-2}\cosh^{2(d-2)}(r)}  \int_{0}^{r} \cosh^{d-1}(s) \, \left( \int_{0}^{\arcosh\left(\frac{\cosh(r)}{\cosh(s)} \right)}  e^{u(d-2)} \ du \right)^{2} \ ds\right|\le I_1+I_2,
 $$
where
$$
I_1:=  \frac{\beta}{(d-2)^{2}} \int_{r}^{\infty} \cosh^{-(d-3)}(s) \ ds \le c\, \beta\, e^{-(d-3)r}.
$$
and
\begin{align*}
I_2:&=
    \frac{\beta}{(d-2)^{2} \, 4^{d-2}} \int_{0}^{r} \cosh^{-(d-3)}(s) \\
    &\qquad\qquad  \times \,\left | 2^{2(d-2)}- \left( \left(1+\sqrt{1-\frac{\cosh^{2}(s)}{\cosh^{2}(r)}} \right)^{d-2}-\left(\frac{\cosh(s)}{\cosh(r)} \right)^{d-2} \right)^{2}\right| \ ds.
\end{align*}
It remains to provide an upper bound for $I_2$. For this we expand the square and use the triangle inequality to get
$I_2\le I_3+I_4$, where
  \begin{align*}
   I_3&\le \beta \int_{0}^{r} \cosh^{-(d-3)}(s) \,\left( 2 \,  \left(1+\sqrt{1-\frac{\cosh^{2}(s)}{\cosh^{2}(r)}} \right)^{d-2} \left(\frac{\cosh(s)}{\cosh(r)} \right)^{d-2}+ \left(\frac{\cosh(s)}{\cosh(r)} \right)^{2(d-2)}\right) \ ds \\
  & \leq  c \,\beta\, \int_{0}^{r} e^{-s(d-3)}\,\left( e^{(s-r)(d-2)}+ e^{(s-r)(2d-4)}\right) \ ds \\
  & \leq  c \,\beta\, e^{-r(d-2)} \int_{0}^{r} e^{s} \ ds
   +c \,\beta\, e^{-r(2d-4)} \int_{0}^{r} e^{s(d-1)} \ ds \leq c \,\beta \, h(d,r),
  \end{align*}
  with some constant $c$. Here we also used that
\begin{equation}\label{uppercb}
\frac{\cosh(s)}{\cosh(r)}= \frac{e^{s}+e^{-s}}{e^{r}+e^{-r}} \leq \frac{2 \,e^{s}}{e^{r}}=2 \, e^{s-r},\qquad  0 \leq s \leq r.
\end{equation}
In order to provide an upper bound for $I_4$, we use the mean value theorem and the inequality $1-\sqrt{1-x} \leq x$, for $ x \in [0,1]$, to get
$$
\left|2^{2(d-2)}-\left(1+\sqrt{1-x}\right)^{2(d-2)}\right|\le 2(d-2)2^{2d-5}x.
$$
Hence we obtain
\begin{align*}
I_4   &\le \beta\int_{0}^{r}\cosh^{-(d-3)}(s)\,\left|2^{2(d-2)}-
\left(1+\sqrt{1-\frac{\cosh^{2}(s)}{\cosh^{2}(r)}} \right)^{2(d-2)}\right|\, ds\\
&\le c\, \beta\, \int_0^r \cosh^{-(d-3)}(s)\frac{\cosh^{2}(s)}{\cosh^{2}(r)}\, ds
  \leq c \,\beta  e^{-2r} \, \int_{0}^{r} e^{s(-d+5)} \ ds \leq c \,\beta\, h(d,r),
  \end{align*}
where also \eqref{uppercb} was used.  This concludes the proof.
\end{proof}

\section{Proofs II -- Mixed K-function and mixed pair-correlation function}\label{sec:ProofKg}

Let $r>0$, $ i, j \in \{0,...,d-1\}$ and let $B\subset\Hd$ be measurable with $\cH^d(B)=1$. Then
\begin{align*}
K_{ij}(r)
& = {1\over\lambda_i \lambda_j}\,\E\int_{{\rm skel}_i}\int_{{\rm skel}_j\cap B}\1\{0<d_h(x,y)\leq r\}\,\cH^j(dy)\,\cH^i(dx).
\end{align*}
Already at this point we see that the condition $0<d_h(x,y)$ can be omitted if $i \ge 1$ or $j\ge 1$.
Requiring that $x\in{\rm skel}_i$ and $y\in{\rm skel}_j$ means that there exist
$$
(H_1,\ldots,H_{d-i})\in\eta_{t,\neq}^{d-i}\qquad\text{and}\qquad(G_1,\ldots,G_{d-j})\in\eta_{t,\neq}^{d-j}
$$
such that $x\in H_1\cap\ldots\cap H_{d-i}$ and $y\in G_1\cap\ldots\cap G_{d-j}$. However, some of the hyperplanes of the first $(d-i)$-tuple may coincide with some of the hyperplanes of the second $(d-j)$-tuple. We denote by $n\in\{0,1,\ldots,d-i\}$ the number of common hyperplanes. Then we obtain the representation
\begin{align*}
K_{ij}(r) &= {1\over\lambda_i \lambda_j}\,\sum_{n=0}^{\min\{d-i,d-j\} }\alpha(d,i,j,n)\,\E\sum_{(H_1,\ldots,H_{d-i},G_1,\ldots,G_{d-j-n})\in\eta_{t,\neq}^{2d-i-j-n}}\int_{H_1\cap\ldots\cap H_{d-i}}\\
&\hspace{4cm}\times\int_{H_1\cap\ldots\cap H_n\cap G_1\cap\ldots G_{d-j-n}\cap B}\1\{0<d_h(x,y)\leq r\}\,\cH^j(dy)\,\cH^i(dx)
\end{align*}
with the combinatorial coefficient given by
$$
\alpha(d,i,j,n) ={1\over n!(d-i-n)!(d-j-n)!} .
$$
Note that if $n=0$ we interpret the second integral as an integral over the set $G_1\cap\ldots\cap G_{d-j}\cap B$ and if $n=d-j$ we understand that the integral ranges over $H_1\cap\ldots\cap H_{d-j}\cap B$. Moreover,  if $i=j=0$, then the summand obtained for $n=d$
is zero, since almost surely $x,y\in H_1\cap \ldots\cap H_d$ and $d_h(x,y)>0$ cannot be satisfied simultaneously. Hence the summation can
be restricted to $n\le m(d,i,j)$ in the following.

An application of \eqref{eq:Mecke} leads to
\begin{align*}
&K_{ij}(r) \\
&= {1\over\lambda_i \lambda_j}\,\sum_{n=0}^{m(d,i,j)}\alpha(d,i,j,n)t^{2d-i-j-n}\int_{A_h(d,d-1)^{2d-i-j-n}}\int_{H_1\cap\ldots\cap H_{d-i}}\int_{H_1\cap\ldots\cap H_n\cap G_1\cap\ldots G_{d-j-n}\cap B}\\
&\qquad\times\1\{0<d_h(x,y)\leq r\}\,\cH^j(dy)\,\cH^i(dx)\,\mu_{d-1}^{2d-i-j-n}(d(H_1,\ldots,H_{d-i},G_1,\ldots,G_{d-j-n}))\\
&={1\over\lambda_i \lambda_j}\,\sum_{n=0}^{m(d,i,j)}\alpha(d,i,j,n)t^{2d-i-j-n}\int_{A_h(d,d-1)^{n}}\int_{A_h(d,d-1)^{d-i-n}}\int_{A_h(d,d-1)^{d-j-n}}\\
&\qquad\times\int_{H_1\cap\ldots\cap H_{d-i}}\int_{H_1\cap\ldots\cap H_n\cap G_{1}\cap\ldots G_{d-j-n}\cap B}\1\{0<d_h(x,y)\leq r\}\,\cH^j(dy)\,\cH^i(dx)\\
&\qquad\times\mu_{d-1}^{d-j-n}(d(G_{1},\ldots,G_{d-j-n}))\,\mu_{d-1}^{d-i-n}(d(H_{n+1},\ldots,H_{d-i}))\,\mu_{d-1}^{n}(d(H_1,\ldots,H_n)),
\end{align*}
where we have used Fubini's theorem to split the integration over $A_h(d,d-1)^{2d-i-j-n}$ in the form $A_h(d,d-1)^{n}\times A_h(d,d-1)^{d-i-n}\times A_h(d,d-1)^{d-j-n}$. The first group of hyperplanes comprises the $n$ common hyperplanes $H_1,\ldots,H_n$, while the second and the third group is associated with the $(d-i-n)$-tuple $H_{n+1},\ldots,H_{d-i}$ and the $(d-j-n)$-tuple $G_{1},\ldots,G_{d-j-n}$, respectively.
We now apply Lemma \ref{lem:mu_almost_surely_d-n_hyperplane} successively to each of the three outer integrals. Together with Fubini's theorem this gives
\begin{align*}
K_{ij}(r) &= {1\over\lambda_i \lambda_j}\,\sum_{n=0}^{m(d,i,j)}\alpha(d,i,j,n)\beta(d,i,j,n) t^{2d-i-j-n}\int_{A_h(d,d-n)}\int_{A_h(d,i+n)}\int_{A_h(d,j+n)}\\
&\qquad\times\int_{E\cap F}\int_{B\cap E\cap G}\1\{0<d_h(x,y)\leq r\}\,\cH^j(dy)\,\cH^i(dx)\,\mu_{j+n}(dG)\,\mu_{i+n}(dF)\,\mu_{d-n}(dE)\\
&= {1\over\lambda_i \lambda_j}\,\sum_{n=0}^{m(d,i,j)}\alpha(d,i,j,n)\beta(d,i,j,n) t^{2d-i-j-n}\int_{A_h(d,d-n)}\int_{A_h(d,j+n)}\int_{B\cap E\cap G}\\
&\qquad\times\int_{A_h(d,i+n)}\int_{E\cap F}\1\{0<d_h(x,y)\leq r\}\,\cH^i(dx)\,\mu_{i+n}(dF)\,\cH^j(dy)\,\mu_{j+n}(dG)\,\mu_{d-n}(dE),
\end{align*}
where $\beta(d,i,j,n) \defeq c(d,d-n)c(d,i+n)c(d,j+n)$.

For the two innermost integrals we get
\begin{align*}
&\int_{A_h(d,i+n)}\int_{E\cap F}\1\{0<d_h(x,y)\leq r\}\,\cH^i(dx)\,\mu_{i+n}(dF)\\
&\qquad =\int_{A_h(d,i+n)}\cH^i(\{x\in E\cap F:0<d_h(x,y)\leq r\})\,\mu_{i+n}(dF)\\
&\qquad =\int_{A_h(d,i+n)}\cH^i(E\cap (B(y,r)\setminus\{y\})\cap F)\,\mu_{i+n}(dF).
\end{align*}
Since $y\in E$, the intersection $E\cap (B(y,r)\setminus\{y\})$ has dimension $d-n$ and we can apply Crofton's formula to conclude that
\begin{align*}
\int_{A_h(d,i+n)}\int_{E\cap F}\1\{0<d_h(x,y)\leq r\}\,\cH^i(dx)\,\mu_{i+n}(dF)={\omega_{d+1}\omega_{i+1}\over\omega_{i+n+1}\omega_{d-n+1}}\cH^{d-n}(E\cap B(y,r)).
\end{align*}
Here we also used that $\cH^{d-n}(E\cap (B(y,r)\setminus\{y\}))=\cH^{d-n}(E\cap B(y,r))$, since $d-n\ge 1$.
Moreover, since $y\in E$ the value of $\cH^{d-n}(E\cap B(y,r))$ is independent of the choice of $E$ and $y$, and is given by the $(d-n)$-dimensional Hausdorff measure
$$
\cH^{d-n}(B_r^{d-n}) = \omega_{d-n}\int_0^r\sinh^{d-n-1}(s)\,ds
$$
of a $(d-n)$-dimensional geodesic ball $B_r^{d-n}$ of radius $r$. We thus arrive at
\begin{align*}
K_{ij}(r) &= {1\over\lambda_i \lambda_j}\,\sum_{n=0}^{m(d,i,j)}\alpha(d,i,j,n)\beta(d,i,j,n){\omega_{d+1}\omega_{i+1}\over\omega_{i+n+1}\omega_{d-n+1}}\cH^{d-n}(B_r^{d-n})t^{2d-i-j-n}\\
&\qquad\times \int_{A_h(d,d-n)}\int_{A_h(d,j+n)}\int_{B\cap E\cap G}1 \ \cH^j(dy)\,\mu_{j+n}(dG)\,\mu_{d-n}(dE)\\
&= {1\over\lambda_i \lambda_j}\,\sum_{n=0}^{m(d,i,j)}\alpha(d,i,j,n)\beta(d,i,j,n){\omega_{d+1}\omega_{i+1}\over\omega_{i+n+1}\omega_{d-n+1}}\cH^{d-n}(B_r^{d-n})t^{2d-i-j-n}\\
&\qquad\times \int_{A_h(d,d-n)}\int_{A_h(d,j+n)}\cH^j(B\cap E\cap G)\,\mu_{j+n}(dG)\,\mu_{d-n}(dE).
\end{align*}
The two remaining integrals can be evaluated by using twice the Crofton formula. Indeed, noting that
 for $\mu_{d-n}$-almost all $E\in A_h(d,d-n)$ the set $B\cap E$ is either empty or has dimension $d-n$ we find that
\begin{align*}
&\int_{A_h(d,d-n)}\int_{A_h(d,j+n)}\cH^j(B\cap E\cap G)\,\mu_{j+n}(dG)\,\mu_{d-n}(dE)\\
&\qquad ={\omega_{d+1}\omega_{j+1}\over\omega_{j+n+1}\omega_{d-n+1}}\int_{A_h(d,d-n)}\cH^{d-n}(B\cap E)\,\mu_{d-n}(dE)\\
&\qquad ={\omega_{d+1}\omega_{j+1}\over\omega_{j+n+1}\omega_{d-n+1}}\cH^d(B).
\end{align*}
Since $\cH^d(B)=1$ we finally conclude that
\begin{align*}
K_{ij}(r) &={1\over\lambda_i \lambda_j}\,\sum_{n=0}^{m(d,i,j)}\alpha(d,i,j,n)\beta(d,i,j,n){\omega_{d+1}^{2}\omega_{i+1}\omega_{j+1}\over\omega_{d-n+1}^{2}\omega_{i+n+1}\omega_{j+n+1}} t^{2d-i-j-n}\cH^{d-n}(B_r^{d-n})\\
&={1\over\lambda_i \lambda_j}\,\sum_{n=0}^{m(d,i,j)}\alpha(d,i,j,n)\beta(d,i,j,n){\omega_{d+1}^{2}\omega_{i+1}\omega_{j+1}\over\omega_{d-n+1}^{2}\omega_{i+n+1}\omega_{j+n+1}}\\
&\hspace{6cm}\times\omega_{d-n}t^{2d-i-j-n}\int_0^r\sinh^{d-n-1}(s)\,ds.
\end{align*}
Simplification of the constant by means of the constants given in \eqref{lambdak} and Lemma \ref{lem:mu_almost_surely_d-n_hyperplane}  completes the proof for the mixed K-function $K_{ij}$. The formula for the mixed pair-correlation function follows by differentiation. This completes the proof of Theorem \ref{thm:KFunctionAndPCF}.\hfill $\Box$


\section{Proofs III -- Univariate limit theorems}\label{sec:ProofUni}

\subsection{The case of growing intensity: Proof of Theorem \ref{thm:CLTtToInfinity}}\label{sec:growing_intensity}

The central limit theorem is in this case a direct consequence of the central limit theorem for general Poisson U-statistics stated as Corollary 4.3 in \cite{SchulteKolmogorov} (see also \cite{EichelsbacherThaele14}).\hfill $\Box$

\subsection{The case of growing windows: Proof of Theorem \ref{thm:CLTrToInfinity}}

Our strategy in the proof of Theorem \ref{thm:CLTrToInfinity} (a) and (b) can be summarized as follows. The normal approximation bound \eqref{eq:Kolmogorov} for general U-statistics of Poisson processes is given by a sum involving terms of the type $M_{u,v}$, which are defined in \eqref{eq:defMij} and \eqref{eq:ChaosKernels} and which in turn are given as sums of integrals over partitions $\sigma\in\Pi_{\geq 2}^{\rm con}(u,u,v,v)$.
 In applying these normal approximation bounds to the Euclidean counterparts of the functionals $F_{r,t}^{(i)}$ it was possible to extract a common scaling factor from each of the integrals in $M_{u,v}$ and to treat the number of terms, that is, the number of elements of $\Pi_{\geq 2}^{\rm con}(u,u,v,v)$ as a constant, see \cite{LPST,ReitznerSchulteCLT}. In the hyperbolic set-up this is no longer possible and each integral in the definition of $M_{u,v}$ needs a separate treatment. In fact, it will turn out that these integrals exhibit  different asymptotic behaviours as  functions of $r$, as $r\to\infty$. For the analysis, we have to determine explicitly the partitions in $\Pi_{\geq 2}^{\rm con}(u,u,v,v)$ and for each such partition we have to provide a bound for the resulting integral. Since $\mu=t\mu_{d-1}$, we can bound the dependence with respect to
  the intensity $t\ge 1$ by $4(d-i)-2(u+v)+|\sigma|$ for each $\sigma \in\Pi_{\geq 2}^{\rm con}(u,u,v,v)$.

  To show that a central limit theorem fails in higher space dimensions $d\geq 4$ is the most technical part in the proof of Theorem \ref{thm:CLTrToInfinity}. We do this by showing that the fourth cumulant of the centred and normalized total volume $F_{r,t}^{(i)}$ is bounded away from $0$ by an absolute and strictly positive constant and hence does not converge to $0$. The latter in turn is the fourth cumulant of a standard Gaussian random variable. However, in view of the well known expression of the fourth cumulant in terms of the first four centred moments this approach can only work if we can ensure that the sequence of random variables
$$
\Bigg({F_{r,t}^{(i)}-\E F_{r,t}^{(i)}\over \sqrt{\Var(F_{r,t}^{(i)})}}\Bigg)^4
$$
is uniformly integrable. We will prove that this is indeed the case by showing that their fifths moments are uniformly bounded. This requires a very careful analysis of the combinatorial formula \eqref{eq:MomentsUstatistic} for the centred moments of U-statistics of Poisson processes.

The representation of a U-statistic will be as in Section \ref{sec:4.1}.
In the following computations, $c$ will be a positive constant only depending on the dimension and whose value may change from occasion to occasion.

\subsubsection{The planar case $d=2$: Proof of Theorem \ref{thm:CLTrToInfinity} (a)}\label{sec:growing_window_d=2}

As indicated above, we will use the bound \eqref{eq:Kolmogorov} in combination with \eqref{eq:defMij} and \eqref{eq:ChaosKernels}. We distinguish the cases $i=0$ and $i=1$. In the following, we can assume that $r,t\ge 1$.

For $i=1$, which corresponds to the total edge length in $B_r$, it is enough to bound $M_{1,1}(f^{(1)})$. For this we note that $\Pi_{\geq 2}^{\rm con}(1,1,1,1)$ only consists of the trivial partition $\sigma_1=\{1,2,3,4\}$, see Figure \ref{fig:proofd=2} (left panel).  Thus, using Lemma \ref{lem:lines_intersecting_ball_inequality}, we have that
  \begin{align*}
    M_{1,1}(f^{(1)}) =c\, t\int_{A_h(2,1)} \cH^1(H\cap B_r)^4\,\mu_1(d H) \leq c\,t \, e^{r}.
  \end{align*}
  Together with the lower variance bound from Lemma \ref{lem:VarianceBoundd=2} this yields
 \begin{equation}\label{eq:Boundd=2i=1}
  d\left(\frac{F_{r,t}^{(1)}-\mathbb{E}F_{r,t}^{(1)}}{\sqrt{\Var(F_{r,t}^{(1)})}},N \right) \leq c\,\frac{\sqrt{te^{r}}}{tc^{(1)}(2)\,e^r} \leq c\,t^{-1/2}\, e^{-r/2}.
 \end{equation}
 Here we used that the exponent of $t$ is given by $4(2-1)-2(1+1)+1=1$.

  Next, we deal with the case $i=0$, which corresponds to the total vertex count in $B_r$. In this situation, we need to bound the terms $M_{1,1}(f^{(0)}), M_{1,2}(f^{(0)}), M_{2,2}(f^{(0)})$. For $M_{1,1}(f^{(0)})$ we can argue as in the case $i=1$, since $\Pi_{\geq 2}^{\rm con}(1,1,1,1)$ only consists of the single partition $\sigma_1$, see Figure \ref{fig:proofd=2} (left panel). This allows us to conclude that
  \begin{align*}
    M_{1,1}(f^{(0)}) & =c \,t^5\int_{A_h(2,1)} \mathcal{H}^1(H \cap B_r)^4 \ \mu_1(dH) \leq c\,t^5\, e^{r},
    \end{align*}
where we used that the exponent of $t$ is given by $4(2-0)-2(1+1)+1=5$.
    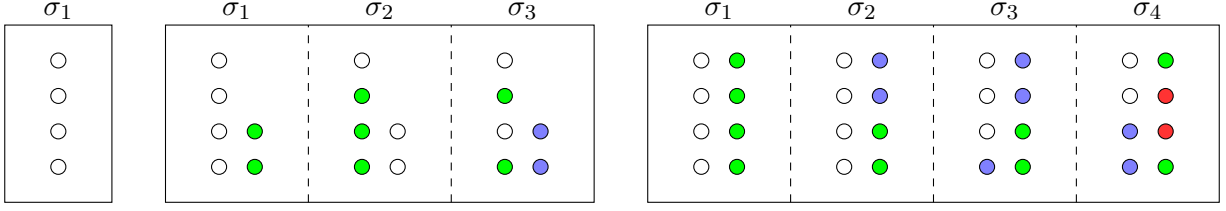
\begin{figure}[t]
    \begin{center}
    \begin{tikzpicture}[scale=0.94]
    \draw (-3.25,-0.5) rectangle (-1.75,2);
    \node at (-2.5,2.2) {$\sigma_1$};
    \filldraw[fill=white, draw=black]
        (-2.5,0) circle (3pt)
        (-2.5,0.5) circle (3pt)
        (-2.5,1) circle (3pt)
        (-2.5,1.5) circle (3pt);


    \draw (-1,-0.5) rectangle (5,2);
    \draw[dashed] (1,-0.5) -- (1,2);
    \draw[dashed] (3,-0.5) -- (3,2);
    \node at (0,2.2) {$\sigma_1$};
    \node at (2,2.2) {$\sigma_2$};
    \node at (4,2.2) {$\sigma_3$};

    \filldraw[fill=white, draw=black]
    (-0.25,0) circle (3pt)
    (-0.25,0.5) circle (3pt)
    (-0.25,1) circle (3pt)
    (-0.25,1.5) circle (3pt);
    \filldraw[fill=green, draw=black]
    (0.25,0) circle (3pt)
    (0.25,0.5) circle (3pt);

    \filldraw[fill=green, draw=black]
    (1.75,0) circle (3pt)
    (1.75,0.5) circle (3pt)
    (1.75,1) circle (3pt);
    \filldraw[fill=white, draw=black]
    (1.75,1.5) circle (3pt)
    (2.25,0) circle (3pt)
    (2.25,0.5) circle (3pt);

    \filldraw[fill=green, draw=black]
    (3.75,0) circle (3pt)
    (3.75,1) circle (3pt);
    \filldraw[fill=white, draw=black]
     (3.75,0.5) circle (3pt)
    (3.75,1.5) circle (3pt);
    \filldraw[fill=blue!50!white, draw=black]
    (4.25,0) circle (3pt)
    (4.25,0.5) circle (3pt);


    \filldraw[fill=white, draw=black]
        (6.5,0) circle (3pt)
        (6.5,0.5) circle (3pt)
        (6.5,1) circle (3pt)
        (6.5,1.5) circle (3pt);
        \filldraw[fill=green, draw=black]
        (7,0) circle (3pt)
        (7,0.5) circle (3pt)
        (7,1) circle (3pt)
        (7,1.5) circle (3pt);

    \filldraw[fill=white, draw=black]
        (8.5,0) circle (3pt)
        (8.5,0.5) circle (3pt)
        (8.5,1) circle (3pt)
        (8.5,1.5) circle (3pt);
        \filldraw[fill=green, draw=black]
        (9,0) circle (3pt)
        (9,0.5) circle (3pt);
        \filldraw[fill=blue!50!white, draw=black]
        (9,1) circle (3pt)
        (9,1.5) circle (3pt);

    \filldraw[fill=white, draw=black]
        (10.5,0.5) circle (3pt)
        (10.5,1) circle (3pt)
        (10.5,1.5) circle (3pt);
        \filldraw[fill=green, draw=black]
        (11,0) circle (3pt)
        (11,0.5) circle (3pt);
        \filldraw[fill=blue!50!white, draw=black]
        (10.5,0) circle (3pt)
        (11,1) circle (3pt)
        (11,1.5) circle (3pt);

    \filldraw[fill=white, draw=black]
        (12.5,1) circle (3pt)
        (12.5,1.5) circle (3pt);
        \filldraw[fill=green, draw=black]
        (13,1.5) circle (3pt)
        (13,0) circle (3pt);
        \filldraw[fill=blue!50!white, draw=black]
        (12.5,0) circle (3pt)
		(12.5,0.5) circle (3pt);
        \filldraw[fill=red!80!white, draw=black]		
        (13,1) circle (3pt)
        (13,0.5) circle (3pt);

    \draw (5.75,-0.5) rectangle (13.75,2);
    \draw[dashed] (7.75,-0.5) -- (7.75,2);
    \draw[dashed] (9.75,-0.5) -- (9.75,2);
    \draw[dashed] (11.75,-0.5) -- (11.75,2);
    \node at (6.75,2.2) {$\sigma_1$};
    \node at (8.75,2.2) {$\sigma_2$};
    \node at (10.75,2.2) {$\sigma_3$};
    \node at (12.75,2.2) {$\sigma_4$};
    \end{tikzpicture}
    \end{center}
    \caption{Left panel: Illustration of the partition in $\Pi_{\geq 2}^{\rm con}(1,1,1,1)$. Middle panel: Illustration of the partitions in $\Pi_{\geq 2}^{\rm con}(1,1,2,2)$. Right panel: Illustration of the partitions in $\Pi_{\geq 2}^{\rm con}(2,2,2,2)$.}
    \label{fig:proofd=2}
    \end{figure}

    To deal with $M_{1,2}(f^{(0)})$ we observe that, up to renumbering of the elements, $\Pi_{\geq 2}^{\rm con}(1,1,2,2)$ consists of precisely three partitions $\sigma_1, \ \sigma_2$ and $\sigma_3$, which are illustrated in Figure \ref{fig:proofd=2} (middle panel). For $\sigma_1$ we obtain, using Crofton's formula and Lemma \ref{lem:lines_intersecting_ball_inequality},
    \begin{align}
	\nonumber &\int_{A_h(2,1)^2}\cH^1(H_1\cap B_r)^2\,\cH^0(H_1\cap H_2\cap B_r)^2\,\mu_1^2(d(H_1,H_2))\\
	\nonumber &\qquad=\int_{A_h(2,1)^2}\cH^1(H_1\cap B_r)^2\,\cH^0(H_1\cap H_2\cap B_r)\,\mu_1^2(d(H_1,H_2))\\
	&\qquad= c\int_{A_h(2,1)}\cH^1(H_1\cap B_r)^3\,\mu_1(dH_1) \leq c\,e^r.\label{eq:26-06-19-1}
    \end{align}
    Moreover, for the partition $\sigma_2$ we compute, using twice that $\cH^1(H\cap B_r)\leq 2r$ for each $H\in A_h(2,1)$ and again Crofton's formula,
    \begin{align}
	\nonumber &\int_{A_h(2,1)^2}\cH^1(H_1\cap B_r)\, \cH^1(H_2\cap B_r)\, \cH^0(H_1\cap H_2\cap B_r)^2\,\mu_1^2(d(H_1,H_2))\\
	&\qquad \leq 4r^{2} \int_{A_h(2,1)^2}\cH^0(H_1\cap H_2\cap B_r)\,\mu_1^2(d(H_1,H_2))\leq c\, r^{2} \, e^r,\label{eq:08-07-19-1}
    \end{align}
    and for partition $\sigma_3$ we get
    \begin{align}
	\nonumber &\int_{A_h(2,1)^3} \cH^1(H_1\cap B_r)\,\cH^1(H_2\cap B_r)\,\cH^0(H_1\cap H_3\cap B_r)\,\cH^0(H_2\cap H_3\cap B_r)\,\mu_1^3(d(H_1,H_2,H_3))\\
	\nonumber &\qquad\leq 2r\int_{A_h(2,1)^2}\cH^1(H_2\cap B_r)\,\cH^1(H_3\cap B_r)\,\cH^0(H_2\cap H_3\cap B_r)\,\mu_1^2(d(H_2,H_3))\\
	&\qquad\leq 4r^2\int_{A_h(2,1)} \cH^1(H_3\cap B_r)^2\,\mu_1(dH_3) \leq c\,r^2\,e^r.\label{eq:26-06-19-2}
    \end{align}
    This yields that $M_{1,2}(f^{(0)})\leq c\,t^5\,(e^r+2r^2e^r)\leq c\,t^5\,r^2e^r$ (recall that $r,t\ge 1$). Here we used that the exponent of
    $t$ is given by $4(2-0)-2 (2+1)+\max\{2,3\}=5$.

	Now we deal with the term $M_{2,2}(f^{(0)})$, which involves a summation over partitions in $\Pi_{\geq 2}^{\rm con}(2,2,2,2)$. Up to renumbering of the elements, there are precisely four such partitions $\sigma_1$, $\sigma_2$, $\sigma_3$ and $\sigma_4$, which are illustrated in Figure \ref{fig:proofd=2} (right panel). For $\sigma_1$ we compute
    \begin{align*}
	\int_{A_h(2,1)^2}\cH^0(H_1\cap H_2\cap B_r)^4\,\mu_1^2(d(H_1,H_2)) &= \int_{A_h(2,1)^2}\cH^0(H_1\cap H_2\cap B_r)\,\mu_1^2(d(H_1,H_2))\\
	& = c\int_{A_h(2,1)}\cH^1(H_1\cap B_r)\,\mu_1(dH_1) \leq c\,e^r,
    \end{align*}
    where we used Crofton's formula and Lemma \ref{lem:lines_intersecting_ball_inequality}. Similarly, for $\sigma_2$ and $\sigma_3$ we get
    \begin{align*}
	&\int_{A_h(2,1)^3}\cH^0(H_1\cap H_2\cap B_r)^2\,\cH^0(H_1\cap H_3\cap B_r)^2\,\mu_1^3(d(H_1,H_2,H_3))\\
	&\qquad=\int_{A_h(2,1)^3}\cH^0(H_1\cap H_2\cap B_r)\,\cH^0(H_1\cap H_3\cap B_r)\,\mu_1^3(d(H_1,H_2,H_3))\\
	&\qquad=c\int_{A_h(2,1)} \cH^1(H_1\cap B_r)^2\,\mu_1(dH_1) \leq c\,e^r,
    \end{align*}
    and, additionally using that $\cH^0(H_1\cap H_2\cap B_r)\leq 1$ for $\mu_1^2$-almost all $(H_1,H_2)\in A_h(2,1)^2$,
    \begin{align*}
    &\int_{A_h(2,1)^3}\cH^0(H_1\cap H_2\cap B_r)^2\,\cH^0(H_1\cap H_3\cap B_r)\,\cH^0(H_2\cap H_3\cap B_r)\,\mu_1^3(d(H_1,H_2,H_3))\\
    &\qquad\leq\int_{A_h(2,1)^3}\cH^0(H_1\cap H_3\cap B_r)\,\cH^0(H_2\cap H_3\cap B_r)\,\mu_1^3(d(H_1,H_2,H_3))\\
    &\qquad=c\int_{A_h(2,1)}\cH^1(H_3\cap B_r)^2\,\mu_1(dH_3) \leq c\,e^r.
    \end{align*}
    Finally, we deal with $\sigma_4$. Using once more that $\cH^0(H_1\cap H_2\cap B_r)\leq 1$ for $\mu_1^2$-almost all $(H_1,H_2)\in A_h(2,1)^2$ and also that $\cH^1(H\cap B_r)\leq 2r$ for each $H\in A_h(2,1)$, and again Crofton's formula together with Lemma \ref{lem:lines_intersecting_ball_inequality}, we obtain
    \begin{align*}
	&\int_{A_h(2,1)^4}\cH^0(H_1\cap H_2\cap B_r)\,\cH^0(H_1\cap H_3\cap B_r)\,\cH^0(H_3\cap H_4\cap B_r)\\
	&\hspace{6cm}\times\cH^0(H_2\cap H_4\cap B_r)\,\mu_1^4(d(H_1,H_2,H_3,H_4))\\
	&\qquad\leq c\int_{A_h(2,1)^2}\cH^1(H_3\cap B_r)\,\cH^0(H_3\cap H_4\cap B_r)\,\cH^1(H_4\cap B_r)\,\mu_1^2(d(H_3,H_4))\\
	&\qquad\leq c\,r\int_{A_h(2,1)}\cH^1(H_4\cap B_r)^2\,\mu_1(dH_4)\leq c\,r\,e^r.
    \end{align*}
    Altogether, this yields that $M_{2,2}(f^{(0)})\leq c\,t^4\,(e^r+e^r+e^r+re^r)\leq c\,t^4\,re^r$, where the exponent of
    $t$ follows from $4\cdot 2-2\cdot 4+\max\{2,3,4\}=4$.

  Combining the bounds for $M_{1,1}(f^{(0)})$, $M_{1,2}(f^{(0)})$ and $M_{2,2}(f^{(0)})$ with the lower variance bound provided by Lemma \ref{lem:VarianceBoundd=2} we deduce from \eqref{eq:Kolmogorov} that
\begin{equation}\label{eq:Boundd=2i=0}
  d\left(\frac{F_{r,t}^{(0)}-\mathbb{E}F_{r,t}^{(0)}}{\sqrt{\Var(F_{r,t}^{(0)})}},N \right) \leq c\,\frac{\sqrt{t^5e^{r}}+\sqrt{t^5r^2 e^{r}}+\sqrt{t^4re^{r}}}{t^3c^{(0)}(2)e^r} \leq c\,t^{-1/2}\, r e^{-r/2}.
\end{equation}
  This completes the proof of Theorem \ref{thm:CLTrToInfinity} (a). \hfill $\Box$

\subsubsection{The spatial case $d=3$: Proof of Theorem \ref{thm:CLTrToInfinity} (b)}\label{sec:growing_window_d=3}

The following lemma will be used repeatedly in deriving upper bounds for integrals. For $H\in A_h(3,2)$ we write
$L_1(H)$ for an arbitrary $1$-dimensional subspace in $H$ which satisfies $d_h(H,p)=d_h(L_1(H),p)$.

\begin{Lemma}\label{upperboundL1}
Let $d=3$ and $a,b\ge 0$. If $r\ge 1$, then
$$
I(a,b):=\int_{A_h(3,2)}\cH^2(H\cap B_r)^a\,\cH^1(L_1(H)\cap B_r)^b\, \mu_2(dH)\le c\,\begin{cases}
\exp(2r)&: 0\le a<2,\\
r^{b+1}\exp(2r)&: a=2,\\
r^{b}\exp(ar)&: a>2,\\
\end{cases}
$$
where $c=c(a,b)$ is a constant depending only on $a$ and $b$.
\end{Lemma}

\begin{proof} We use the definition \eqref{eq:Croftonmass} of the measure $\mu_2$, Lemma \ref{lem:H_d_bounds} and the argument in the proof of Lemma \ref{lem:lines_intersecting_ball_inequality} to get
$$
I(a,b)\le c\int_0^re^{2s}e^{(r-s)a}(r-s+\log 2)^b\, ds.
$$
If $0\le a<2$, then
$$
I(a,b)\le c\, e^{2r}\int_0^r e^{s(a-2)}(s+\log 2)^b\, ds\le c\, e^{2r}.
$$
This also shows that $I(2,b)\le ce^{2r}r^{b+1}$. For $a>2$, we get
$$
I(a,b)\le c\, e^{ar}\int_0^re^{s(2-a)}(r-s+\log 2)^b\, ds\le c\, r^be^{ar},
$$
which completes the argument.
\end{proof}

For $d=3$ we need to distinguish the cases $i=2$, $i=1$ and $i=0$. If $i=2$ there is only one partition $\sigma_1$ (compare with the left panel of Figure \ref{fig:proofd=2}) and we obtain
\begin{align}\label{eq:M11d=3}
\int_{A_h(3,2)}\cH^2(H\cap B_r)^4\,\mu_2(dH) \leq c\, g(2,4,3,r)\leq c\,e^{4r}.
\end{align}
This proves that $M_{1,1}(f^{(2)})\leq c\,t\,e^{4r}$ and together with the lower variance bound from Lemma \ref{lem:VarianceBoundd=3} and \eqref{eq:Kolmogorov} this yields
\begin{equation}\label{eq:Boundd=3i=2}
d\left(\frac{F_{r,t}^{(2)}-\mathbb{E}F_{r,t}^{(2)}}{\sqrt{\Var(F_{r,t}^{(2)})}},N \right) \leq c\,\frac{\sqrt{te^{4r}}}{tc^{(2)}(3)e^{2r}r} \leq c\, t^{-1/2}\,r^{-1}.
\end{equation}
To deal with the case $i=1$, we need to bound $M_{1,1}(f^{(1)})$, $M_{1,2}(f^{(1)})$ and $M_{2,2}(f^{(1)})$. As in the case $d=2$, to bound $M_{1,1}(f^{(1)})$ we can argue as for $i=2$ to obtain $M_{1,1}(f^{(1)})\leq c\,t^{5}\,e^{4r}$. Next, we consider $M_{1,2}(f^{(1)})$, which requires an analysis of the integrals resulting from the three partitions $\sigma_1, \ \sigma_2$ and $\sigma_3$ shown in the middle panel of Figure \ref{fig:proofd=2}. For $\sigma_1$ we compute
\begin{align}\label{eq:M12d=3a}
&\int_{A_h(3,2)^2}\cH^2(H_1\cap B_r)^2\,\cH^1(H_1\cap H_2\cap B_r)^2\,\mu_2^2(d(H_1,H_2))\nonumber\\
&\qquad\leq \int_{A_h(3,2)^2}\cH^2(H_1\cap B_r)^2\,\cH^1(L_1(H_1)\cap B_r)\cH^1(H_1\cap H_2\cap B_r)\,\mu_2^2(d(H_1,H_2))\nonumber\\
&\qquad\leq c\, I(3,1)\le c \,re^{3r},
\end{align}
where we used the Crofton formula and Lemma \ref{upperboundL1}.
Arguing similarly for the partition $\sigma_2$ from the middle panel of Figure \ref{fig:proofd=2} we obtain
\begin{align}\label{eq:M12d=3b}
\nonumber &\int_{A_h(3,2)^2}\cH^2(H_1\cap B_r)\,\cH^2(H_2\cap B_r)\,\cH^1(H_1\cap H_2\cap B_r)^{2}\,\mu_2^2(d(H_1,H_2))\\
\nonumber &\qquad \leq c \int_{A_h(3,2)^2}\cH^2(H_1\cap B_r)\,\cH^2(H_2\cap B_r)\,\cH^1(L_1(H_1)\cap B_r)\,\cH^1(L_1(H_2)\cap B_r)\,\mu_2^2(d(H_1,H_2))\\
&\qquad \leq c\,I(1,1)^2 \leq c\,e^{4r},
\end{align}
and for $\sigma_3$ we get
\begin{align}\label{eq:M12d=3c}
 &\int_{A_h(3,2)^3}\cH^2(H_1\cap B_r)\,\cH^2(H_2\cap B_r)\,\cH^1(H_1\cap H_3\cap B_r)\,\cH^1(H_2\cap H_3\cap B_r)\,\mu_2^3(d(H_1,H_2,H_3))\nonumber\\
 &\quad \leq\int_{A_h(3,2)^3}\cH^2(H_1\cap B_r)\,\cH^2(H_2\cap B_r)\,\cH^1(L_1(H_1)\cap B_r)\,\cH^1(H_2\cap H_3\cap B_r)\,\mu_2^3(d(H_1,H_2,H_3))\nonumber\\
 &\quad \leq c\int_{A_h(3,2)^2}\cH^2(H_1\cap B_r)\,\cH^2(H_2\cap B_r)^{2}\,\cH^1(L_1(H_1)\cap B_r)\,\mu_2^2(d(H_1,H_2))\nonumber\\
 &\quad \leq c\, I(1,2) \, g(2,2,3,r) \leq c\,r\,e^{4r}.
\end{align}
We thus conclude that $M_{1,2}(f^{(1)})\leq c\,t^5\,(re^{3r}+e^{4r}+re^{4r})\leq c\,t^5\,re^{4r}$.

Finally, we deal with $M_{2,2}(f^{(1)})$ for which an analysis of the four partitions $\sigma_1$, $\sigma_2$, $\sigma_3$ and $\sigma_4$ shown in the right panel of Figure \ref{fig:proofd=2} is necessary. For $\sigma_1$ we have
\begin{align*}
&\int_{A_h(3,2)^2}\cH^1(H_1\cap H_2\cap B_r)^4\,\mu_2^2(d(H_1,H_2))\\
&\qquad\leq\int_{A_h(3,2)^2}\cH^1(H_1\cap H_2\cap B_r)\,\cH^1(L_1(H_1)\cap B_r)^3\,\mu_2^2(d(H_1,H_2))
 \leq c\, I(1,3) \leq c\,e^{2r},
\end{align*}
where we also used Crofton's formula. We continue with $\sigma_2$ and get, by similar arguments,
\begin{align*}
&\int_{A_h(3,2)^3}\cH^1(H_1\cap H_2\cap B_r)^2\,\cH^1(H_1\cap H_3\cap B_r)^2\,\mu_2^3(d(H_1,H_2,H_3))\\
&\qquad\leq \int_{A_h(3,2)^3}\cH^1(L_1(H_1)\cap B_r)^2\,\cH^1(H_1\cap H_2\cap B_r)\,\cH^1(H_1\cap H_3\cap B_r)\,\mu_2^3(d(H_1,H_2,H_3))\\
&\qquad= c\int_{A_h(3,2)}\cH^1(L_1(H_1)\cap B_r)^2\,\cH^2(H_1\cap B_r)^2\,\mu_2(dH_1)\\
&\qquad\leq c\, I(2,2)\leq c\,r^3\,e^{2r}.
\end{align*}
Moreover, for $\sigma_3$ and $\sigma_4$ we have the bounds
\begin{align*}
&\int_{A_h(3,2)^3}\cH^1(H_1\cap H_2\cap B_r)^2\,\cH^1(H_1\cap H_3\cap B_r)\,\cH^1(H_2\cap H_3\cap B_r)\,\mu_2^3(d(H_1,H_2,H_3))\\
&\qquad\leq c\int_{A_h(3,2)^3}\cH^1(L_1(H_1)\cap B_r)^3\,\cH^1(H_2\cap H_3\cap B_r)\,\mu_2^3(d(H_1,H_2,H_3))\\
&\qquad\leq c\,\cH^3(B_r)\, I(0,3) \leq c\,e^{4r}
\end{align*}
and
\begin{align*}
&\int_{A_h(3,2)^4}\cH^1(H_1\cap H_2\cap B_r)\,\cH^1(H_1\cap H_3\cap B_r)\,\cH^1(H_3\cap H_4\cap B_r)\\
&\hspace{4cm}\times\cH^1(H_2\cap H_4\cap B_r)\,\mu_2^4(d(H_1,H_2,H_3,H_4))\\
&\qquad\leq \int_{A_h(3,2)^4}\cH^1(L_1(H_1)\cap B_r)^2\,\cH^1(H_2\cap H_4\cap B_r)\,\cH^1(H_3\cap H_4\cap B_r)\,\mu_2^4(d(H_1,H_2,H_3,H_4))\\
&\qquad= c \int_{A_h(2,3)^2}\cH^1(L_1(H_1)\cap B_r)^2\,\cH^2(H_4\cap B_r)^2\,\mu_2^2(d(H_1,H_4))\\
&\qquad\leq c\, I(0,2)\, g(2,2,3,r)\leq c\,r\,e^{4r}.
\end{align*}
Altogether this gives $M_{2,2}(f^{(1)})\leq c\,t^4\,(e^{2r}+r^3e^{2r}+e^{4r}+re^{4r})\leq c\, t^4\,re^{4r}$. The estimates for $M_{1,1}(f^{(1)})$, $M_{1,2}(f^{(1)})$ and $M_{2,2}(f^{(1)})$ together with Lemma \ref{lem:VarianceBoundd=3} and \eqref{eq:Kolmogorov} show that
\begin{equation}\label{eq:Boundd=3i=1}
  d\left(\frac{F_{r,t}^{(1)}-\mathbb{E}F_{r,t}^{(1)}}{\sqrt{\Var(F_{r,t}^{(1)})}},N \right) \leq c\, \frac{\sqrt{t^5e^{4r}}+\sqrt{t^5re^{4r}}+\sqrt{t^4re^{4r}}}{t^3c^{(1)}(3)e^{2r}r} \leq c\,t^{-1/2}\,r^{-1/2}.
\end{equation}

    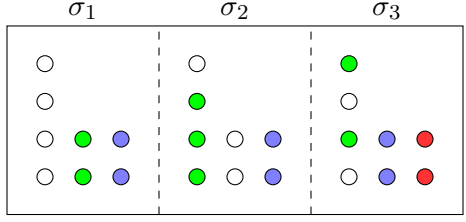
\begin{figure}[t]
    \begin{center}
    \begin{tikzpicture}
    \draw (-0.5,-0.5) rectangle (5.5,2);
    \draw[dashed] (1.5,-0.5) -- (1.5,2);
    \draw[dashed] (3.5,-0.5) -- (3.5,2);
    \node at (0.5,2.2) {$\sigma_1$};
    \node at (2.5,2.2) {$\sigma_2$};
    \node at (4.5,2.2) {$\sigma_3$};

    \filldraw[fill=white, draw=black]
    (0,0) circle (3pt)
    (0,0.5) circle (3pt)
    (0,1) circle (3pt)
    (0,1.5) circle (3pt);
    \filldraw[fill=green, draw=black]
    (0.5,0) circle (3pt)
    (0.5,0.5) circle (3pt);
    \filldraw[fill=blue!50!white, draw=black]
     (1,0) circle (3pt)
     (1,0.5) circle (3pt);

    \filldraw[fill=white, draw=black]
    (2,1.5) circle (3pt)
    (2.5,0.5) circle (3pt)
    (2.5,0) circle (3pt);
    \filldraw[fill=green, draw=black]
    (2,0) circle (3pt)
    (2,0.5) circle (3pt)
    (2,1) circle (3pt);
    \filldraw[fill=blue!50!white, draw=black]
     (3,0) circle (3pt)
     (3,0.5) circle (3pt);

         \filldraw[fill=white, draw=black]
         (4,0) circle (3pt)
         (4,1) circle (3pt);
         \filldraw[fill=green, draw=black]
         (4,0.5) circle (3pt)
         (4,1.5) circle (3pt);
         \filldraw[fill=blue!50!white, draw=black]
          (4.5,0) circle (3pt)
          (4.5,0.5) circle (3pt);
         \filldraw[fill=red!80!white, draw=black]
          (5,0) circle (3pt)
          (5,0.5) circle (3pt);

    \end{tikzpicture}
    \end{center}
    \caption{Illustration of the partition in $\Pi_{\geq 2}^{\rm con}(1,1,3,3)$. }
    \label{fig:proofd=3}
    \end{figure}

Finally, we need to treat the case of $F_{r,t}^{(0)}$, which requires to find upper bounds for the terms $M_{u,v}(f^{(0)})$ with $(u,v)\in\{(1,1),(1,2),(1,3),(2,2),(2,3),(3,3)\}$. We have $M_{1,1}(f^{(0)})\leq c\,t^9\,e^{4r}$ from \eqref{eq:M11d=3}. To treat $M_{1,2}(f^{(0)})$ we need to consider the partitions $\sigma_1,\sigma_2$ and $\sigma_3$ shown in the middle panel of Figure \ref{fig:proofd=2} and to obtain upper bounds for
the three integrals which are already treated in \eqref{eq:M12d=3a}, \eqref{eq:M12d=3b} and \eqref{eq:M12d=3c}.  This implies that  $M_{1,2}(f^{(0)})\leq c\,t^9\,re^{4r}$. Next, we deal with $M_{1,3}(f^{(0)})$, which can be expressed as a sum over the three partitions $\sigma_1$, $\sigma_2$ and $\sigma_3$ shown in Figure \ref{fig:proofd=3}. For $\sigma_1$, using that $\cH^0(H_1\cap H_2\cap H_3\cap B_r)\leq 1$ for $\mu_2^3$-almost all $(H_1,H_2,H_3)\in A_h(3,2)^3$, we have that
\begin{align*}
&\int_{A_h(3,2)^3}\cH^2(H_1\cap B_r)^2\,\cH^0(H_1\cap H_2\cap H_3\cap B_r)^2\,\mu_2^3(d(H_1,H_2,H_3))\\
&\qquad = \int_{A_h(3,2)^3}\cH^2(H_1\cap B_r)^2\,\cH^0(H_1\cap H_2\cap H_3\cap B_r)\,\mu_2^3(d(H_1,H_2,H_3))\\
&\qquad=c \int_{A_h(3,2)}\cH^2(H_1\cap B_r)^3\,\mu_2(dH_1) \leq c \,g(2,3,3,r) \leq c\,e^{3r},
\end{align*}
where we also used Crofton's formula and Lemma \ref{lem:lines_intersecting_ball_inequality}. Similarly, for $\sigma_2$ we obtain
\begin{align*}
&\int_{A_h(3,2)^3}\cH^2(H_1\cap B_r)\,\cH^2(H_2\cap B_r)\,\cH^0(H_1\cap H_2\cap H_3\cap B_r)^2\,\mu_2^3(d(H_1,H_2,H_3))\\
&\qquad = \int_{A_h(3,2)^3}\cH^2(H_1\cap B_r)\,\cH^2(H_2\cap B_r)\,\cH^0(H_1\cap H_2\cap H_3\cap B_r)\,\mu_2^3(d(H_1,H_2,H_3))\\
&\qquad\leq c\,\cH^3(B_r) \, I(1,1)\leq c\,e^{4r},
\end{align*}
and for $\sigma_3$ we have that
\begin{align*}
&\int_{A_h(3,2)^4}\cH^2(H_1\cap B_r)\,\cH^2(H_2\cap B_r)\,\cH^0(H_1\cap H_3\cap H_4\cap B_r)\\
&\hspace{4cm}\times\cH^0(H_2\cap H_3\cap H_4\cap B_r)\,\mu_2^4(d(H_1,H_2,H_3,H_4))\\
&\qquad\leq \int_{A_h(3,2)^4}\cH^2(H_1\cap B_r)\,\cH^2(H_2\cap B_r)\,\cH^0(H_1\cap H_3\cap H_4\cap B_r)\,\mu_2^4(d(H_1,H_2,H_3,H_4))\\
&\qquad=c\,\cH^3(B_r)\int_{A_h(3,2)}\cH^2(H_1\cap B_r)^2\,\mu_2(dH_1)\leq c\, e^{2r}\, g(2,2,3,r)\le  c\,re^{4r}.
\end{align*}
This proves that $M_{1,3}(f^{(0)})\leq c\,t^8\,(e^{3r}+e^{4r}+re^{4r})\leq c\,t^8\,re^{4r}$.

The next term is $M_{2,2}(f^{(0)})$. However, up to a constant, this term is the same as $M_{2,2}(f^{(1)})$, which was already bounded above. This yields that $M_{2,2}(f^{(0)})\leq c\,t^8\, re^{4r}$ and it remains to consider $M_{2,3}(f^{(0)})$ and $M_{3,3}(f^{(0)})$.
    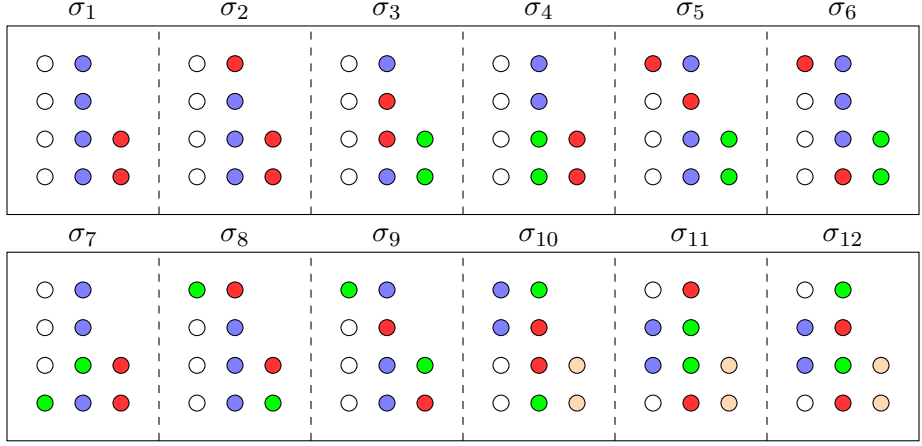
\begin{figure}[t]
    \begin{center}
    \begin{tikzpicture}
    \draw (-0.5,-0.5) rectangle (13.5-2,2);
    \draw[dashed] (1.5,-0.5) -- (1.5,2);
    \draw[dashed] (3.5,-0.5) -- (3.5,2);
    \draw[dashed] (5.5,-0.5) -- (5.5,2);
    \draw[dashed] (7.5,-0.5) -- (7.5,2);
    \draw[dashed] (9.5,-0.5) -- (9.5,2);
    \node at (0.5,2.2) {$\sigma_1$};
    \node at (2.5,2.2) {$\sigma_2$};
    \node at (4.5,2.2) {$\sigma_3$};
    \node at (6.5,2.2) {$\sigma_4$};
    \node at (8.5,2.2) {$\sigma_5$};
    \node at (10.5,2.2) {$\sigma_6$};

	\filldraw[fill=white!50!white, draw=black] (0,0) circle (3pt);
	\filldraw[fill=white!50!white, draw=black] (0,0.5) circle (3pt);
	\filldraw[fill=white!50!white, draw=black] (0,1) circle (3pt);
	\filldraw[fill=white!50!white, draw=black] (0,1.5) circle (3pt);
	
	\filldraw[fill=blue!50!white, draw=black] (0.5,0) circle (3pt);
	\filldraw[fill=blue!50!white, draw=black] (0.5,0.5) circle (3pt);
	\filldraw[fill=blue!50!white, draw=black] (0.5,1) circle (3pt);
	\filldraw[fill=blue!50!white, draw=black] (0.5,1.5) circle (3pt);
	
	\filldraw[fill=red!80!white, draw=black] (1,0) circle (3pt);
	\filldraw[fill=red!80!white, draw=black] (1,0.5) circle (3pt);

	 \filldraw[fill=white!50!white, draw=black] (2,0) circle (3pt);
	 \filldraw[fill=white!50!white, draw=black] (2,0.5) circle (3pt);
	 \filldraw[fill=white!50!white, draw=black] (2,1) circle (3pt);
	 \filldraw[fill=white!50!white, draw=black] (2,1.5) circle (3pt);
	
	 \filldraw[fill=blue!50!white, draw=black] (2.5,0) circle (3pt);
	 \filldraw[fill=blue!50!white, draw=black] (2.5,0.5) circle (3pt);
	 \filldraw[fill=blue!50!white, draw=black] (2.5,1) circle (3pt);
   	  \filldraw[fill=red!80!white, draw=black]  (2.5,1.5) circle (3pt);
	
	 \filldraw[fill=red!80!white, draw=black] (3,0) circle (3pt);
	 \filldraw[fill=red!80!white, draw=black] (3,0.5) circle (3pt);

	 \filldraw[fill=white!50!white, draw=black] (4,0) circle (3pt);
	 \filldraw[fill=white!50!white, draw=black] (4,0.5) circle (3pt);
	 \filldraw[fill=white!50!white, draw=black] (4,1) circle (3pt);
	 \filldraw[fill=white!50!white, draw=black] (4,1.5) circle (3pt);
	
	 \filldraw[fill=blue!50!white, draw=black] (4.5,0) circle (3pt);
	 \filldraw[fill=red!80!white, draw=black] (4.5,0.5) circle (3pt);
	 \filldraw[fill=red!80!white, draw=black] (4.5,1) circle (3pt);
	 \filldraw[fill=blue!50!white, draw=black] (4.5,1.5) circle (3pt);
	
	 \filldraw[fill=green, draw=black] (5,0) circle (3pt);
	 \filldraw[fill=green, draw=black] (5,0.5) circle (3pt);

	 \filldraw[fill=white!50!white, draw=black] (6,0) circle (3pt);
	 \filldraw[fill=white!50!white, draw=black] (6,0.5) circle (3pt);
	 \filldraw[fill=white!50!white, draw=black] (6,1) circle (3pt);
	 \filldraw[fill=white!50!white, draw=black] (6,1.5) circle (3pt);
	
	 \filldraw[fill=green, draw=black] (6.5,0) circle (3pt);
	 \filldraw[fill=green, draw=black] (6.5,0.5) circle (3pt);
	 \filldraw[fill=blue!50!white, draw=black] (6.5,1) circle (3pt);
	 \filldraw[fill=blue!50!white, draw=black] (6.5,1.5) circle (3pt);
	
	 \filldraw[fill=red!80!white, draw=black] (7,0) circle (3pt);
	 \filldraw[fill=red!80!white, draw=black] (7,0.5) circle (3pt);

	 \filldraw[fill=white!50!white, draw=black] (8,0) circle (3pt);
	 \filldraw[fill=white!50!white, draw=black] (8,0.5) circle (3pt);
	 \filldraw[fill=white!50!white, draw=black] (8,1) circle (3pt);
	 \filldraw[fill=red!80!white, draw=black] (8,1.5) circle (3pt);
	
	 \filldraw[fill=blue!50!white, draw=black] (8.5,0) circle (3pt);
	 \filldraw[fill=blue!50!white, draw=black] (8.5,0.5) circle (3pt);
	 \filldraw[fill=red!80!white, draw=black] (8.5,1) circle (3pt);
	 \filldraw[fill=blue!50!white, draw=black] (8.5,1.5) circle (3pt);
	
	 \filldraw[fill=green, draw=black] (9,0) circle (3pt);
	 \filldraw[fill=green, draw=black] (9,0.5) circle (3pt);

	 \filldraw[fill=white!50!white, draw=black] (10,0) circle (3pt);
	 \filldraw[fill=white!50!white, draw=black] (10,0.5) circle (3pt);
	 \filldraw[fill=white!50!white, draw=black] (10,1) circle (3pt);
	 \filldraw[fill=red!80!white, draw=black] (10,1.5) circle (3pt);
	
	 \filldraw[fill=red!80!white, draw=black] (10.5,0) circle (3pt);
	 \filldraw[fill=blue!50!white, draw=black] (10.5,0.5) circle (3pt);
	 \filldraw[fill=blue!50!white, draw=black] (10.5,1) circle (3pt);
	 \filldraw[fill=blue!50!white, draw=black] (10.5,1.5) circle (3pt);
	
	 \filldraw[fill=green, draw=black] (11,0) circle (3pt);
	 \filldraw[fill=green, draw=black] (11,0.5) circle (3pt);

	
    \draw (-0.5,-3.5) rectangle (13.5-2,-1);
    \draw[dashed] (1.5,-3.5) -- (1.5,-1);
    \draw[dashed] (3.5,-3.5) -- (3.5,-1);
    \draw[dashed] (5.5,-3.5) -- (5.5,-1);
    \draw[dashed] (7.5,-3.5) -- (7.5,-1);
    \draw[dashed] (9.5,-3.5) -- (9.5,-1);
    \node at (0.5,2.2-3) {$\sigma_7$};
    \node at (2.5,2.2-3) {$\sigma_8$};
    \node at (4.5,2.2-3) {$\sigma_{9}$};
    \node at (6.5,2.2-3) {$\sigma_{10}$};
    \node at (8.5,2.2-3) {$\sigma_{11}$};
    \node at (10.5,2.2-3) {$\sigma_{12}$};

	\filldraw[fill=green, draw=black] (0,-3) circle (3pt);
	\filldraw[fill=white, draw=black] (0,-2.5) circle (3pt);
	\filldraw[fill=white, draw=black] (0,-2) circle (3pt);
	\filldraw[fill=white, draw=black] (0,-1.5) circle (3pt);
	
	\filldraw[fill=blue!50!white, draw=black] (0.5,-3) circle (3pt);
	\filldraw[fill=green, draw=black] (0.5,-2.5) circle (3pt);
	\filldraw[fill=blue!50!white, draw=black] (0.5,-2) circle (3pt);
	\filldraw[fill=blue!50!white, draw=black] (0.5,-1.5) circle (3pt);
	
	\filldraw[fill=red!80!white, draw=black] (1,-3) circle (3pt);
	\filldraw[fill=red!80!white, draw=black] (1,-2.5) circle (3pt);

	\filldraw[fill=white, draw=black] (0+2,-3) circle (3pt);
	\filldraw[fill=white, draw=black] (0+2,-2.5) circle (3pt);
	\filldraw[fill=white, draw=black] (0+2,-2) circle (3pt);
	\filldraw[fill=green, draw=black] (0+2,-1.5) circle (3pt);
	
	\filldraw[fill=blue!50!white, draw=black] (0.5+2,-3) circle (3pt);
	\filldraw[fill=blue!50!white, draw=black] (0.5+2,-2.5) circle (3pt);
	\filldraw[fill=blue!50!white, draw=black] (0.5+2,-2) circle (3pt);
	\filldraw[fill=red!80!white, draw=black] (0.5+2,-1.5) circle (3pt);
	
	\filldraw[fill=green, draw=black] (1+2,-3) circle (3pt);
	\filldraw[fill=red!80!white, draw=black] (1+2,-2.5) circle (3pt);

	\filldraw[fill=white, draw=black] (0+4,-3) circle (3pt);
	\filldraw[fill=white, draw=black] (0+4,-2.5) circle (3pt);
	\filldraw[fill=white, draw=black] (0+4,-2) circle (3pt);
	\filldraw[fill=green, draw=black] (0+4,-1.5) circle (3pt);
	
	\filldraw[fill=blue!50!white, draw=black] (0.5+4,-3) circle (3pt);
	\filldraw[fill=blue!50!white, draw=black] (0.5+4,-2.5) circle (3pt);
	\filldraw[fill=red!80!white, draw=black] (0.5+4,-2) circle (3pt);
	\filldraw[fill=blue!50!white, draw=black] (0.5+4,-1.5) circle (3pt);
	
	\filldraw[fill=red!80!white, draw=black] (1+4,-3) circle (3pt);
	\filldraw[fill=green, draw=black] (1+4,-2.5) circle (3pt);

	\filldraw[fill=white, draw=black] (0+6,-3) circle (3pt);
	\filldraw[fill=white, draw=black] (0+6,-2.5) circle (3pt);
	\filldraw[fill=blue!50!white, draw=black] (0+6,-2) circle (3pt);
	\filldraw[fill=blue!50!white, draw=black] (0+6,-1.5) circle (3pt);
	
	\filldraw[fill=green, draw=black] (0.5+6,-3) circle (3pt);
	\filldraw[fill=red!80!white, draw=black] (0.5+6,-2.5) circle (3pt);
	\filldraw[fill=red!80!white, draw=black] (0.5+6,-2) circle (3pt);
	\filldraw[fill=green, draw=black] (0.5+6,-1.5) circle (3pt);
	
	\filldraw[fill=orange!30!white, draw=black] (1+6,-3) circle (3pt);
	\filldraw[fill=orange!30!white, draw=black] (1+6,-2.5) circle (3pt);

	\filldraw[fill=white, draw=black] (2+6,-3) circle (3pt);
	\filldraw[fill=blue!50!white, draw=black] (2+6,-2.5) circle (3pt);
	\filldraw[fill=blue!50!white, draw=black] (2+6,-2) circle (3pt);
	\filldraw[fill=white, draw=black] (2+6,-1.5) circle (3pt);
	
	\filldraw[fill=red!80!white, draw=black] (2.5+6,-3) circle (3pt);
	\filldraw[fill=green, draw=black] (2.5+6,-2.5) circle (3pt);
	\filldraw[fill=green, draw=black] (2.5+6,-2) circle (3pt);
	\filldraw[fill=red!80!white, draw=black] (2.5+6,-1.5) circle (3pt);
	
	\filldraw[fill=orange!30!white, draw=black] (3+6,-3) circle (3pt);
	\filldraw[fill=orange!30!white, draw=black] (3+6,-2.5) circle (3pt);

	\filldraw[fill=white, draw=black] (4+6,-3) circle (3pt);
	\filldraw[fill=blue!50!white, draw=black] (4+6,-2.5) circle (3pt);
	\filldraw[fill=blue!50!white, draw=black] (4+6,-2) circle (3pt);
	\filldraw[fill=white, draw=black] (4+6,-1.5) circle (3pt);
	
	\filldraw[fill=red!80!white, draw=black] (4.5+6,-3) circle (3pt);
	\filldraw[fill=green, draw=black] (4.5+6,-2.5) circle (3pt);
	\filldraw[fill=red!80!white, draw=black] (4.5+6,-2) circle (3pt);
	\filldraw[fill=green, draw=black] (4.5+6,-1.5) circle (3pt);
	
	\filldraw[fill=orange!30!white, draw=black] (5+6,-3) circle (3pt);
	\filldraw[fill=orange!30!white, draw=black] (5+6,-2.5) circle (3pt);


    \end{tikzpicture}
    \end{center}
    \caption{Illustration of the partition in $\Pi_{\geq 2}^{\rm con}(2,2,3,3)$. }
    \label{fig:proofd=3b}
    \end{figure}

In order to deal with $M_{2,3}(f^{(0)})$, up to renumbering of the elements precisely the $12$ partitions $\sigma_1,\ldots,\sigma_{12}$ in $\Pi_{\geq 2}^{\rm con}(2,2,3,3)$ have to be considered, see Figure \ref{fig:proofd=3b}. Using that $\cH^0(H_1\cap H_2\cap H_3\cap B_r)\leq 1$ for $\mu_2^3$-almost all $(H_1,H_2,H_3)\in A_h(3,2)^3$ we find for $\sigma_1$ that
\begin{align*}
&\int_{A_h(3,2)^3}\cH^1(H_1\cap H_2\cap B_r)^2\,\cH^0(H_1\cap H_2\cap H_3\cap B_r)^2\,\mu_2^3(d(H_1,H_2,H_3))\\
&\qquad = \int_{A_h(3,2)^3}\cH^1(H_1\cap H_2\cap B_r)^2\,\cH^0(H_1\cap H_2\cap H_3\cap B_r)\,\mu_2^3(d(H_1,H_2,H_3)).
\end{align*}
Applying now Crofton's formula, we obtain the upper bound
\begin{align*}
&c\int_{A_h(3,2)^2}\cH^1(H_1\cap H_2\cap B_r)^3\,\mu_2^3(d(H_1,H_2))\le c\, I(1,2)\le c\, e^{2r}.
\end{align*}
The same arguments also lead to bounds for the remaining partitions $\sigma_2,\ldots,\sigma_{12}$. As for $\sigma_1$, the first step is always to bound the $0$-dimensional Hausdorff measure $\cH^0(\,\cdot\,)$ of the intersection of the three planes corresponding to the last row of the partition by $1$, which is a valid estimate for $\mu_2^3$-almost all triples of planes. For this reason we systematically skip this first step in our following computations and only show how to deal with the integral of the three remaining terms
\begin{align*}
&\cH^1(\text{intersection of the  2 planes corresponding to the first row})\\
&\times\cH^1(\text{intersection of the  2 planes corresponding to the second row})\\
&\quad\times\cH^0(\text{intersection of the  3 planes corresponding to the third row}).
\end{align*}
For $\sigma_2$ we get
\begin{align*}
&\int_{A_h(3,2)^3}\cH^1(H_1\cap H_2\cap B_r)\,\cH^1(H_1\cap H_3\cap B_r)\cH^0(H_1\cap H_2\cap H_3\cap B_r)\,\mu_2^3(d(H_1,H_2,H_3))\\
&\qquad\leq c\int_{A_h(3,2)^3}\cH^1(L_1(H_1)\cap B_r)^2\,\cH^0(H_1\cap H_2\cap H_3\cap B_r)\,\mu_2^3(d(H_1,H_2,H_3)) \\
&\qquad\leq c\, I(1,2)\leq c\,e^{2r},
\end{align*}
for $\sigma_3$ we get
\begin{align*}
&\int_{A_h(3,2)^4}\cH^1(H_1\cap H_2\cap B_r)\,\cH^1(H_1\cap H_3\cap B_r)\,\cH^0(H_1\cap H_3\cap H_4\cap B_r)\,\mu_2^4(d(H_1,H_2,H_3,H_4))\\
&\qquad\leq c\int_{A_h(3,2)^3}\cH^1(H_1\cap H_2\cap B_r)\,\cH^1(L_1(H_1)\cap B_r)\,\cH^1(L_1(H_3)\cap B_r)\,\mu_2^3(d(H_1,H_2,H_3))\\
&\qquad\leq c\, I(1,1)\, I(0,1)\leq c\,e^{4r},
\end{align*}
for $\sigma_4$ we get
\begin{align*}
&\int_{A_h(3,2)^4}\cH^1(H_1\cap H_2\cap B_r)^2\,\cH^0(H_1\cap H_3\cap H_4\cap B_r)\,\mu_2^4(d(H_1,H_2,H_3,H_4))\\
&\qquad\leq c\int_{A_h(3,2)^2}\cH^1(L_1(H_1)\cap B_r)\,\cH^1(L_1(H_2)\cap B_r)\,\cH^2(H_1\cap B_r)\,\mu_2^2(d(H_1,H_2))\\
&\qquad\leq c\, I(1,1)\, I(0,1)\leq c\,e^{4r},
\end{align*}
for $\sigma_5$ we get
\begin{align*}
&\int_{A_h(3,2)^4}\cH^1(H_1\cap H_2\cap B_r)\,\cH^1(H_1\cap H_3\cap B_r)\,\cH^0(H_2\cap H_3\cap H_4\cap B_r)\,\mu_2^4(d(H_1,H_2,H_3,H_4))\\
&\qquad\leq c\int_{A_h(3,2)^3}\cH^1(H_1\cap H_2\cap B_r)\,\cH^1(L_1(H_3)\cap B_r)^{2}\,\mu_2^3(d(H_1,H_2,H_3))\\
&\qquad \leq c\,\cH^3(B_r)\, I(0,2)\leq c\,e^{4r},
\end{align*}
for $\sigma_6$ we get
\begin{align*}
& \int_{A_h(3,2)^4}\cH^1(H_1\cap H_2\cap B_r)\,\cH^1(H_2\cap H_3\cap B_r)\cH^0(H_2\cap H_3\cap H_4\cap B_r)\,\mu_2^4(d(H_1,H_2,H_3,H_4))\\
&\qquad\leq c\int_{A_h(3,2)^3}\cH^1(H_1\cap H_2\cap B_r)\,\cH^1(L_1(H_3)\cap B_r)^{2}\,\mu_2^3(d(H_1,H_2,H_3)) ,
\end{align*}
which is the same as for $\sigma_5$ and thus bounded by $e^{4r}$. For $\sigma_7$ we have
\begin{align*}
&\int_{A_h(3,2)^4}\cH^1(H_1\cap H_2\cap B_r)^2\,\cH^0(H_1\cap H_3\cap H_4\cap B_r)\,\mu_2^4(d(H_1,H_2,H_3,H_4))\\
&\qquad\leq c\int_{A_h(3,2)^2}\cH^1(L_1(H_1)\cap B_r)\,\cH^1(L_1(H_2)\cap B_r)\,\cH^2(H_1\cap B_r)\,\mu_2^2(d(H_1,H_2))\\
&\qquad\leq c\, I(1,1)\, I(0,1) \leq c\,e^{4r},
\end{align*}
for $\sigma_8$ we have
\begin{align*}
&\int_{A_h(3,2)^4}\cH^1(H_1\cap H_2\cap B_r)\,\cH^1(H_3\cap H_4\cap B_r)\,\cH^0(H_2\cap H_3\cap H_4\cap B_r)\,\mu_2^4(d(H_1,H_2,H_3,H_4))\\
&\qquad\leq \int_{A_h(3,2)^4}\cH^1(H_1\cap H_2\cap B_r)\,\cH^1(H_3\cap H_4\cap B_r)\,\mu_2^4(d(H_1,H_2,H_3,H_4))\\
&\qquad =c\,\cH^3(B_r)^2 \leq c\,e^{4r},
\end{align*}
for $\sigma_9$ we have
\begin{align*}
&\int_{A_h(3,2)^4}\cH^1(H_1\cap H_2\cap B_r)\,\cH^1(H_3\cap H_4\cap B_r)\,\cH^0(H_1\cap H_2\cap H_3\cap B_r)\,\mu_2^4(d(H_1,H_2,H_3,H_4))\\
&\qquad\leq \int_{A_h(3,2)^4}\cH^1(H_1\cap H_2\cap B_r)\,\cH^1(H_3\cap H_4\cap B_r)\,\mu_2^4(d(H_1,H_2,H_3,H_4))\\
&\qquad
=c\,\cH^3(B_r)^2 \leq c\,e^{4r}.
\end{align*}
Next, for $\sigma_{10}$ we get
\begin{align*}
&\int_{A_h(3,2)^5}\cH^1(H_1\cap H_2\cap B_r)\,\cH^1(H_1\cap H_3\cap B_r)\\
&\hspace{4cm}\times\cH^0(H_3\cap H_4\cap H_5\cap B_r)\,\mu_2^5(d(H_1,H_2,H_3,H_4,H_5))\\
&\qquad\leq c\int_{A_h(3,2)^3}\cH^1(H_1\cap H_2\cap B_r)\,\cH^1(L_1(H_3)\cap B_r)\,\cH^2(H_3\cap B_r)\,\mu_2^3(d(H_1,H_2,H_3))\\
&\qquad\leq c\,\cH^3(B_r)\, I(1,1)\leq c\,e^{4r},
\end{align*}
for $\sigma_{11}$ we get
\begin{align*}
&\int_{A_h(3,2)^5}\cH^1(H_1\cap H_2\cap B_r)\,\cH^1(H_3\cap H_4\cap B_r)\\
&\hspace{4cm}\times \cH^0(H_3\cap H_4\cap H_5\cap B_r)\,\mu_2^5(d(H_1,H_2,H_3,H_4,H_5))\\
&\qquad=c\int_{A_h(3,2)^4}\cH^1(H_1\cap H_2\cap B_r)\,\cH^1(H_3\cap H_4\cap B_r)^2\,\mu_2^4(d(H_1,H_2,H_3,H_4))\\
&\qquad\leq c\,\cH^3(B_r)\int_{A_h(3,2)^2}\cH^1(L_1(H_3)\cap B_r)\,\cH^1(H_3\cap H_4\cap B_r)\,\mu_2^2(d(H_3,H_4))\\
&\qquad= c\,\cH^3(B_r)\int_{A_h(3,2)}\cH^1(L_1(H_3)\cap B_r)\,\cH^2(H_3\cap B_r)\,\mu_2(d H_3)
\le c\,\cH^3(B_r)\,I(1,1) \leq c\,e^{4r}
\end{align*}
and for $\sigma_{12}$ we get
\begin{align*}
&\int_{A_h(3,2)^5}\cH^1(H_1\cap H_2\cap B_r)\,\cH^1(H_3\cap H_4\cap B_r)\\
&\hspace{4cm}\times \cH^0(H_2\cap H_3\cap H_5\cap B_r)\,\mu_2^5(d(H_1,H_2,H_3,H_4,H_5))\\
&\qquad\leq c\,\cH^3(B_r)\int_{A_h(3,2)}\cH^2(H_3\cap B_r)\,\cH^1(L_1(H_3)\cap B_r)\,\mu_2(d H_3)
\leq  c\,e^{2r}\, I(1,1)\le c\, e^{4r}.
\end{align*}
Altogether this yields that $M_{2,3}(f^{(0)})\leq c\,t^7\,(2\, e^{2r}+10 \, e^{4r})\leq c\,t^7\,e^{4r}$.
    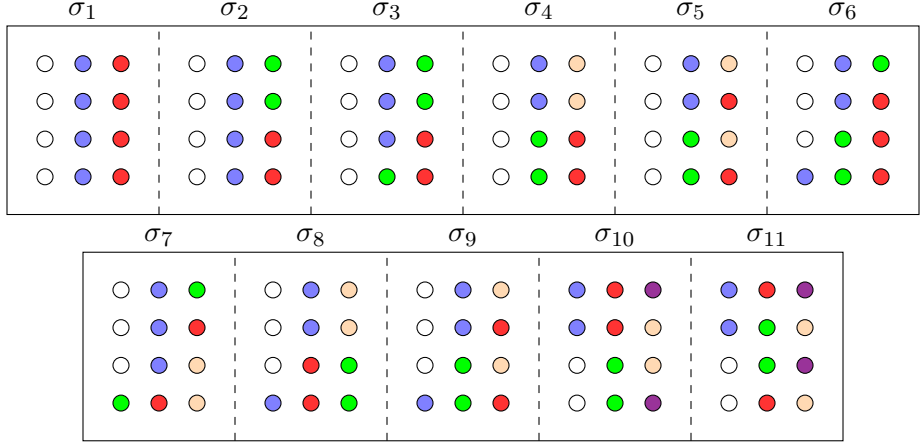
\begin{figure}[t]
    \begin{center}
    \begin{tikzpicture}
    \draw (-0.5,-0.5) rectangle (13.5-2,2);
    \draw[dashed] (1.5,-0.5) -- (1.5,2);
    \draw[dashed] (3.5,-0.5) -- (3.5,2);
    \draw[dashed] (5.5,-0.5) -- (5.5,2);
    \draw[dashed] (7.5,-0.5) -- (7.5,2);
    \draw[dashed] (9.5,-0.5) -- (9.5,2);
    \node at (0.5,2.2) {$\sigma_1$};
    \node at (2.5,2.2) {$\sigma_2$};
    \node at (4.5,2.2) {$\sigma_3$};
    \node at (6.5,2.2) {$\sigma_4$};
    \node at (8.5,2.2) {$\sigma_5$};
    \node at (10.5,2.2) {$\sigma_6$};

	\filldraw[fill=white!50!white, draw=black] (0,0) circle (3pt);
	\filldraw[fill=white!50!white, draw=black] (0,0.5) circle (3pt);
	\filldraw[fill=white!50!white, draw=black] (0,1) circle (3pt);
	\filldraw[fill=white!50!white, draw=black] (0,1.5) circle (3pt);
	
	\filldraw[fill=blue!50!white, draw=black] (0.5,0) circle (3pt);
	\filldraw[fill=blue!50!white, draw=black] (0.5,0.5) circle (3pt);
	\filldraw[fill=blue!50!white, draw=black] (0.5,1) circle (3pt);
	\filldraw[fill=blue!50!white, draw=black] (0.5,1.5) circle (3pt);
	
	\filldraw[fill=red!80!white, draw=black] (1,0) circle (3pt);
	\filldraw[fill=red!80!white, draw=black] (1,0.5) circle (3pt);
	\filldraw[fill=red!80!white, draw=black] (1,1) circle (3pt);
	\filldraw[fill=red!80!white, draw=black] (1,1.5) circle (3pt);

	\filldraw[fill=white!50!white, draw=black] (2,0) circle (3pt);
	\filldraw[fill=white!50!white, draw=black] (2,0.5) circle (3pt);
	\filldraw[fill=white!50!white, draw=black] (2,1) circle (3pt);
	\filldraw[fill=white!50!white, draw=black] (2,1.5) circle (3pt);
	
	\filldraw[fill=blue!50!white, draw=black] (2.5,0) circle (3pt);
	\filldraw[fill=blue!50!white, draw=black] (2.5,0.5) circle (3pt);
	\filldraw[fill=blue!50!white, draw=black] (2.5,1) circle (3pt);
	\filldraw[fill=blue!50!white, draw=black] (2.5,1.5) circle (3pt);
	
	\filldraw[fill=red!80!white, draw=black] (3,0) circle (3pt);
	\filldraw[fill=red!80!white, draw=black] (3,0.5) circle (3pt);
	\filldraw[fill=green, draw=black] (3,1) circle (3pt);
	\filldraw[fill=green, draw=black] (3,1.5) circle (3pt);

	 \filldraw[fill=white!50!white, draw=black] (4,0) circle (3pt);
	 \filldraw[fill=white!50!white, draw=black] (4,0.5) circle (3pt);
	 \filldraw[fill=white!50!white, draw=black] (4,1) circle (3pt);
	 \filldraw[fill=white!50!white, draw=black] (4,1.5) circle (3pt);
	
	 \filldraw[fill=green, draw=black] (4.5,0) circle (3pt);
	 \filldraw[fill=blue!50!white, draw=black] (4.5,0.5) circle (3pt);
	 \filldraw[fill=blue!50!white, draw=black] (4.5,1) circle (3pt);
	 \filldraw[fill=blue!50!white, draw=black] (4.5,1.5) circle (3pt);
	
	 \filldraw[fill=red!80!white, draw=black] (5,0) circle (3pt);
	 \filldraw[fill=red!80!white, draw=black] (5,0.5) circle (3pt);
	\filldraw[fill=green, draw=black] (5,1) circle (3pt);
	\filldraw[fill=green, draw=black] (5,1.5) circle (3pt);

	 \filldraw[fill=white!50!white, draw=black] (6,0) circle (3pt);
	 \filldraw[fill=white!50!white, draw=black] (6,0.5) circle (3pt);
	 \filldraw[fill=white!50!white, draw=black] (6,1) circle (3pt);
	 \filldraw[fill=white!50!white, draw=black] (6,1.5) circle (3pt);
	
	 \filldraw[fill=green, draw=black] (6.5,0) circle (3pt);
	 \filldraw[fill=green, draw=black] (6.5,0.5) circle (3pt);
	 \filldraw[fill=blue!50!white, draw=black] (6.5,1) circle (3pt);
	 \filldraw[fill=blue!50!white, draw=black] (6.5,1.5) circle (3pt);
	
	 \filldraw[fill=red!80!white, draw=black] (7,0) circle (3pt);
	 \filldraw[fill=red!80!white, draw=black] (7,0.5) circle (3pt);
	 \filldraw[fill=orange!30!white, draw=black] (7,1) circle (3pt);
	 \filldraw[fill=orange!30!white, draw=black] (7,1.5) circle (3pt);

	 \filldraw[fill=white!50!white, draw=black] (8,0) circle (3pt);
	 \filldraw[fill=white!50!white, draw=black] (8,0.5) circle (3pt);
	 \filldraw[fill=white!50!white, draw=black] (8,1) circle (3pt);
	 \filldraw[fill=white!50!white, draw=black] (8,1.5) circle (3pt);
	
	 \filldraw[fill=green, draw=black] (8.5,0) circle (3pt);
	 \filldraw[fill=green, draw=black] (8.5,0.5) circle (3pt);
	 \filldraw[fill=blue!50!white, draw=black] (8.5,1) circle (3pt);
	 \filldraw[fill=blue!50!white, draw=black] (8.5,1.5) circle (3pt);
	
	 \filldraw[fill=red!80!white, draw=black] (9,0) circle (3pt);
	 \filldraw[fill=orange!30!white, draw=black] (9,0.5) circle (3pt);
	  \filldraw[fill=red!80!white, draw=black](9,1) circle (3pt);
	 \filldraw[fill=orange!30!white, draw=black] (9,1.5) circle (3pt);

	 \filldraw[fill=blue!50!white, draw=black] (10,0) circle (3pt);
	 \filldraw[fill=white!50!white, draw=black] (10,0.5) circle (3pt);
	 \filldraw[fill=white!50!white, draw=black] (10,1) circle (3pt);
	 \filldraw[fill=white!50!white, draw=black] (10,1.5) circle (3pt);
	
	 \filldraw[fill=green, draw=black] (10.5,0) circle (3pt);
	 \filldraw[fill=green, draw=black] (10.5,0.5) circle (3pt);
	 \filldraw[fill=blue!50!white, draw=black] (10.5,1) circle (3pt);
	 \filldraw[fill=blue!50!white, draw=black] (10.5,1.5) circle (3pt);
	
	 \filldraw[fill=red!80!white, draw=black] (11,0) circle (3pt);
	 \filldraw[fill=red!80!white, draw=black] (11,0.5) circle (3pt);	
	  \filldraw[fill=red!80!white, draw=black] (11,1) circle (3pt);
	 \filldraw[fill=green, draw=black] (11,1.5) circle (3pt);

	
    \draw (-0.5+1,-3.5) rectangle (11.5+1-2,-1);
    \draw[dashed] (1.5+1,-3.5) -- (1.5+1,-1);
    \draw[dashed] (3.5+1,-3.5) -- (3.5+1,-1);
    \draw[dashed] (5.5+1,-3.5) -- (5.5+1,-1);
    \draw[dashed] (7.5+1,-3.5) -- (7.5+1,-1);
    \node at (0.5+1,2.2-3) {$\sigma_7$};
    \node at (2.5+1,2.2-3) {$\sigma_8$};
    \node at (4.5+1,2.2-3) {$\sigma_{9}$};
    \node at (6.5+1,2.2-3) {$\sigma_{10}$};
    \node at (8.5+1,2.2-3) {$\sigma_{11}$};

	\filldraw[fill=green, draw=black] (0+1,-3) circle (3pt);
	\filldraw[fill=white, draw=black] (0+1,-2.5) circle (3pt);
	\filldraw[fill=white, draw=black] (0+1,-2) circle (3pt);
	\filldraw[fill=white, draw=black] (0+1,-1.5) circle (3pt);
	
	\filldraw[fill=red!80!white, draw=black] (0.5+1,-3) circle (3pt);
	\filldraw[fill=blue!50!white, draw=black] (0.5+1,-2.5) circle (3pt);
	\filldraw[fill=blue!50!white, draw=black] (0.5+1,-2) circle (3pt);
	\filldraw[fill=blue!50!white, draw=black] (0.5+1,-1.5) circle (3pt);
	
	\filldraw[fill=orange!30!white, draw=black] (1+1,-3) circle (3pt);
	\filldraw[fill=orange!30!white, draw=black] (1+1,-2.5) circle (3pt);
	\filldraw[fill=red!80!white, draw=black] (1+1,-2) circle (3pt);
	\filldraw[fill=green, draw=black] (1+1,-1.5) circle (3pt);

	\filldraw[fill=blue!50!white, draw=black] (0+2+1,-3) circle (3pt);
	\filldraw[fill=white, draw=black] (0+2+1,-2.5) circle (3pt);
	\filldraw[fill=white, draw=black] (0+2+1,-2) circle (3pt);
	\filldraw[fill=white, draw=black] (0+2+1,-1.5) circle (3pt);
	
	\filldraw[fill=red!80!white, draw=black] (0.5+2+1,-3) circle (3pt);
	\filldraw[fill=red!80!white, draw=black] (0.5+2+1,-2.5) circle (3pt);
	\filldraw[fill=blue!50!white, draw=black] (0.5+2+1,-2) circle (3pt);
	\filldraw[fill=blue!50!white, draw=black] (0.5+2+1,-1.5) circle (3pt);
	
	\filldraw[fill=green, draw=black] (1+2+1,-3) circle (3pt);
	\filldraw[fill=green, draw=black] (1+2+1,-2.5) circle (3pt);
	\filldraw[fill=orange!30!white, draw=black] (1+2+1,-2) circle (3pt);
	\filldraw[fill=orange!30!white, draw=black] (1+2+1,-1.5) circle (3pt);

	\filldraw[fill=blue!50!white, draw=black] (0+4+1,-3) circle (3pt);
	\filldraw[fill=white, draw=black] (0+4+1,-2.5) circle (3pt);
	\filldraw[fill=white, draw=black] (0+4+1,-2) circle (3pt);
	\filldraw[fill=white, draw=black] (0+4+1,-1.5) circle (3pt);
	
	\filldraw[fill=green, draw=black] (0.5+4+1,-3) circle (3pt);
	\filldraw[fill=green, draw=black] (0.5+4+1,-2.5) circle (3pt);
	\filldraw[fill=blue!50!white, draw=black] (0.5+4+1,-2) circle (3pt);
	\filldraw[fill=blue!50!white, draw=black] (0.5+4+1,-1.5) circle (3pt);
	
	\filldraw[fill=red!80!white, draw=black] (1+4+1,-3) circle (3pt);
	\filldraw[fill=orange!30!white, draw=black] (1+4+1,-2.5) circle (3pt);
	\filldraw[fill=red!80!white, draw=black] (1+4+1,-2) circle (3pt);
	\filldraw[fill=orange!30!white, draw=black] (1+4+1,-1.5) circle (3pt);			
	
	\filldraw[fill=white, draw=black] (0+6+1,-3) circle (3pt);
	\filldraw[fill=white, draw=black] (0+6+1,-2.5) circle (3pt);
	\filldraw[fill=blue!50!white, draw=black] (0+6+1,-2) circle (3pt);
	\filldraw[fill=blue!50!white, draw=black] (0+6+1,-1.5) circle (3pt);
	
	\filldraw[fill=green, draw=black] (0.5+6+1,-3) circle (3pt);
	\filldraw[fill=green, draw=black] (0.5+6+1,-2.5) circle (3pt);
	\filldraw[fill=red!80!white, draw=black] (0.5+6+1,-2) circle (3pt);
	\filldraw[fill=red!80!white, draw=black] (0.5+6+1,-1.5) circle (3pt);
	
	\filldraw[fill=violet!80!white, draw=black] (1+6+1,-3) circle (3pt);
	\filldraw[fill=orange!30!white, draw=black] (1+6+1,-2.5) circle (3pt);
	\filldraw[fill=orange!30!white, draw=black] (1+6+1,-2) circle (3pt);
	\filldraw[fill=violet!80!white, draw=black]  (1+6+1,-1.5) circle (3pt);

	\filldraw[fill=white, draw=black] (2+6+1,-3) circle (3pt);
	\filldraw[fill=white, draw=black] (2+6+1,-2.5) circle (3pt);
	\filldraw[fill=blue!50!white, draw=black] (2+6+1,-2) circle (3pt);
	\filldraw[fill=blue!50!white, draw=black] (2+6+1,-1.5) circle (3pt);
	
	\filldraw[fill=red!80!white, draw=black] (2.5+6+1,-3) circle (3pt);
	\filldraw[fill=green, draw=black] (2.5+6+1,-2.5) circle (3pt);
	\filldraw[fill=green, draw=black] (2.5+6+1,-2) circle (3pt);
	\filldraw[fill=red!80!white, draw=black] (2.5+6+1,-1.5) circle (3pt);
	
	\filldraw[fill=orange!30!white, draw=black] (3+6+1,-3) circle (3pt);
	\filldraw[fill=violet!80!white, draw=black] (3+6+1,-2.5) circle (3pt);
	\filldraw[fill=orange!30!white, draw=black] (3+6+1,-2) circle (3pt);
	\filldraw[fill=violet!80!white, draw=black]  (3+6+1,-1.5) circle (3pt);

    \end{tikzpicture}
    \end{center}
    \caption{Illustration of the partition in $\Pi_{\geq 2}^{\rm con}(3,3,3,3)$. }
    \label{fig:proofd=3c}
    \end{figure}

Finally, we deal with the term $M_{3,3}(f^{(0)})$. This requires to consider the partitions in $\Pi_{\geq 2}^{\rm con}(3,3,3,3)$. Up to renumbering of the elements there are precisely $11$ partitions $\sigma_1,\ldots,\sigma_{11}$ of this type and they are shown in Figure \ref{fig:proofd=3c}. The analysis of the resulting integrals works the same way as above. Especially, we use once again systematically that $\cH^0(H_1\cap H_2\cap H_3\cap B_r)\leq 1$ for $\mu_2^3$-almost all $(H_1,H_2,H_3)\in A_h(3,2)^3$ and apply this to the term corresponding to the last row of each of the partitions. This leaves us with integrals over
\begin{align*}
&\cH^0(\text{intersection of the  3 planes corresponding to the first row})\\
&\times\cH^0(\text{intersection of the  3 planes corresponding to the second row})\\
&\quad\times\cH^0(\text{intersection of the  3 planes corresponding to the third row}),
\end{align*}
which in turn can be bounded using Crofton's formula,
Lemma \ref{lem:lines_intersecting_ball_inequality} and Lemma \ref{upperboundL1}. For $\sigma_1$ this yields
\begin{align*}
&\int_{A_h(3,2)^3}\cH^0(H_1\cap H_2\cap H_3\cap B_r)^3\,\mu_2^3(d(H_1,H_2,H_3))\\
&\qquad = \int_{A_h(3,2)^3}\cH^0(H_1\cap H_2\cap H_3\cap B_r)\,\mu_2^3(d(H_1,H_2,H_3)) = c\,\cH^3(B_r) \leq c\,e^{2r},
\end{align*}
for $\sigma_2$ and $\sigma_3$ we obtain
\begin{align*}
&\int_{A_h(3,2)^4}\cH^0(H_1\cap H_2\cap H_3\cap B_r)^2\,\cH^0(H_1\cap H_2\cap H_4\cap B_r)\,\mu_2^4(d(H_1,H_2,H_3,H_4))\\
&\qquad = \int_{A_h(3,2)^4}\cH^0(H_1\cap H_2\cap H_3\cap B_r)\,\cH^0(H_1\cap H_2\cap H_4\cap B_r)\,\mu_2^4(d(H_1,H_2,H_3,H_4))\\
&\qquad =c\int_{A_h(3,2)^2}\cH^1(H_1\cap H_2\cap B_r)^2\,\mu_2^2(d(H_1,H_2)) \le  c\, I(1,1)\le c\, e^{2r},
\end{align*}
for $\sigma_4$ we obtain
\begin{align*}
&\int_{A_h(3,2)^5}\cH^0(H_1\cap H_2\cap H_3\cap B_r)^2\,\cH^0(H_1\cap H_4\cap H_5\cap B_r)\,\mu_2^5(d(H_1,H_2,H_3,H_4,H_5))\\
&\qquad = \int_{A_h(3,2)^5}\cH^0(H_1\cap H_2\cap H_3\cap B_r)\,\cH^0(H_1\cap H_4\cap H_5\cap B_r)\,\mu_2^5(d(H_1,H_2,H_3,H_4,H_5))\\
&\qquad =c\int_{A_h(3,2)}\cH^2(H_1\cap B_r)^2\,\mu_2(dH_1) \leq c\, g(2,2,3,r)\leq c\,re^{2r},
\end{align*}
for $\sigma_5$ we have
\begin{align*}
&\int_{A_h(3,2)^5}\cH^0(H_1\cap H_2\cap H_3\cap B_r)\,\cH^0(H_1\cap H_2\cap H_4\cap B_r)\\
&\hspace{4cm}\times\cH^0(H_1\cap H_3\cap H_5\cap B_r)\,\mu_2^5(d(H_1,H_2,H_3,H_4,H_5))\\
&\qquad\leq c\int_{A_h(3,2)}\cH^2(H_1\cap B_r)\,\cH^1(L_1(H_1)\cap B_r)^2\,\mu_2(dH_1)\le c\, I(1,2)\leq c\,e^{2r},
\end{align*}
for $\sigma_6$ we have
\begin{align*}
&\int_{A_h(3,2)^4}\cH^0(H_1\cap H_2\cap H_3\cap B_r)\,\cH^0(H_1\cap H_2\cap H_4\cap B_r)\\
&\hspace{4cm}\times\cH^0(H_1\cap H_3\cap H_4\cap B_r)\,\mu_2^4(d(H_1,H_2,H_3,H_4))\\
&\qquad\leq\int_{A_h(3,2)^4}\cH^0(H_1\cap H_2\cap H_3\cap B_r)\,\cH^0(H_1\cap H_2\cap H_4\cap B_r)\,\mu_2^4(d(H_1,H_2,H_3,H_4))\\
&\qquad=c\int_{A_h(3,2)^2}\cH^1(H_1\cap H_2\cap B_r)^{2}\,\mu_2^2(d(H_1,H_2)) \leq c\,e^{2r}
\end{align*}
by the same argument as for $\sigma_2$ and $\sigma_3$. For $\sigma_7$ we have
\begin{align*}
&\int_{A_h(3,2)^5}\cH^0(H_1\cap H_2\cap H_3\cap B_r)\,\cH^0(H_1\cap H_2\cap H_4\cap B_r)\\
&\hspace{4cm}\times\cH^0(H_1\cap H_2\cap H_5\cap B_r)\,\mu_2^5(d(H_1,H_2,H_3,H_4,H_5))\\
&\qquad = c\int_{A_h(3,2)^2}\cH^1(H_1\cap H_2\cap B_r)^3\,\mu_2(d(H_1,H_2))\\
&\qquad \leq c\int_{A_h(3,2)^2}\cH^1(H_1\cap H_2\cap B_r)\,\cH^1(L_1(H_1)\cap B_r)^2\,\mu_2^2(d(H_1,H_2))\leq c\, I(1,2) \leq c\,e^{2r},
\end{align*}
for $\sigma_8$ we obtain
\begin{align*}
&\int_{A_h(3,2)^5}\cH^0(H_1\cap H_2\cap H_3\cap B_r)^2\,\cH^0(H_1\cap H_4\cap H_5\cap B_r)\,\mu_2^5(d(H_1,H_2,H_3,H_4,H_5))\\
&\qquad = \int_{A_h(3,2)^5}\cH^0(H_1\cap H_2\cap H_3\cap B_r)\,\cH^0(H_1\cap H_4\cap H_5\cap B_r)\,\mu_2^5(d(H_1,H_2,H_3,H_4,H_5))\\
&\qquad=c\int_{A_h(3,2)}\cH^2(H_1\cap B_r)^2\,\mu_2(dH_1) \leq c\, g(2,2,3,r) \leq c\,re^{2r},
\end{align*}
for $\sigma_9$ we get
\begin{align*}
&\int_{A_h(3,2)^5}\cH^0(H_1\cap H_2\cap H_3\cap B_r)\,\cH^0(H_1\cap H_2\cap H_4\cap B_r)\\
&\hspace{4cm}\times\cH^0(H_1\cap H_3\cap H_5\cap B_r)\,\mu_2^5(d(H_1,H_2,H_3,H_4,H_5))\\
&\qquad\leq  c\int_{A_h(3,2)^3}\cH^1(H_1\cap H_2\cap B_r)\,\cH^1(H_1\cap H_3\cap B_r)\,\mu_2^3(d(H_1,H_2,H_3))\\
&\qquad = c\int_{A_h(3,2)}\cH^2(H_1\cap B_r)^2\,\mu_2(dH_1)\leq c\, g(2,2,3,r) \leq c\,re^{2r},
\end{align*}
for $\sigma_{10}$ we obtain
\begin{align*}
&\int_{A_h(3,2)^6}\cH^0(H_1\cap H_2\cap H_3\cap B_r)\,\cH^0(H_1\cap H_2\cap H_4\cap B_r)\\
&\hspace{4cm}\times\cH^0(H_4\cap H_5\cap H_6\cap B_r)\,\mu_2^6(d(H_1,\ldots,H_6))\\
&\qquad \leq c \int_{A_h(3,2)^4}\cH^0(H_1\cap H_2\cap H_3\cap B_r)\,\cH^2(H_4\cap B_r)\,\mu_2^4(d(H_1,H_2,H_3,H_4))\\
&\qquad
=c\,\cH^3(B_r)^2\leq c\,e^{4r}
\end{align*}
and, finally, for $\sigma_{11}$ we have
\begin{align*}
&\int_{A_h(3,2)^6}\cH^0(H_1\cap H_2\cap H_3\cap B_r)\,\cH^0(H_1\cap H_4\cap H_5\cap B_r)\\
&\hspace{4cm}\times\cH^0(H_3\cap H_4\cap H_6\cap B_r)\,\mu_2^6(d(H_1,\ldots,H_6))\\
&\qquad = c\int_{A_h(3,2)^3}\cH^{1}(H_1 \cap H_3 \cap B_r) \,\cH^{1}(H_1 \cap H_4 \cap B_r) \,\cH^{1}(H_3 \cap H_4 \cap B_r) \,\mu_2^3(d(H_1,H_3,H_4))\\
&\qquad \leq c\int_{A_h(3,2)^3}\cH^{1}(L_1(H_1) \cap B_r)^{2} \,\cH^{1}(H_3 \cap H_4 \cap B_r) \,\mu_2^3(d(H_1,H_3,H_4))\\
&\qquad
=c\,\cH^3(B_r)\,I(0,2)\leq c\,e^{4r}.
\end{align*}
We thus conclude that $M_{3,3}(f^{(0)})\leq c\,t^6\,(6e^{2r}+3\,re^{2r}+2e^{4r})\leq c\,t^6\,e^{4r}$. An application of the upper bounds for $M_{u,v}(f^{(0)})$ with $(u,v)\in\{(1,1),(1,2),(1,3),(2,2),(2,3),(3,3)\}$ and the lower bound for the variance  from Lemma \ref{lem:VarianceBoundd=3} in \eqref{eq:Kolmogorov} shows that
\begin{equation}\label{eq:Boundd=3i=0}
  d\left(\frac{F_{r,t}^{(0)}-\mathbb{E}F_{r,t}^{(0)}}{\sqrt{\Var(F_{r,t}^{(0)})}},N \right) \leq c\, \frac{3\sqrt{t^9e^{4r}}+3\sqrt{t^9re^{4r}}}{t^5c^{(1)}(3)e^{2r}r} \leq c\,t^{-1/2}\,r^{-1/2}
\end{equation}
  and the proof of Theorem \ref{thm:CLTrToInfinity} (b) is complete.
 \hfill $\Box$

\subsubsection{The higher dimensional cases $d\geq 4$: Proof of Theorem \ref{thm:CLTrToInfinity} (c)}

In order to show that for $d\geq 4$ and $i=d-1$ and for $d\ge 7$ and $i\in\{0,\ldots,d-1\}$ non of the centred and normalized functionals $F_{r,t}^{(i)}$ converges in distribution to a Gaussian random variable, as $r\to\infty$, we will argue that the fourth cumulant
$$
\cum_4:=\E\left(\widetilde{F_{r,t}^{(i)}}\right)^4-3,\qquad\qquad \widetilde{F_{r,t}^{(i)}}:={F_{r,t}^{(i)}-\E F_{r,t}^{(i)}\over \sqrt{\Var(F_{r,t}^{(i)})}}
$$
does not converge to zero, which is the value of the fourth cumulant of a standard Gaussian random variable. We start with the following crucial, but rather technical result, which is based on the formula \eqref{eq:MomentsUstatistic} for the centred moments of a Poisson U-statistic.

\begin{Lemma}\label{lem:UnifInt}
Let $d\geq 4$, $i\in\{0,1,\ldots,d-1\}$ and $t\ge t_0>0$. If $d \in \{4,5,6\}$ and $i=d-1$ or if $d\ge 7$, then
$$
\sup_{r\geq 1}\E\left(\widetilde{F_{r,t}^{(i)}}\right)^5<\infty.
$$
\end{Lemma}
\begin{figure}[t]
\begin{center}
\begin{tikzpicture}
\filldraw
(0,0) circle (3pt)
(0,0.5) circle (3pt)
(0,1) circle (3pt)
(0,1.5) circle (3pt)
(0,2) circle (3pt)
(2,0) circle (3pt)
(2,0.5) circle (3pt)
(2,1) circle (3pt)
(2,1.5) circle (3pt)
(2,2) circle (3pt);

\filldraw
(6,0) circle (3pt)
(6,0.5) circle (3pt)
(6,1) circle (3pt)
(6,1.5) circle (3pt)
(6,2) circle (3pt)
(6.5,0) circle (3pt)
(6.5,0.5) circle (3pt)
(6.5,1) circle (3pt)
(6.5,1.5) circle (3pt)
(6.5,2) circle (3pt)
(7,0) circle (3pt)
(7,0.5) circle (3pt)
(7,1) circle (3pt)
(7,1.5) circle (3pt)
(7,2) circle (3pt)
(7.5,0) circle (3pt)
(7.5,0.5) circle (3pt)
(7.5,1) circle (3pt)
(7.5,1.5) circle (3pt)
(7.5,2) circle (3pt);

\draw[rotate=90] (1,0) ellipse (38pt and 10pt);

\draw[rotate=90] (1.5,-2) ellipse (21pt and 7pt);
\draw[rotate=90] (0.25,-2) ellipse (12pt and 6pt);

\draw[rotate=90] (1.5,-6) ellipse (21pt and 7pt);
\draw[rotate=90] (0.5,-6.5) ellipse (21pt and 7pt);
\draw[rotate=90] (1.75,-7) ellipse (14pt and 7pt);
\draw[rotate=90] (1,-7.5) ellipse (21pt and 7pt);

\draw (9,2.25) -- (11,2.25);
\draw (10,2.75) -- (10,0);

\node at (9.5,2.5) {$b(\,\cdot\,)$};
\node at (9.5,2) {$1$};
\node at (9.5,1.5) {$1$};
\node at (9.5,1) {$2$};
\node at (9.5,0.5) {$1$};
\node at (9.5,0) {$0$};

\node at (10.5,2.5) {$c(\,\cdot\,)$};
\node at (10.5,2) {$2$};
\node at (10.5,1.5) {$2$};
\node at (10.5,1) {$1$};
\node at (10.5,0.5) {$3$};
\node at (10.5,0) {$4$};

\end{tikzpicture}
\end{center}
\caption{Left panel: The two types of (sub-)partitions in $\Pi_{\geq 2}^{**}(1,1,1,1,1)$. Right panel: Example of a sub-partition $\sigma$ from $\Pi_{\geq 2}^{**}(4,4,4,4,4)$ with $m(\sigma)=3$.}
\label{fig:SurfaceArea}
\end{figure}
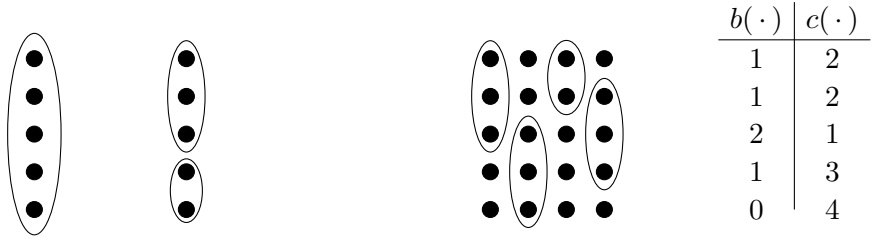

\begin{proof}
We start by explaining our method by considering the case $i=d-1$. In this situation
\begin{align*}
\E\left(\widetilde{F_{r,t}^{(d-1)}}\right)^5 = {\E\big(F_{r,t}^{(d-1)}-\E F_{r,t}^{(d-1)}\big)^5\over(\Var(F_{r,t}^{(d-1)}))^{5/2}} \leq c\,{\E\big(F_{r,t}^{(d-1)}-\E F_{r,t}^{(d-1)}\big)^5\over e^{5r(d-2)}},
\end{align*}
where we used the variance bound from Lemma \ref{lem:VarianceBoundd>=4}, which is available since $t\ge t_0$ and $r\geq 1$. For the centred fifth moment,  \eqref{eq:MomentsUstatistic} implies that
$$
\E\big(F_{r,t}^{(d-1)}-\E F_{r,t}^{(d-1)}\big)^5 = \sum_{\sigma\in\Pi_{\geq 2}^{**}(1,1,1,1,1)}t^{5-|\sigma|+\|\sigma\|}\int_{A_h(d,d-1)^{5-|\sigma|-\|\sigma\|}}\big((f^{(d-1)})^{\otimes 5}\big)_\sigma\,d\mu_{d-1}^{5-|\sigma|+\|\sigma\|}.
$$
The set $\Pi_{\geq 2}^{**}(1,1,1,1,1)$ consists only of two types of sub-partitions of $\{1,2,3,4,5\}$, which are actually partitions, see Figure \ref{fig:SurfaceArea}. The first type only consists of one partition, namely the trivial partition,  only containing the single block $\{1,2,3,4,5\}$. The second type contains ${5\choose 2}=10$ partitions having precisely two blocks, one of size $2$ and the other of type $3$. Since the integrals corresponding to these partitions all yield the same contribution, we can restrict our computations to $\{\{1,2,3\},\{4,5\}\}$, for example. Thus,
\begin{align*}
\E\big(F_{r,t}^{(d-1)}-\E F_{r,t}^{(d-1)}\big)^5 &= t^9\int_{A_h(d,d-1)}\cH^{d-1}(H\cap B_r)^5\,\mu_{d-1}(dH)\\
&\quad+10t^8\int_{A_h(d,d-1)^2}\cH^{d-1}(H_1\cap B_r)^3\,\cH^{d-1}(H_2\cap B_r)^2\,\mu_{d-1}^2(d(H_1,H_2)).
\end{align*}
By Lemma \ref{lem:lines_intersecting_ball_inequality} we have
$$
\int_{A_h(d,d-1)}\cH^{d-1}(H\cap B_r)^5\,\mu_{d-1}(dH)\le c\, g(d-1,5,d,r)\le c\, e^{5r(d-2)},
$$
since $5(d-2)-(d-1)=4d-9>0$.  Again by Lemma \ref{lem:lines_intersecting_ball_inequality} we obtain
\begin{align*}
&\int_{A_h(d,d-1)^2}\cH^{d-1}(H_1\cap B_r)^3\,\cH^{d-1}(H_2\cap B_r)^2\,\mu_{d-1}^2(d(H_1,H_2))\\
&\qquad
\le c\,  g(d-1,3,d,r)\, g(d-1,2,d,r)\le c\, e^{3r(d-2)}e^{2r(d-2)} \le c\, e^{5r(d-2)},
\end{align*}
since $d>3$. Thus we get
$$
\sup_{r\geq 1}\E\left(\widetilde{F_{r,t}^{(i)}}\right)^5 \leq c\,\sup_{r\geq 1}{e^{5r(d-2)}+e^{5r(d-2)}\over e^{5r(d-2)}} = c<\infty.
$$
This proves the claim for $i=d-1$.

\medskip

Now we fix $i\in\{0,1,\ldots,d-2\}$ arbitrarily and assume that $d\geq 7$. Furthermore, we fix an arbitrary partition $\sigma\in\Pi_{\geq 2}^{**}(d-i,d-i,d-i,d-i,d-i)$. We denote by $m(\sigma)\in\{2,3,4,5\}$ the size of the maximal block of $\sigma$ and represent $\sigma$ as a diagram. The elements of this diagram are labelled $a_{p,q}$. Here, $p\in\{1,\ldots,5\}$ represents the row number and  $q\in\{1,\ldots,d-i\}$ stands for the column number. Without loss of generality we can and will assume that the maximal block of $\sigma$ sits in the left upper corner of the diagram of $\sigma$, that is, the maximal block is of the form $\{a_{1,1},\ldots,a_{m(\sigma),1}\}$. To each row $p\in\{1,\ldots,5\}$ we associate two numbers $b(p)$ and $c(p)$ in the following way. By $b(p)$ we denote the number of elements of row $p$  in position
$$(p,q)\in (\{1,\ldots,m(\sigma)\}\times \{2,\ldots,d-i\})\cup (\{m(\sigma)+1,\ldots,5\}\times \{1,\ldots,d-i\})$$
which are contained in a block of $\sigma$ that has at least one element in a row below $p$, and we let $c(p)$ be the number of elements in position $(p,q)$ (with the same restrictions as above) in row $p$ not contained in any block of $\sigma$ that has at least one element in a row below $p$, see Figure \ref{fig:SurfaceArea} for an example. Note that $b(5)=0$, $c(5)=d-i$ if $m(\sigma)<5$, and $c(p)=d-i-b(p)-1$ if $p\in\{1,\ldots,m(\sigma)\}$. Our task is to show that the integral (in symbolic notation)
\begin{align*}
\mathscr{I} :&= \int\cdots\int \Big(\big(f^{(i)}\big)^{\otimes 5}\Big)_\sigma\\
&=\int\cdots\int f^{(i)}(H_1, G_1,\ldots, G_{b(1)}, K_1,\ldots, K_{c(1)})\\
&\qquad\qquad\qquad\times f^{(i)}(\ldots)\,f^{(i)}(\ldots)\,f^{(i)}(\ldots)\,f^{(i)}(\ldots) \, \mu_{d-1}(d H_1) \ldots
\end{align*}
is bounded by a constant multiple of $e^{5(d-2)r}$, which is the order of $(\Var(F_{r,t}^{(i)}))^{5/2}$.
We first integrate with respect to the hyperplanes $K_1,\ldots,K_{c(1)}$, which do not appear in any of the arguments of the other four functions $f^{(i)}(\ldots)$. By Crofton's formula  this gives $c\,\cH^{d-1-b(1)}(B_r\cap H_1\cap G_1\cap\ldots\cap G_{b(1)})$. Now we replace $H_1\cap G_1\cap\ldots\cap G_{b(1)}$ by a $(d-1-b(1))$-dimensional subspace $L_{d-1-b(1)}(s_1)$ having distance $s_1=d_h(H_1,p)$ from $p$. This leads to
\begin{align}\label{eq:UnifInt_1}
\cH^{d-1-b(1)}(B_r\cap H_1\cap G_1\cap\ldots\cap G_{b(1)})\leq\cH^{d-1-b(1)}(B_r\cap L_{d-1-b(1)}(s_1)).
\end{align}
Then $G_1,\ldots,G_{b(1)}$ are active integration variables for rows below the first row.
Repeating the same argument for $p=2,\ldots,m(\sigma)$, we arrive at (again in symbolic notation)
\begin{align*}
\mathscr{I} &\leq c \int \cdots \int \cH^{d-1-b(1)}(B_r\cap L_{d-1-b(1)}(s_1))\cdots \cH^{d-1-b(m(\sigma))}(B_r\cap L_{d-1-b(m(\sigma))}(s_1))\\
&\hspace{5cm}\times f^{(i)}(\ldots)\cdots f^{(i)}(\ldots)\, \mu_{d-1}(dH_1) \ldots ,
\end{align*}
where $f^{(i)}(\ldots)$ appears $5-m(\sigma)$ times. From now on we distinguish the following two cases:
\begin{itemize}
\item[(a)] there is no block that contains precisely two elements from the rows below $m(\sigma)$,
\item[(b)] there exists a block that contains precisely two elements from the rows below $m(\sigma)$.
\end{itemize}
We start by treating case (a). If $m(\sigma)=2$, then all blocks of $\sigma$ have two elements. In particular, no element of
row $p\ge 3$ can be in a ($2$-element) block with another element in a block below. Hence, we have $c(p)=d-i$ for $p\ge 3$. If $m(\sigma)=3$, then an element of row $p=4$ cannot be in a common block with an element of row $5$ due to assumption (a). Hence $c(4)=c(5)=d-i$. This shows that $c(p)=d-i$  for $p\in\{m(\sigma)+1,\ldots,5\}$. We can thus carry out the $5-m(\sigma)$ integrals involving the functions $f^{(i)}(\ldots)$, which by Crofton's formula and Lemma \ref{lem:lines_intersecting_ball_inequality} leads to the upper bound
\begin{align}\label{eq:NCLT1}
\cH^d(B_r)^{5-m(\sigma)} \leq c\,e^{(5-m(\sigma))(d-1)r}.
\end{align}
The only remaining integral in $\mathscr{I}$ is
\begin{align*}
\mathscr{J}:=\int_0^r\cosh^{d-1}(s)\,\cH^{d-1-b(1)}(B_r\cap L_{d-1-b(1)}(s))\cdots \cH^{d-1-b(m(\sigma))}(B_r\cap L_{d-1-b(m(\sigma))}(s))\,ds.
\end{align*}
To proceed, we define for $p\in\{1,\ldots,m(\sigma)\}$ the function
\begin{align*}
g_p(s):&=e^{-r(d-2)}\cdot\begin{cases}
e^{(r-s)(d-2-b(p))}&:d-1-b(p) \geq 2,\\
r-s+\log(2)&: d-1-b(p) = 1,\\
1&: d-1-b(p) = 0.
\end{cases}
\end{align*}
Then, Lemma \ref{lem:CoshBound}, \eqref{eq:volume_ball} and Lemma \ref{lem:H_d_bounds} imply that
\begin{align}\label{eq:NCLT2}
\mathscr{J} \leq c \,e^{m(\sigma)(d-2)r}\,\mathscr{K}\qquad\text{with}\qquad\mathscr{K}:=\int_0^r \cosh^{d-1}(s)\,g_1(s)\cdots g_{m(\sigma)}(s)\,ds.
\end{align}
We let
\begin{align*}
Z_{01} &:= \{p\in\{1,\ldots,m(\sigma)\}:d-1-b(p)\in\{0,1\}\},\\
Z_{1} &:= \{p\in\{1,\ldots,m(\sigma)\}:d-1-b(p)=1\}.
\end{align*}
Then
\begin{align}
\mathscr{K} &\leq
c\,e^{-r(d-2)|Z_{01}|-r\sum_{p=1, p\notin Z_{01}}^{m(\sigma)}b(p)}
\int_0^r (r-s+\log(2))^{|Z_1|}\,e^{sE}\,ds,\label{eq:NCLT3}
\end{align}
where the exponent $E$ is given by
$$
E := (d-1)-(d-2)(m(\sigma)-|Z_{01}|)+\sum_{p=1, p\notin Z_{01}}^{m(\sigma)}b(p).
$$
If $E<0$ the integral in \eqref{eq:NCLT3} is bounded by a constant times $r^{|Z_1|}$. In view of \eqref{eq:NCLT1} and \eqref{eq:NCLT2} we conclude that
\begin{align}\label{eq:NCLT4}
\mathscr{I} \leq c\,e^{(5-m(\sigma))(d-1)r}\,e^{m(\sigma)(d-2)r}\,e^{-(d-2)|Z_{01}|r-r\sum_{p=1, p\notin Z_{01}}^{m(\sigma)}b(p)}r^{|Z_1|}.
\end{align}
In order to bound $\mathscr{I}$ from above by a constant times $e^{5(d-2)r}$, we use the decomposition
\begin{align}\label{eq:NCLT5}
e^{5(d-2)r} = e^{(5-m(\sigma))(d-1)r}\,e^{m(\sigma)(d-2)r}\,e^{-(5-m(\sigma))r}.
\end{align}
A comparison of the exponents in \eqref{eq:NCLT4} and \eqref{eq:NCLT5} shows that if $E<0$,
then it is sufficient to prove that
$$
(d-2)|Z_{01}|+\sum_{p=1, p\notin Z_{01}}^{m(\sigma)}b(p)
\begin{cases}
\geq 5-m(\sigma)&: \text{if } |Z_1|=0,\\
> 5-m(\sigma)&: \text{if } |Z_1|>0.
\end{cases}
$$
If $|Z_{01}|>0$, then $(d-2)|Z_{01}|\ge 4> 5-m(\sigma)$ for $d\ge 6$.
If $|Z_{01}|=0$, then also $|Z_1|=0$, and in this case it is sufficient to show that $\sum_{p=1}^{m(\sigma)}b(p)\ge 5-m(\sigma)$. To see this, note that, for any $m(\sigma)\in\{2,\ldots,5\}$, under condition (a) we know that for $5-m(\sigma)$ of the positions $(p,q)\in\{1,\ldots,m(\sigma)\}\times \{2,\ldots,d-i\}$
there has to be a block containing the element at $(p,q)$ and exactly one element at $(p',q')\in\{m(\sigma)+1,\ldots,5\}\times \{1,\ldots,d-i\}$, since each row has to be visited by some block. But this implies  the required inequality.

 Next, suppose that $E=0$. Then the integral in \eqref{eq:NCLT3} is bounded by a polynomial in $r$ of degree at most $|Z_1|+1$ and another comparison of exponents in \eqref{eq:NCLT4} and \eqref{eq:NCLT5} implies that in this case we need to prove that
\begin{align}\label{eq:NCLT6}
(d-2)|Z_{01}| + \sum_{p=1, p\notin Z_{01}}^{m(\sigma)}b(p) > 5-m(\sigma).
\end{align}
Using the assumption that $E=0$, we see that in this case
\begin{align*}
(d-2)|Z_{01}| + \sum_{p=1, p\notin Z_{01}}^{m(\sigma)}b(p) 
&= m(\sigma)(d-2)-(d-1).
\end{align*}
This shows that the inequality in \eqref{eq:NCLT6} is equivalent to $(d-1)(m(\sigma)-1) > 5$,
which is always satisfied for $d\geq 7$.

Finally, we suppose that $E>0$ in which case a comparison of the exponents in \eqref{eq:NCLT4} and \eqref{eq:NCLT5} shows that we have to verify that
\begin{align*}
(d-2)|Z_{01}|+ \sum_{p=1, p\notin Z_{01}}^{m(\sigma)}b(p) - (d-1) + (d-2)(m(\sigma)-|Z_{01}|) - \sum_{p=1, p\notin Z_{01}}^{m(\sigma)}b(p) \geq 5-m(\sigma).
\end{align*}
After simplification, this is equivalent to $(d-1)(m(\sigma)-1) \geq 5$,
which  holds for $d\geq 6$. This completes the argument in case (a) for $d\ge 7$.

\medskip

We turn now to case (b), where we have to distinguish the sub-cases $m(\sigma)=2$ and $m(\sigma)=3$. We start with the case $m(\sigma)=2$.  Then, arguing as at the beginning of the proof for case (a), we have
$$
\mathscr{I}\le c\, \mathscr{I}_1\mathscr{I}_2\cH^d(B_r)
$$
with
\begin{align*}
  \mathscr{J}_j:= & \int_0^r \cosh^{d-1}(s)\,\cH^{d-1-\bar b(2j-1)}(B_r\cap L_{d-1-\bar b(2j-1)}(s))\,\cH^{d-1-\bar b(2j)}(B_r\cap L_{d-1-\bar b(2j)}(s))\,ds
\end{align*}
for $j\in\{1,2\}$, where $\bar b(i)=b(i)$ for $i\in\{1,2,4\}$ and $\bar b(3)=b(3)-1\ge 0$. Moreover, without loss of generality, we can assume that $b(1)\ge 1$.
Similarly to \eqref{eq:NCLT2}, for $j \in \{1,2\}$ we get
$$\mathscr{J}_j \leq e^{2(d-2)r}\,\mathscr{K}_j \qquad\text{with}\qquad\mathscr{K}_j:=\int_0^r \cosh^{d-1}(s)\,g_{2j-1}(s)\,g_{2j}(s)\,ds. $$
For $j \in \{1,2\}$ we let
\begin{align*}
Z_{01}^{j} &:= \{p\in\{2j-1,2j\}:d-1-\bar b(p)\in\{0,1\}\},\\
Z_{1}^{j} &:= \{p\in\{2j-1,2j\}:d-1-\bar b(p)=1\}.
\end{align*}
Then
\begin{align}
\mathscr{K}_j &\leq
c\,e^{-r(d-2)|Z_{01}^{j}|-r\sum_{p=2j-1, p\notin Z_{01}^{j}}^{2j}\bar b(p)}\int_0^r (r-s+\log(2))^{|Z_1^{j}|}\,e^{sE_j}\,ds,\label{eq:NCLT7}
\end{align}
where the exponents $E_j$, $j\in\{1,2\}$, are given by
\begin{align*}
  E_j :=& (d-1)-(d-2)(2-|Z_{01}^{j}|)+\sum_{p=2j-1, p\notin Z_{01}^{j}}^{2j}\bar b(p).
\end{align*}
We will show that $\mathscr{K}_1$ is bounded by a constant multiple of $e^{-r}$ and $\mathscr{K}_2$ by a constant.
Then we can conclude that
$$
\mathscr{I}\le c\, e^{(d-1)r}\mathscr{I}_1\mathscr{I}_2\le c\, e^{(d-1)r}e^{4(d-2)r}e^{-r}\le e^{5(d-2)r}.
$$
We first consider $\mathscr{K}_1$.  For $E_1<0$ the integral in \eqref{eq:NCLT7} is bounded by a constant multiple of $r^{|Z_1^1|}$. Therefore it is sufficient to compare the exponents and to show that
$$(d-2)|Z_{01}^{1}|+\sum_{p=1, p\notin Z_{01}^{1}}^{2}b(p) \begin{cases}
\ge 1&:|Z^1_1|=0,\\
> 1&:|Z^1_1|>0.
\end{cases}
$$
Since $b(1)\ge 1$ and $d\ge 4$, this is satisfied.

Next, suppose that $E_1=0$. In this case, the integral in \eqref{eq:NCLT7} is bounded by a polynomial in $r$ and we have to show the inequality
\begin{align}\label{eq:NCLT8}
(d-2)|Z_{01}^{1}|+\sum_{p=1, p\notin Z_{01}^{1}}^{2}b(p) > 1.
\end{align}
Using the assumption that $E_1=0$, we get
$$(d-2)|Z_{01}^{1}|+\sum_{p=1, p\notin Z_{01}^{1}}^{2}b(p)=-(d-1)+2(d-2)=d-3.$$
Hence \eqref{eq:NCLT8} is true for $d\geq 5$.

Finally, we suppose that $E_1>0$.  Then we have to show that
$$(d-2)|Z_{01}^{1}|+\sum_{p=1, p\notin Z_{01}^{1}}^{2}b(p)-(d-1)+(d-2)(2-|Z_{01}|^{j})-\sum_{p=1, p\notin Z_{01}^{1}}^{2}b(p)\geq 1.$$
After simplifications this is equivalent to $d\geq 4$.

Now we prove that $\nonumber\mathscr{K}_2$ is bounded by a constant. For $E_2<0$, a comparison of the exponents in \eqref{eq:NCLT7} shows that we need that
$$(d-2)|Z_{01}^{2}|+ \sum_{p=3, p \notin Z_{01}^{2}}^{4}\bar b(p)\begin{cases}
\geq 0& :|Z^2_{1}|=0,\\
> 0& :|Z^2_{1}|>0,
\end{cases}
$$
which is trivially satisfied.

For $E_2=0$ the required inequality is
$$(d-2)|Z_{01}^{2}|+ \sum_{p=3, p \notin Z_{01}^{2}}^{4}\bar b(p)> 0,$$
which is equivalent to $-(d-1)+2(d-2)>0$, that is, to $d\ge 4$.

Finally, if $E_2>0$ then we have to verify that
$$(d-2)|Z_{01}^{2}|+ \sum_{p=3, p \notin Z_{01}^{2}}^{4}\bar b(p)-(d-1)+(d-2)(2-|Z_{01}|^{2})-\sum_{p=3, p \notin Z_{01}^{2}}^{4}\bar b(p)\geq 0.$$
Again simplification yields that this is equivalent to $d\ge 3$.

\medskip
Now we turn to the case $m(\sigma)=3$. Then we have
$$
\mathscr{I}\le c\, \mathscr{I}_3\mathscr{I}_4
$$
with
\begin{align*}
\mathscr{I}_3&:=\int_0^r\cosh^{d-1}(s)\prod_{i=1}^3\cH^{d-1-b(i)}(B_r\cap L_{d-1-b(i)}(s))\, ds,\\
\mathscr{I}_4&:= \int_0^r\cosh^{d-1}(s)\prod_{i=4}^5\cH^{d-1-\bar b(i)}(B_r\cap L_{d-1-\bar b(i)}(s))\, ds,
\end{align*}
where $0\le \bar b(4):=b(4)-1\le d-i-1\le d-1$ and $\bar b(5)=0$. We will prove that $\mathscr{I}_3\le c\, e^{3(d-2)r}$ and $\mathscr{I}_4\le c\, e^{2(d-2)r}$, which in turn proves that $\mathscr{I}\le c\, e^{5(d-2)r}$.

As in the proof of case (a) (and  for $m(\sigma)=3$ there), we obtain
$$
\mathscr{I}_3\le c\, e^{3(d-2)r}\mathscr{K}_3\qquad\text{with}\qquad \mathscr{K}_3:=\int_0^r\cosh^{d-1}(s) g_1(s)g_2(s)g_3(s)\, ds.
$$
We show that $\mathscr{K}_3\le c$. For this, we proceed as before and obtain
\begin{align*}
\mathscr{K}_3 &\leq
c\,e^{-r(d-2)|Z_{01}^3|-r\sum_{p=1, p\notin Z_{01}^3}^3 b(p)}
\int_0^r (r-s+\log(2))^{|Z_1^3|}\,e^{sE_3}\,ds,
\end{align*}
where
$$
Z_{01}^3 := \{p\in\{1,\ldots,3\}:d-1-b(p)\in\{0,1\}\}, \quad
Z_{1}^3 := \{p\in\{1,\ldots,3:d-1-b(p)=1\}
$$
and
$$
E_3:=(d-1)-(d-2)(3-|Z^3_{01}|)+\sum_{p=1,p\notin Z_{01}^3}^3b(p).
$$
If $E_3\le 0$, then
$$
r^{|Z^3_{01}|}e^{-r(d-2)|Z^3_{01}|}e^{-r\sum_{p=1,p\notin Z_{01}^3}^3 b(p)}\le c
$$
provided that
$$
(d-2)|Z^3_{01}|+\sum_{p=1,p\notin Z_{01}^3}^3 b(p)  \begin{cases}
\ge 0&:|Z^3_1|=0,\\
>0 &:|Z^3_1|>0.
\end{cases}
$$
This is obviously true, since $|Z^3_{01}|\ge |Z^3_1|$ and $d\ge 4$. Hence, if $E_3\le 0$, then $\mathscr{K}_3\le c$.

If $E_3>0$,   then $\mathscr{K}_3\le c$ follows provided that
$$
(d-2)|Z^3_{01}|+\sum_{p=1,p\notin Z_{01}^3}^3 b(p)-E_3\ge 0.
$$
The latter is equivalent to $(d-2)3-(d-1)\ge 0$, that is, to $2d\ge 5$. Thus we have shown that $\mathscr{I}_3\le c\, e^{3(d-2)}$. In order to show that $\mathscr{I}_4\le c\, e^{2(d-2)}$, we distinguish several cases.

If $\bar b(4)<d-3$, then
\begin{align*}
\mathscr{I}_4&\le c\int_0^re^{s(d-1)}e^{(r-s)(d-2-\bar b(4))}e^{(r-s)(d-2)}\, ds\\
&\le c\,e^{(2(d-2)-\bar b(4))r}\int_0^re^{s(-d+3+\bar b(4))}\, ds\le c\, e^{2(d-2)r}.
\end{align*}

If $\bar b(4)=d-3$, then
$$
\mathscr{I}_4\le c\, e^{(2(d-2)-d+3)r}r=c\, r e^{r(d-1)}\le c\, e^{2(d-2)r},
$$
since $d-1< 2(d-2)$ for $d\ge 4$.

If $\bar b(4)=d-2$, then
$$
\mathscr{I}_4\le c\int_0^re^{s(d-1)}(r-s+\log(2))e^{(r-s)(d-2)}\, ds\le c\, e^{r(d-1)}.
$$

If $\bar b(4)=d-1$, then
$$
\mathscr{I}_4\le c\int_0^re^{s(d-1)} e^{(r-s)(d-2)}\, ds\le c\, e^{r(d-1)}.
$$
Thus in all cases  we have $\mathscr{I}_4\le c\, e^{2(d-2)r}$, which completes the proof.
\end{proof}

\begin{proof}[Proof of Theorem \ref{thm:CLTrToInfinity} (c)]
Let $d$ and $i$ be as in the statement of Theorem \ref{thm:CLTrToInfinity} (c), and suppose to the contrary that $\widetilde{F_{r,t}^{(i)}}$ converges in distribution, as $r\to\infty$, to a standard Gaussian random variable $N$. As a consequence of Lemma \ref{lem:UnifInt}, the family of random variables $\big((\widetilde{F_{r,t}^{(i)}})^4\big)_{r\geq 1}$ is uniformly integrable, which implies that $\E (\widetilde{F_{r,t}^{(i)}})^4\to \E N^4=3$, as $r\to\infty$. Thus, we would also have that
\begin{align}\label{eq:13-06-1}
\cum_4 = \E \left(\widetilde{F_{r,t}^{(i)}}\right)^4-3 \to \E N^4 - 3 = 0,
\end{align}
as $r\to\infty$. On the other hand, from {\cite[page 112]{SchulteKolmogorov}} we know that
$$
\frac{M_{1,1}(f^{(i)})}{(\Var(F_{r,t}^{(i)}))^2} \leq \cum_4.
$$
In addition, we have the following lower bound for $M_{1,1}(f^{(i)})$:
$$    M_{1,1}(f^{(i)})
     =c t^{4(d-1-i)+1} \int_{A_h(d,d-1)} \mathcal{H}^{d-1}(\tilde{H}_1\cap B_r)^4 \ \mu_{d-1}(d\tilde{H}_1) \ge c\, g(d-1,4,d,r) \ge c\,
   e^{4r(d-2)},
$$
since $4(d-2)-(d-1)>0$, which follows from our assumption that $d\geq 4$, and since  $i\le d-1$ and $t\ge 1$. In combination with Lemma \ref{lem:VarianceBoundd>=4} we thus find that
\begin{align*}
\cum_4\geq \frac{M_{1,1}(f^{(i)})}{(\Var(F_{r,t}^{(i)}))^2} \geq {c\over c^{(i)}(d)}{e^{4r(d-2)}\over e^{4r(d-2)}}=c>0,
\end{align*}
which is a contradiction to \eqref{eq:13-06-1}. Consequently, the family of random variables $\big(\widetilde{F_{r,t}^{(i)}}\big)_{r\geq 1}$ cannot satisfy a central limit theorem as $r\to\infty$.
\end{proof}

\begin{remark}\label{rem:LimitDistribution}\rm
Let $d\geq 4$ and $i=d-1$ or $d\geq 7$ and $i\in\{0,1,\ldots,d-1\}$. For such $d$ and $i$ the proof of Theorem \ref{thm:CLTrToInfinity} (c) in combination with \cite[Corollary 4.7.19]{Bogachev}, a corollary of the Eberlein-\u{S}mulian theorem, shows that there exists a subsequence $\widetilde{F_{r_k,t}^{(i)}}$ such that $\widetilde{F_{r_k,t}^{(i)}}$ converges in distribution and in $L^4$ to some limiting random variable $X$, say. Especially this implies that $\E X=0$, $\E X^2=1$ and $\E X^m<\infty$ for $m\in\{3,4\}$. In particular, this rules out for $X$ the classical $\alpha$-stable distributions for any $0<\alpha<2$ and, since we have shown that $\cum_4(X)>0$, also a Gaussian distribution. We leave the determination of the distribution of the limiting random variable $X$ as a challenging open problem for future research.
\end{remark}

\subsection{The case of simultaneous growth of intensity and window: Proof of Theorem \ref{thm:RandTtoInfinity}}

According to Lemma \ref{lem:UnifInt} we have that, for any fixed $t\geq 1$,
$$
\sup_{r\geq 1}\E\left(\widetilde{F_{r,t}^{(i)}}\right)^5<\infty,\qquad\text{where}\qquad \widetilde{F_{r,t}^{(i)}}={F_{r,t}^{(i)}-\E F_{r,t}^{(i)}\over \sqrt{\Var(F_{r,t}^{(i)})}}
$$
and where $d$ and $i$ are as in the statement of Theorem \ref{thm:RandTtoInfinity}. Then, taking $t=1$, by H\"older's inequality it follows that
\begin{equation}\label{eq:4MomentSup}
\sup_{r\geq 1}\E\left(\widetilde{F_{r,1}^{(i)}}\right)^4 \leq \sup_{r\geq 1}\left(\E\left(\widetilde{F_{r,1}^{(i)}}\right)^5\right)^{4/5} < \infty.
\end{equation}

Next, we recall the definition of the integrals $M_{u,v}(h)$, $u,v\in\{1,\ldots,m\}$, from \eqref{eq:defMij} that are associated with a general Poisson U-statistic of order $m\in\N$ with kernel function $h$. In order to emphasize the role of the measure these integrals are taken with, we will write $M_{u,v}(h\,;\,\mu)$ in what follows. By definition of the integrated kernels in \eqref{eq:ChaosKernels} we have that
\begin{align}\label{eq:Mij_in_t}
   M_{u,v}(f^{(i)}\,;\,t \mu_{d-1})  \leq  t^{4(d-i-1)+1} M_{u,v}(f^{(i)}\,;\,\mu_{d-1})
\end{align}
for any $t\geq 1$ and any fixed $r\geq 1$. In fact, $f_u^{(i)}$ and $f_v^{(i)}$ contribute twice the factor $t^{d-i-u}$ and twice the factor $t^{d-i-v}$ by \eqref{eq:ChaosKernels}, respectively, and the integral in \eqref{eq:defMij} leads to an additional factor $t^{|\sigma|}$. By the choice $u=v=1$ we maximize the resulting exponent and see that their product is bounded by $t^{4(d-i-1)+1}$. Indeed, if $u=v=1$ we necessarily have that $|\sigma|=1$ since $\sigma$ has to be connected. On the other hand, if $u+v\geq 3$ then $|\sigma|\leq u+v$ and hence
\begin{align*}
2(d-i-u)+2(d-i-v)+|\sigma| &\leq 2(d-i-u)+2(d-i-v)+u+v\\
&=4(d-i-1)-(u+v)+4\\
&\leq 4(d-i-1)+1.
\end{align*}

Now, we apply the normal approximation bound \eqref{eq:Kolmogorov} to the Poisson U-statistic $F_{r,t}^{(i)}$. Together with \eqref{eq:Mij_in_t} and the lower and the upper variance bound from Lemma \ref{lem:VarianceBoundd>=4} this yields
\begin{align*}
d\left(\frac{F_{r,t}^{(i)}- \mathbb{E}F_{r,t}^{(i)}}{\sqrt{\Var F_{r,t}^{(i)}}},N \right) &\leq c \sum_{u,v=1}^{d-i} \frac{\sqrt{M_{u,v}(f^{(i)}{;} \ t \mu_{d-1})}}{\Var(F_{r,t}^{(i)})} \\
&\leq c \sum_{u,v=1}^{d-i} {t^{2(d-i-1)+1/2}\over t^{2(d-i)-1}}\frac{\sqrt{M_{u,v}(f^{(i)}\,;\, \mu_{d-1})}}{\Var(F_{r,1}^{(i)})}\\
&= {c\over\sqrt{t}} \sum_{u,v=1}^{d-i} \frac{\sqrt{M_{u,v}(f^{(i)}\,;\, \mu_{d-1})}}{\Var(F_{r,1}^{(i)})}
\end{align*}
for any $t\geq 1$ and $r\geq 1$. Note that the expression in the sum has now become a function of the parameter $r$ only. We can now apply for any $u,v\in\{1,\ldots,d-i\}$ the estimate
$$
\frac{\sqrt{M_{u,v}(f^{(i)}\,;\, \mu_{d-1})}}{\Var(F_{r,1}^{(i)})} \leq \sqrt{\E\left(\widetilde{F_{r,1}^{(i)}}\right)^4 - 3}
$$
from the discussion after {\cite[Corollary 4.3]{SchulteKolmogorov}} (see also \cite[Proposition 3.8]{LaPecc}). This leads to the bound
\begin{align*}
d\left(\frac{F_{r,t}^{(i)}- \mathbb{E}F_{r,t}^{(i)}}{\sqrt{\Var F_{r,t}^{(i)}}},N \right) &\leq \frac{c}{\sqrt{t}} \sqrt{\E\left(\widetilde{F_{r,1}^{(i)}}\right)^4 - 3} \leq \frac{c}{\sqrt{t}} \sqrt{\E\left(\widetilde{F_{r,1}^{(i)}}\right)^4}.
\end{align*}
However, in view of \eqref{eq:4MomentSup} the last expression is bounded by $c/\sqrt{t}$ for all $t\geq 1$ and $r\geq 1$. This completes the proof of Theorem \ref{thm:RandTtoInfinity}. $\hfill \Box$


\section{Proofs IV -- Multivariate limit theorems}\label{sec:Multi}

\subsection{The case of growing intensity: Proof of Theorem \ref{thm:CLTMultivariate} (a)}

This is a direct consequence of \cite[Theorem 5.2]{LPST}. \hfill $\Box$

\subsection{The case of growing windows: Proof of Theorem \ref{thm:CLTMultivariate} (b) and (c)}

\subsubsection{The planar case $d=2$: Proof of Theorem \ref{thm:CLTMultivariate} (b)}

Our goal is to use \eqref{eq:d_3}. The first term in \eqref{eq:d_3} is bounded by a constant multiple of $r^2e^{-r}$ by Lemma \ref{lem:dist_cov_d=2}. To evaluate the second term we have to combine the lower variance bound from Lemma \ref{lem:VarianceBoundd=2} with upper bounds for the terms $M_{1,1}$, $M_{1,2}$ and $M_{2,2}$. In the proof of Theorem \ref{thm:CLTrToInfinity} (a) we have already shown that
$M_{1,1}(f^{(i)},f^{(i)})\leq ce^r$ for $i\in\{0,1\}$ and $M_{2,2}(f^{(0)},f^{(0)})\leq cre^r$, which implies that
\begin{align*}
M_{1,1}(e^{-r/2}{f}^{(i)},e^{-r/2}f^{(i)}) &\leq c\,e^{-2r}\,e^{r} = c\,e^{-r},\\
M_{2,2}(e^{-r/2}{f}^{(0)},e^{-r/2}f^{(0)}) & \leq c\,r\,e^{-2r}\,e^{r} = c\,r\,e^{-r}.
\end{align*}
Finally, up to a constant factor an upper bound for $M_{1,2}(e^{-r/2}{f}^{(i)},e^{-r/2}f^{(0)})$, for $i\in\{0,1\}$, is
given by $M_{1,2}(e^{-r/2}{f}^{(0)},e^{-r/2}f^{(0)})$, which is equal to
$$
e^{-2r}M_{1,2}( {f}^{(0)} ) \leq c\,e^{-2r}\,(e^r+2 r^2\,e^r) \leq c\,r^2\,e^{-r}.
$$
Thus we conclude from \eqref{eq:d_3} that
\begin{align*}
d_3(\textbf{F}_{r,t},N_{\Sigma_2}) \leq c\,(r^2\,e^{-r}+e^{-r/2}+r^{1/2}\,e^{-r/2}+r\,e^{-r/2}) \leq c\, r \,e^{-r/2}.
\end{align*}
Since the covariance matrix $\Sigma_2$ is invertible, $\|\Sigma_2^{-1}\|_{\rm op}\|\Sigma_2\|_{\rm op}^{1/2}$ and $\|\Sigma_2^{-1}\|_{\rm op}\|^{3/2}\Sigma_2\|_{\rm op}$ are positive and finite constants only depending on $t$. Together with \eqref{eq:d_2} this also implies that
$$
d_2(\textbf{F}_{r,t},N_{\Sigma_2}) \leq c\, r \,e^{-r/2}.
$$
and completes the proof of Theorem \ref{thm:CLTMultivariate} (b). \hfill $\Box$

\subsubsection{The spatial case $d=3$: Proof of Theorem \ref{thm:CLTMultivariate} (c)}

Our goal is again to use the normal approximation bound \eqref{eq:d_3}. By Lemma \ref{lem:dist_cov_d=3} the first term in \eqref{eq:d_3} is bounded from above by a constant multiple of $r^{-1}$. Next, it remains to provide upper bounds for the terms
$$
M_{u,v} \qquad\text{for}\qquad (u,v)\in\{(1,1),(1,2),(1,3),(2,2),(2,3),(3,3)\}.
$$
As in the planar case $d=2$ all integrals which are involved have already been treated in the proof of the univariate limit theorem. Thus, using the bounds derived in the proof of Theorem \ref{thm:CLTrToInfinity} (b) we can complete the proof in dimension $d=3$. \hfill $\Box$


\subsection*{Acknowledgement}
This project was initiated when FH was visiting Ruhr University Bochum in March 2019, the financial support by the Deutsche Forschungsgemeinschaft (DFG) through RTG 2131 \textit{High-dimensional Phenomena in Probability -- Fluctuations and Discontinuity} is gratefully acknowledged. Our thanks also go to Tom Kaufmann (Bochum) whose programming skills helped us to verify the list of permutations we had to use to prove Theorem \ref{thm:CLTrToInfinity} (a) and (b).



\begin{thebibliography}{50} \small
\addcontentsline{toc}{section}{References}

\bibitem{ArbeiterZahle}
E. Arbeiter, M. Z\"ahle (1992): Geometric measures for random mosaics in spherical spaces. Stochastics Stochastics Rep. \textbf{46}, 63--77.

\bibitem{BaddeleyTurner}
A. Baddeley and P. Turner (2016): \textit{Spatial Point Patterns. Methodology and Applications with R}. Chapman \& Hall/CRC Press.

\bibitem{BaranyHugReitznerSchneider}
I. B\'ar\'any, D. Hug, M. Reitzner and R. Schneider (2017): Random points in halfspheres. Random Structures Algorithms {\bf 50}, 3--22.

\bibitem{BenjaminiEtAl}
I. Benjamini, J. Jonasson, O. Schramm and J. Tykesson (2009): Visibility to infinity in the hyperbolic plane, despite obstacles. ALEA Latin Amer. J. Prob. Math. Statist. \textbf{6}, 323--342.

\bibitem{BesauThaele}
F. Besau and C. Th\"ale (2019): Asymptotic normality for random polytopes in non-Euclidean geometries. arXiv:1909.05607.

\bibitem{BodeFoun...}
M. Bode, N. Fountoulakis and T. M\"uller (2016): The probability that the hyperbolic random graph is connected. Random Structures Algorithms \textbf{49}, 65--94.

\bibitem{Bogachev}
V.I. Bogachev (2007): \textit{Measure Theory, Vol. 1}. Springer.



\bibitem{Brothers}
J.E. Brothers (1966): Integral geometry in homogeneous spaces. Trans. Am. Math. Soc. \textbf{124}, 480--517.

\bibitem{CalkaChapronEnriquez}
P. Calka, A. Chapron and N. Enriquez (2018): Mean asymptotics for a Poisson-Voronoi cell on a Riemannian manifold. arXiv:1807.09043.

\bibitem{Chavel}
I. Chavel (1993): \textit{Riemannian Geometry -- A Modern Introduction}. Cambridge Tracts in Mathematics \textbf{108}, Cambridge University Press.

\bibitem{DeussHoerrmannThaele}
C. Deuss, J. H\"orrmann and C. Th\"ale (2017): A random cell splitting scheme on the sphere. Stochastic Process. Appl. \textbf{127}, 1544--1564.

\bibitem{EichelsbacherThaele14}
P. Eichelsbacher and C. Th\"ale (2014): New Berry-Esseen bounds for non-linear functionals of Poisson random measures. Electron. J. Probab. \textbf{19}, article 102.

\bibitem{Federer69}
H. Federer (1969): \textit{Geometric Measure Theory}. Springer.

\bibitem{FountulakisYukich}
N. Fountoulakis and J.E. Yukich (2019): Limit theory for the number of isolated and extreme points in hyperbolic random geometric graphs. arXiv: 1902.03998.

\bibitem{GallegoSolanes}
E. Gallego and G. Solanes (2005): Integral geometry and geometric inequalities in hyperbolic space.
Differential Geom. Appl. \textbf{22}, 315--325.

\bibitem{Gloaguen}
C. Gloaguen, F. Fleischer, H. Schmidt and V. Schmidt (2006): Fitting of stochastic telecommunication network models, via distance measures and Monte-Carlo tests. Telecom. Systems \textbf{31}, 353--377.

\bibitem{GH}
P. Goodey and R. Howard (1990): Processes of flats induced by higher dimensional processes. Adv. Math. \textbf{80}, 92--109.

\bibitem{HeinrichCLTforPHT}
L. Heinrich (2009): Central limit theorems for motion-invariant Poisson hyperplanes in expanding convex bodies. Rend. Circ. Mat. Palermo (2) Suppl. \textbf{81}, 187--212.

\bibitem{HeinrichMuchePVT}
L. Heinrich and L. Muche (2008): Second-order properties of the point process of nodes in a stationary Voronoi tessellation. Math. Nachr. \textbf{281}, 350--375.

\bibitem{HeinrichSchmidtSchmidtCLT}
L. Heinrich, H. Schmidt and V. Schmidt (2006): Central limit theorems for Poisson hyperplane tessellations. Ann. Appl. Probab. \textbf{16}, 919--950.


\bibitem{HugReichenbacher}
D. Hug and A. Reichenbacher (2019): Geometric inequalities, stability results and Kendall's problem in spherical space. arXiv 1709.06522.

\bibitem{HugSchneiderConicalTessellations}
D. Hug  and R. Schneider (2016): Conical random tessellations. Discrete Comput. Geom. \textbf{52}, 395--426.

\bibitem{HuThspstitt}
D. Hug and C. Th\"ale (2019): Splitting tessellations in spherical spaces. Electron. J. Probab. {\bf 24}, article 24.

\bibitem{HTW}
D. Hug, C. Th\"ale and W. Weil (2015): Intersection and proximity for processes of flats. J. Math. Anal. Appl. \textbf{426}, 1--42.

\bibitem{KabluchkoMarynychTemesvariThaele}
Z. Kabluchko, A. Marynych, D. Temesvari and C. Th\"ale (2019+): Cones generated by random points on half-spheres and convex hulls of Poisson point processes. To appear in Probab. Theory Related Fields.

\bibitem{Kallenberg76}
O. Kallenberg (1976): On the structure of stationary flat processes. Z. Wahrscheinlichkeitstheor. Verw. Geb., \textbf{37}, 157--174.

\bibitem{Kallenberg80}
O. Kallenberg (1980): On the structure of stationary flat processes. II. Z. Wahrscheinlichkeitstheor. Verw. Geb., \textbf{52}, 127--147.

\bibitem{LaPecc0}
R. Lachi\`eze-Rey and G. Peccati (2013): Fine Gaussian fluctuations on the Poisson space, I: contractions, cumulants and random geometric graphs. Electron. J. Probab. \textbf{18}, article 32.

\bibitem{LaPecc}
R. Lachi\`eze-Rey and G. Peccati (2013): Fine Gaussian fluctuations on the Poisson space, II: rescaled kernels, marked processes and geometric U-statistics. Stochastic Proce. Appl. \textbf{123}, 4186--4218.

\bibitem{LP}
G. Last and M.D. Penrose (2017): {\em Lecture on the Poisson Process}. Cambridge University Press.

\bibitem{LPST}
G. Last, M.D. Penrose, M. Schulte and C. Th\"ale (2014): Moments and central limit theorems for some multivariate Poisson functionals. Adv. Appl. Probab. \textbf{46}, 348--364.

\bibitem{Liao}
M. Liao (2018): {\em Invariant Markov Processes Under Lie Group Actions.} Springer, Cham.

\bibitem{MaeharaMartini18}
H. Maehara and H. Martini (2018): An analogue of Sylvester's four-point problem on the sphere. Acta Math. Hungar. {\bf 155}, 479--488.

\bibitem{MarinucciPeccati}
D. Marinucci and G. Peccati (2011): {\em Random Fields on the Sphere}. Cambridge University Press, Cambridge.

\bibitem{Malyarenko2013}
A. Malyarenko (2013): {\em Invariant random fields on spaces with a group action.} With a foreword by Nikolai Leonenko. Probability and its Applications (New York). Springer, Heidelberg.

\bibitem{Matheron}
G. Matheron (1975): \textit{Random Sets and Integral Geometry}. Wiley.

\bibitem{Mecke91}
J. Mecke (1991):  On the intersection density of flat processes. Math. Nachr. \textbf{151}, 69--74.

\bibitem{MilesSphere}
R.E. Miles (1971): Random points, sets and tessellations on the surface of a sphere, Sankhya: The Indian Journal of Statistics, Series A, \textbf{33}, 145--174.

\bibitem{Morgan}
F. Morgan (2016): \textit{Geometric measure theory. A beginner's guide}. Fourth edition. Elsevier/Academic Press.

\bibitem{MullerStaps}
T. M\"uller and M. Staps (2017): The diameter of KPKVB random graphs. arXiv 1707.09555.

\bibitem{OhserMuecklich}
J. Ohser and F. M\"ucklich (2000): \textit{Statistical Analysis of Microstructures in Materials Science}. Wiley.

\bibitem{Okabe}
A. Okabe, B. Boots, K. Sugihara and S.N. Chiu (2000): \textit{Spatial Tessellations: Concepts and Applications of Voronoi Diagrams}. 2nd edition, Wiley.

\bibitem{TakashiYogesh}
T. Owada and D. Yogeshwaran (2018): Sub-tree counts on hyperbolic random geometric graphs. arXiv 1802.06105.

\bibitem{ReitznerPeccati}
G. Peccati and M. Reitzner (eds.) (2016): \textit{Stochastic Analysis for Poisson Point Processes}. Bocconi \& Springer.

\bibitem{PTbook}
G. Peccati and M. Taqqu (2011): {\em Wiener Chaos: Moments, Cumulants and Diagrams}. Bocconi \& Springer.

\bibitem{PeccatiZheng}
G. Peccati and G. Zheng (2010): Multi-dimensional Gaussian fluctuations on the Poisson space. Electron. J. Probab. \textbf{15}, article 48.

\bibitem{PenroseYukichMf}
M.D. Penrose and J.E. Yukich (2013): Limit theory for point processes in manifolds. Ann. Appl. Probab. {\bf 23}, 2161--2211.

\bibitem{PorretBlanc}
S. Porret-Blanc (2007): Sur le caract\`ere born\'e de la cellule de Crofton des mosaques de g\'eod\'esiques dans le plan hyperbolique. C. R. Math. Acad. Sci. Paris \textbf{344}, 477--481.

\bibitem{Ratcliffe}
J.C. Ratcliffe (2007): \textit{Foundations of Hyperbolic Manifolds}. Springer.

\bibitem{Redenbach09}
C. Redenbach (2009):  Modelling foam structures using random tessellations. In: \textit{Proceeding of the 10th European Conference of ISS}, ESCULAPIO Pub. Co.

\bibitem{ReitznerSchulteCLT}
M. Reitzner and M. Schulte (2013): Central limit theorems for U-statistics of Poisson point processes. Ann. Probab. \textbf{41}, 3879--3909.

\bibitem{ReitznerSchulteThaele}
M. Reitzner, M. Schulte and C. Th\"ale (2017): Limit theory for the Gilbert graph. Adv. Appl. Math. \textbf{88}, 26--61.

\bibitem{Ripley}
B.D. Ripley (1977): Modelling spatial patterns (with discussion). J. Royal Statist. Soc., Series B, \textbf{39}, 172--212.


\bibitem{Santalo}
L.A. Santal\'o (2004): \textit{Integral Geometry and Geometric Probability}. 2nd edition, Cambridge University Press.

\bibitem{SantaloYanetz}
L.A. Santal\'o and I. Yanez (1972): Averages for polygons formed by random lines in Euclidean and hyperbolic planes. Adv. Appl. Probab. \textbf{9}, 140--157.

\bibitem{SW}
R. Schneider and W. Weil (2008): \textit{Stochastic and Integral Geometry}. Springer.

\bibitem{SchreiberThaele2ndOrder}
T. Schreiber and C. Th\"ale (2011): Second-order theory for iteration stable tessellations. Probab. Math. Statist. \textbf{32}, 281--300.

\bibitem{SchulteDissertation}
M. Schulte (2013): \textit{Malliavin-Stein Method in Stochastic Geometry}. PhD Thesis, University of Osnabr\"uck.

\bibitem{SchulteKolmogorov}
M. Schulte (2016): Normal approximation of Poisson functionals in Kolmogorov distance. J. Theor. Probab. \textbf{29}, 96--117.

\bibitem{Solanes}
G. Solanes (2003): Integral geometry and curvature integrals in hyperbolic
space. PhD Thesis, Universitat Aut{\`o}noma de Barcelona.

\bibitem{TykessonCalka}
J. Tykesson and P. Calka (2013): Asymptotics of visibility in the hyperbolic plane. Adv. Appl. Probab. \textbf{45}, 332--350.

\end{thebibliography}
\end{document}